\documentclass[UTF8]{amsart}
\usepackage{amssymb,amsmath,amsthm,mathrsfs,multirow,xcolor,framed,url,mathtools}
\usepackage{float}
\usepackage{tikz-cd} 
\usepackage{amsaddr}
\usepackage{lipsum} 
\usepackage{tabularx}
\usepackage{longtable} 
\usepackage{threeparttablex} 
\usepackage{booktabs} 
\usepackage[title]{appendix}
\usepackage{tikz}
\usepackage[all]{xy}

\usepackage[ 
bookmarks=true,
bookmarksnumbered=true
bookmarksopen=flase,
hypertexnames=false,
breaklinks=true,
colorlinks=true,
pagebackref=true,
hyperindex=true,
plainpages=false
pdfauthor={Zhou }
pdftitle={Ternary}
pdfsubject={ }
pdfkeywords={}
 ]
{hyperref}
\newtheorem{theorem}{Theorem}[section]
\newtheorem{lemma}[theorem]{Lemma}
\newtheorem{cor}[theorem]{Corollary}

\newtheorem{prop}[theorem]{Proposition}

\theoremstyle{definition}
\newtheorem{definition}{Definition}        

\newtheorem{example}{Example}              

\theoremstyle{remark}
\newtheorem{remark}{Remark}    

\usepackage{tikz}
\usetikzlibrary{arrows.meta}
\theoremstyle{remark}

\numberwithin{equation}{section}

\usepackage{hyperref}
\newcommand\nutwid{\overset{\text{\lower3pt\hbox{$\sim$}}}\nu}

\newcommand{\tr}{\operatorname{tr}}
\newcommand{\n}{\operatorname{n}}

\newcommand{\discrd}{\operatorname{discrd}}

\newcommand{\Aut}{\operatorname{Aut}}

\allowdisplaybreaks

\begin{document}

\title[On trace-zero norm forms of orders of rational quaternion algebras with applications]{On trace-zero norm forms of orders of rational quaternion algebras with applications}

	\author{Yifan Luo$^{1}$, Tianze Wang$^{1,2}$ and Haigang Zhou$^{3}$}
	
	\address{$^{1}$Institute of Mathematics, Henan Academy of Sciences, Zhengzhou, China, 450046}
	\email{kotori@hnas.ac.cn}
	
	\address{$^{2}$School of Mathematics and Statistics, North China University of Water Resources \\  and Electric Power, Zhengzhou, China, 450046}
	\email{wtz@ncwu.edu.cn}
	
	\address{$^{3}$School of Mathematical Sciences, Tongji University, Shanghai, China, 200092}
	\email{haigangz@tongji.edu.cn}

\begin{abstract}
In this paper, by normalizing the trace-zero norm forms according to the parity of orders of definite rational quaternion algebras, we give an explicit formulation to the even Clifford correspondence followed by the reciprocity of ternary forms. We obtain an exact identity between ternary representation numbers and the number of elements of an order with prescribed reduced norm and trace, which applies to all orders without squarefreeness constraint on the reduced discriminant. Applications include a criterion for realizing ternary forms as trace-zero norm forms, a formula for traces of Brandt matrices in terms of representation numbers, and explicit theta series relations associated with local reciprocity. We also describe reciprocal orders and order chains, and interpret the ternary forms occurring in the congruent number problem through their associated quaternion orders.
\end{abstract}

\maketitle

\section{Introduction}\label{Sec:1}

Let $Q$ be a definite quaternion algebra over $\mathbb Q$, with
reduced norm $\n$ and reduced trace $\operatorname{tr}$, and let
$\mathcal O\subset Q$ be a $\mathbb Z$-order. The restriction of the
norm to the trace-zero sublattice
\[
 \mathcal O^0\!:\,=\{\alpha\in\mathcal O:
                    \operatorname{tr}(\alpha)=0\}
\]
defines a positive definite integral ternary quadratic form
$$f_{\mathcal O^0}\!:\,=\n|_{\mathcal O^0},$$ 
considered up to integral
equivalence. This assignment does not determine the order in
general: the Lipschitz and Hurwitz orders have the same trace-zero
norm form $x^2+y^2+z^2$. Our purpose is to describe a normalization
of integral trace-zero norm forms according to the parity of orders explicitly, and gives a direct interpretation of
representation numbers in terms of elements of $\mathcal O$.

The correspondence between ternary quadratic forms and quaternion
orders through the even Clifford construction is classical.
Modern treatments include Gross and Lucianovic \cite{GL09}
and Voight \cite{Voi11}; see also \cite[Chapter~22]{Voi21}.
Over $\mathbb Z$, one realization of the ternary form associated
with $\mathcal O$ is
\[
 f_{\mathcal O}\!:\,
   =\operatorname{discrd}(\mathcal O)\,
      \n|_{(\mathcal O^\sharp)^0},
\]
where $\mathcal O^\sharp$ denotes the dual of $\mathcal O$, and
$\operatorname{discrd}(\mathcal O)$ denotes the reduced discriminant of $\mathcal O$.
This is an integral quadratic form in a basis of the lattice
$(\mathcal O^\sharp)^0$. Its representation numbers count vectors
in that lattice; these vectors need not belong to $\mathcal O$.
For counting elements of the original order with specified norm
and trace, we instead use lattices obtained directly from its
integral elements.

The arithmetic construction can be viewed on the additive quotient
$\mathcal O/\mathbb Z$. The expression
\[
 {\frak q}_{\mathcal O}([\alpha])
   =4\n(\alpha)-\operatorname{tr}(\alpha)^2
\]
is unchanged when an integer is added to $\alpha$ and defines a
positive definite integral ternary quadratic form on this quotient.
The map
\[
 [\alpha]\longmapsto2\alpha-\operatorname{tr}(\alpha)
\]
identifies it with the norm lattice
$(\mathbb Z+2\mathcal O)^0$.
For specified integers $n\geq0$ and $r$, each quotient class of
value $4n-r^2$ has a unique representative of trace $r$ and norm
$n$. This observation gives the element-counting interpretation
used throughout the paper.

We determine how this integral construction is related to the
classical ternary form. With the normalization and reciprocal
operation defined below, we prove
\[
 \widetilde f_{\mathcal O}=\psi_{f_{\mathcal O}},
 \quad
 \mathcal M=\psi\circ C_0^{-1}
\]
on the corresponding equivalence and isomorphism classes.
Thus the bijection on classes is identified with the classical
correspondence followed by reciprocity. The arithmetic content
is the realization by integral trace-zero norm lattices and the
determination of the scaling factor in the norm--trace counting
identity, including its dependence on the Gorenstein conductor.

Related norm--trace counting problems were studied for maximal
orders by Boylan, Skoruppa and Zhou \cite{BSZ19}, and for Eichler
orders of squarefree level by Li, Skoruppa and Zhou \cite{LSZ22}.
Luo and Zhou \cite{LZ24} classified ternary forms of level $4N$
for squarefree $N$ and obtained formulas for their weighted
genus representation numbers. The basic counting identity here
applies to arbitrary definite $\mathbb Z$-orders, including
non-Gorenstein orders and orders of non-squarefree reduced
discriminant. It supplies an elementwise description and a
normalization for the subsequent formulas involving genus
averages, theta series, and Brandt matrices.

In statements involving all definite orders, the ambient
rational quaternion algebra $Q$ is allowed to vary.
For a specified order or genus, its ambient algebra is fixed.
We now introduce the notations and invariants used in the normalization.
Write
\[
 f=(a,b,c,r,s,t),\quad
 f(x,y,z)=ax^2+by^2+cz^2+ryz+sxz+txy.
\]
Denote
\begin{align*}
d_f\!:&=4abc+rst-ar^2-bs^2-ct^2,\,\,\,\text{the {discriminant} of $f$},\\
m_{f}\!:&=\operatorname{gcd}(4bc-r^2,4ac-s^2,4ab-t^2,2st-4ar,2rt-4bs,2rs-4ct),\,\,\,\text{the {divisor} of $f$},\\
N_{f}\!:&=4d_f/m_f,\,\,\,\text{the {level} of $f$},\\
\widetilde{m}_{f}\!:&=\operatorname{gcd}(a,b,c,r,s,t),\,\,\,\text{the {content} of $f$},\\
\psi_{f}\!:&=\widetilde{m}_{f}\left(\frac{4bc-r^2}{m_{f}},\frac{4ac-s^2}{m_{f}},\frac{4ab-t^2}{m_{f}},\frac{2st-4ar}{m_{f}},
\frac{2rt-4bs}{m_{f}},\frac{2rs-4ct}{m_{f}}\right),\,\,\,\text{the reciprocal of $f$}.
\end{align*}
Notice that, if we write $f=\widetilde{m}_{f}f_1$ with $f_1=(a/\widetilde{m}_{f},b/\widetilde{m}_{f},c/\widetilde{m}_{f},r/\widetilde{m}_{f},s/\widetilde{m}_{f},t/\widetilde{m}_{f})$, then the last defining formula implies 
$$
\psi_{f}=\widetilde{m}_{f}\psi_{f_1}.
$$

Since $\operatorname{tr}(1)=2$, the trace ideal
$\operatorname{tr}(\mathcal O)$ is either $\mathbb Z$ or
$2\mathbb Z$. We distinguish these two cases as follows.

\vskip 0.2cm
\begin{definition}
We say that an order $\mathcal{O}\subset Q$ is \textbf{odd} if there exists an element $\alpha\in\mathcal{O}$ such that $2\nmid\tr(\alpha)$, otherwise we call it \textbf{even}.
\end{definition}

For any odd order $\mathcal{O}$, we use  
$$
S\!:\,=\mathbb{Z}+2\mathcal{O}
$$ 
to denote the Gross suborder of $\mathcal{O}$ (see~\cite{Gro87}). 
Then the following newly defined trace-zero norm form $\widetilde{f}_{\mathcal{O}}$ will serve as a solid foundation throughout the study of this paper.\\

\begin{definition}\label{def:tildef}
Let $Q$ be a definite quaternion algebra over $\mathbb{Q}$. For a Gorenstein order $\mathcal{O}$ of $Q$, we define 
\begin{equation*}
\widetilde{f}_{\mathcal{O}}\!:\,=
\begin{cases}
f_{\mathcal{O}^0}/\widetilde{m}_{f_{\mathcal{O}^0}}\!\!\!\!& \text{if}\,\,\, \mathcal{O}\,\, \text{is even},\\  
f_{S^0}/\widetilde{m}_{f_{S^0}}\!\!\!\!& \text{if}\,\,\, \mathcal{O}\,\, \text{is odd}.
\end{cases}
\end{equation*}
For an arbitrary order $\mathcal{O}\subset Q$ with conductor $m$ and Gorenstein saturation $\operatorname{Gor}(\mathcal{O})$, i.e., 
\begin{equation*}
\mathcal{O}=\mathbb{Z}+m\operatorname{Gor}(\mathcal{O}),
\end{equation*}
we define
\begin{equation*}
\widetilde{f}_{\mathcal{O}}\!:\,=m\widetilde{f}_{\operatorname{Gor}(\mathcal{O})}.
\end{equation*}
Then we define the map $\mathcal M$ to be the correspondence 
\begin{equation*}
{\mathcal M}\!:\,\mathcal{O}\mapsto \widetilde{f}_{\mathcal{O}}
\end{equation*}
from definite orders to positive definite integral ternary quadratic forms. 
\end{definition}

\vskip 0.2cm
With these definitions and notions in place, our main results can be stated as the following Theorem ~\ref{the:bijection} and Corollaries ~\ref{cor:tqf} to ~\ref{cor:trace}.

\begin{theorem}\label{the:bijection}
The map $\mathcal M$ induces a bijection between the isomorphism
classes of $\mathbb Z$-orders in definite quaternion algebras over
$\mathbb Q$ and the
$\mathrm{GL}_3(\mathbb Z)$-equivalence classes of positive definite
integral ternary quadratic forms. Moreover:
\begin{enumerate}
\item for all orders ${\mathcal O}$ of $Q$ we have $\widetilde{f}_{\mathcal{O}}=\psi_{f_{\mathcal{O}}}$; 
\item an order ${\mathcal O}$ of $Q$ is Gorenstein if and only if $\widetilde{f}_{\mathcal{O}}$ is primitive; hence the map ${\mathcal M}$ induces a bijection from the isomorphism classes of Gorenstein orders to the equivalence classes of primitive positive definite integral ternary quadratic forms;
\item for any order ${\mathcal O}$ of $Q$ of conductor $m$, there holds  
\begin{equation*}
\rho_{\mathcal{O}}(n,r)=R_{\widetilde{f}_{\mathcal{O}}}\left(\frac{m(4n-r^2)}{m_{f_{\mathcal{O}}}}\right),
\end{equation*}
where $\rho_{\mathcal{O}}(n,r)$ denotes the number of elements $\alpha\in\mathcal{O}$ with $\n(\alpha)=n$ and $\tr(\alpha)=r$ for integers $n\geq 0$ and $r\in\mathbb{Z}$, and $R_f(\gamma)$ denotes the number of integral solutions $(x,y,z)\in\mathbb{Z}^3$ of $f(x,y,z)=\gamma$ for given quadratic form $f$ and rational number $\gamma\in\mathbb{Q}$;   
\item for any order ${\mathcal O}$ of $Q$ of conductor $m$, let $\mathcal{O}_1,\ldots,\mathcal{O}_t$ be a complete set of representatives of isomorphism classes of genus of 
$\mathcal{O}$, then 
\begin{equation*}
\sum_{\mu=1}^t \frac{\rho_{\mathcal{O}_\mu}(n,r)}{|\Aut(\mathcal{O}_\mu)|}
=2R_{\operatorname{gen}(\widetilde{f}_\mathcal{O})}\left(\frac{m(4n-r^2)}{m_{f_{\mathcal{O}}}}\right),
\end{equation*}
where $|\Aut(\mathcal{O}_\mu)|$ denotes the cardinality of the group of automorphisms of $\mathcal{O}_\mu$, and 
\begin{equation*}
R_{\operatorname{gen}(\widetilde f_{\mathcal{O}})}\left(\gamma\right)=\sum\limits_{f\in \operatorname{gen}(\widetilde f_{\mathcal{O}})}\frac{R_{f}\left({\gamma}\right)}{|\Aut(f)|},
\end{equation*}
with the last summation running over a complete set of representatives of $\mathrm{GL}_3(\mathbb Z)$-equivalence classes of forms in the genus $\operatorname{gen}(\widetilde f_{\mathcal{O}})$ of $\widetilde f_{\mathcal{O}}$, the value of which can be explicitly computed by Siegel-Weil formula. Here $|\Aut(f)|$ denotes the cardinality of the subgroup of automorphisms of $f$ of $\mathrm{GL}_3(\mathbb Z)$.
\end{enumerate}
\end{theorem}

\begin{remark}
The divisor $m_{f_{\mathcal O}}$ has a direct interpretation in
terms of integral elements of the order:
\[
 m_{f_{\mathcal O}}
   =\gcd\{4\n(\alpha)-\operatorname{tr}(\alpha)^2:
                                      \alpha\in\mathcal O\}.
\]
Indeed, the norm--trace form on $\mathcal O/\mathbb Z$ satisfies
\[
 {\frak q}_{\mathcal O}
   \sim\frac{m_{f_{\mathcal O}}}{m}\,
               \widetilde f_{\mathcal O},
\]
where $m$ is the Gorenstein conductor of $\mathcal O$.
The content of $\widetilde f_{\mathcal O}$ is $m$, so the content
of ${\frak q}_{\mathcal O}$ is $m_{f_{\mathcal O}}$.
The content of an integral quadratic form equals the greatest
common divisor of its integral values, proving the assertion.
Consequently,
\[
 \widetilde f_{\mathcal O}
   \sim\frac{m}{m_{f_{\mathcal O}}}\,{\frak q}_{\mathcal O}.
\]
This describes the normalization using norm and trace values
together with the Gorenstein conductor. In particular, for a
Gorenstein order the normalized form is the primitive part of
${\frak q}_{\mathcal O}$.
\end{remark}

\begin{remark}
As an arithmetic interpretation of (3) in Theorem~\ref{the:bijection}, we see that the representation number of $m(4n-r^2)/m_{f_{\mathcal{O}}}$ by the ternary quadratic form $\widetilde f_{\mathcal{O}}$ is equal exactly to the number of elements with reduced characteristic polynomial $x^2-rx+n$ in $\mathcal{O}$. This observation sometimes allows us to translate arithmetic questions about orders into that of ternary quadratic forms, and vice versa. 
\end{remark}

\vskip 0.2cm
We next determine which positive definite integral ternary forms
occur as the unnormalized norm form on the trace-zero lattice
of an order. The following Corollary gives a necessary and sufficient
criterion in terms of the discriminant, level, and content.

\begin{cor}\label{cor:tqf}
For a positive definite integral ternary quadratic form $f$, there exists a quaternion order $\mathcal{O}$ such that 
$$
f_{\mathcal{O}^0}=f
$$ 
if and only if $16d_{f}/N_{f}^2$ is a square integer and, $(2\widetilde{m}^3_f) \nmid d_{f}$ in addition if $4d_{f}/N_{f}^2$ is not a square integer.   
\end{cor}

As an application of this corollary we will give an explanation in section ~\ref{section07}, see Proposition~\ref{exam:congruent}, why the quadratic forms relevant to the theory of congruent numbers due to Tunnell~\cite{Tun83} and Qin~\cite{Qin22} take their particular forms, from the quaternion order point of view. The most important difference and the relationships of the forms of Qin~\cite{Qin22} and Tunnell~\cite{Tun83} are also discussed by means of orders. 

\vskip 0.2cm
The norm--trace identity also gives a criterion for an order
to have unit group $\{\pm1\}$.

\begin{cor}\label{cor:unitgroup}
Let $\mathcal{O}$ be an arbitrary order. If $\mathcal{O}$ is not Gorenstein, or if $\mathcal{O}$ is Gorenstein with $m_{f_{\mathcal{O}}}\geq5$, then we have 
\begin{equation*}
\mathcal{O}^{\times}=\{\pm1\},
\end{equation*}
where $\mathcal{O}^{\times}$ denotes the unit group of $\mathcal{O}$. Moreover, if we let $h({\mathcal{O}})$ denote the class number of $\mathcal{O}$ and $\operatorname{mass}(\operatorname{Cls} (\mathcal{O}))$ the mass of the right class set $\operatorname{Cls} (\mathcal{O})$ of $\mathcal{O}$, then under the same conditions on $\mathcal{O}$ we have 
\begin{equation*}
h({\mathcal{O}})=\operatorname{mass}(\operatorname{Cls} (\mathcal{O})).
\end{equation*}
Here the value of the right-hand side can be computed explicitly by Eichler mass formula.
\end{cor}

Note that the unit group $\mathcal{O}^\times=\{\pm1,\pm i,\pm j,\pm k\}$ for Lipschitz order $\mathcal{O}$, and $m_{f_{\mathcal{O}}}=4$. Hence the lower bound $m_{f_{\mathcal{O}}}\geq5$ for $m_{f_{\mathcal{O}}}$ in this  corollary is the best possible if we want $\mathcal{O}^{\times}=\{\pm1\}$ in general.

\vskip 0.2cm
The same identity gives relations of the Eichler symbol of a Gorenstein
order to invariants of its associated ternary forms.
The following corollary gives explicit local formulas.
\begin{cor}\label{cor:Eichlersym}
Let $\mathcal{O}$ be a Gorenstein order. 
Then at places of odd prime $p\neq2$ we have
\begin{equation*}
\left(\frac{\mathcal{O}}{p}\right)=
\begin{cases}
0 &\text{if $p\mid m_{f_\mathcal{O}}$,}\\
\left(\frac{-m_{f_\mathcal{O}}}{p}\right)\left(\frac{{\widetilde{f}_\mathcal{O}}}{p}\right) &\text{if $p\nmid m_{f_\mathcal{O}}$, $p\mid m_{\widetilde{f}_\mathcal{O}}$},\\
* &\text{if $p\nmid m_{f_\mathcal{O}}m_{\widetilde{f}_\mathcal{O}}$,}\\
\end{cases}
\end{equation*}
and at place of prime $2$ we have
\begin{equation*}
\left(\frac{\mathcal{O}}{2}\right)=
\begin{cases}
0 &\text{if $\mathcal{O}$ is even,}\\
\left(\frac{-m_{f_\mathcal{O}}}{2}\right)\left(\frac{{\widetilde{f}_\mathcal{O}}}{8}\right) &\text{if $\mathcal{O}$ is odd and $2\mid d_{f_\mathcal{O}}$,}\\
* &\text{if $\mathcal{O}$ is odd and $2\nmid d_{f_\mathcal{O}}$,}\\
\end{cases}
\end{equation*}
where $\left(\frac{\mathcal{O}}{p}\right)$ is the Eichler symbol of order $\mathcal{O}$, $\left(\frac{f}{p}\right)$ is the genus symbol of ternary quadratic form $f$, and $*$ is a formal symbol to indicate the split maximal situation 
$\mathcal O_p\simeq M_2(\mathbb Z_p)$.
\end{cor}

Transporting reciprocity of ternary forms through the classical correspondence
defines a reciprocal operation on isomorphism classes of orders.
We can also give the associated correspondences
between type sets when reciprocity is prescribed at a single prime.

\begin{cor}\label{cor:automorphism}
There is a bijection $\psi$ up to isomorphism from the set of all orders to itself such that $f_{\mathcal{O}}=\widetilde{f}_{\psi(\mathcal{O})}$ and $f_{\psi(\mathcal{O})}=\widetilde{f}_{\mathcal{O}}$. Moreover, if $\mathcal{O}$ is an odd order, then $\psi(\mathcal{O})$ is an even order. Furthermore, if $\mathcal{O}$ and $\mathcal{O}'$ are two orders such that $\mathcal{O}_p\cong\psi(\mathcal{O}')_p$ for a prime $p\geq 2$ and $\mathcal{O}_q\cong\mathcal{O}'_q$ for all primes $q\neq p$, then there exists a bijection $\varphi_p$ between the type set of $\mathcal{O}$ and the type set of $\mathcal{O}'$.   
\end{cor}

As an application of this corollary, one can construct new types of orders and study the arithmetic and analytic properties of relevant subjects. The first one concerns with modular forms: we may establish explicit relationships as for the theta series associated with $\mathcal{O}$ and its reciprocal $\varphi_p(\mathcal{O})$ at prime $p$, see e.g. Proposition ~\ref{prop:theta}. 
Also, using Schiemann's theorem ~\cite{Sch97} that two ternary quadratic forms are determined up to equivalence by their theta series, we can prove in Theorem~\ref{the:theta} that two odd (or even) orders are determined completely up to equivalence by their theta series $\theta_{S^0}$ (or $\theta_{\mathcal{O}^0}$). 

Further, we may construct a counterexample (see e.g. Example ~\ref{exam:nodd} in section ~\ref{section6.00}) to a statement of Hashimoto~\cite[p.116]{Has06}, which says that there exists orders $\mathcal{O}$ and $\mathcal{O}'$, if $\mathcal{O}$ is not isomorphic but locally isomorphic to $\mathcal{O}'$, then $\theta_{\mathcal{O}^0}\neq\theta_{\mathcal{O}'^0}$ (see \eqref{gongshi5.3} for definition). 
A conjecture due to Hashimoto~\cite{Has06} on the linear relations over $\mathbb{C}$ of theta series associated to orders is also discussed. 

Moreover, recall that, although Pizer~\cite{Piz80} can successfully give linear combination representations for all newforms in $S_k(p^{2r+1}N)$ using the theta series attached to orders of level $p^{2r+1}N$, he found that the situation is quite different for the cases of level $p^{2}N$. To pursue the differences and the relationships of these two kinds of orders we prove in Proposition ~\ref{proposition7.2} that the orders of level $pN$ and $p^2N$ of Pizer~\cite{Piz80} are in fact reciprocal one another under $\varphi_p$. Following the same idea, Proposition ~\ref{proposition7.4} gives a complete construction of all the similar orders using $\varphi_p$. 

\vskip 0.2cm
Finally, we express traces of Brandt matrices in terms of
weighted genus representation numbers of ternary forms.
 Different from the usual ones, which often depend crucially on the theory of optimal embeddings (see, e.g., \cite{Hij74,Piz76a,Piz76b,HPS89}), our result relies only on representation numbers of ternary quadratic forms as well as the Picard group of orders, by an approach based mainly on Theorem~\ref{the:bijection}, instead of the classical optimal embeddings. 
The positive-index genus averages in the
formula can be evaluated through local representation densities
by the Siegel--Weil formula; the zero-index contribution is the
corresponding genus mass.


\begin{cor}\label{cor:trace}
Let $T_{\mathcal{O}}(n)$ denote the $n$-th Brandt matrix related to $\mathcal{O}$ with Gorenstein conductor $m$ for $n\geq1$. Then we have
\begin{equation*}
\tr(T_\mathcal{O}(n))=|\operatorname{Pic}(\mathcal{O})|\sum_{r^2\leq 4n} R_{\operatorname{gen}(\widetilde{f}_\mathcal{O})}\left(\frac{m(4n-r^2)}{m_{f_{\mathcal{O}}}}\right),
\end{equation*}
where $\operatorname{Pic}(\mathcal{O})$ is the Picard group of $\mathcal{O}$.
\end{cor}

Finally, we note that one can obtain more connections as for traces of Brandt matrices, see e.g. Proposition~\ref{prop:trace}, using the action of the map $\varphi_p$ defined by Corollary ~\ref{cor:automorphism} above.

\section{Basics on relations of orders and ternary quadratic forms}

In this section we first give a quick glance for the basic notations and fundamental results on quaternion algebras and quaternion orders over $\mathbb{Q}$. Then we collect some results on relations of quaternion orders and ternary quadratic forms for the use of our discussions in what follows. For more details, one is referred to~\cite{Voi21}.

\subsection{Basic notations and results on quaternion algebras}
Let $F$ stand for the fields $\mathbb{Q}$, $\mathbb{Q}_p$ of $p$-adic numbers, or $\mathbb{Q}_{\infty}=\mathbb{R}$ of real numbers, where $p$ denotes primes. For given nonzero elements $a,b\in F^{\times}$, we write $\left(\frac{a,\,b}{F}\right)$ for the $F$-algebra with a basis $1,i,j,k=ij=-ji$ and the defining relations $i^2=a, j^2=b$. We call $\left(\frac{a,\,b}{F}\right)$ a quaternion algebra, or simply an algebra. When $F=\mathbb{Q}_p$ or $\mathbb{Q}_{\infty}$, it is known that there exist exactly two isomorphism classes of $\left(\frac{a,\,b}{F}\right)$, namely the class of the algebra $M_2(F)$ of $2\times2$ matrices over $F$, and the class of a skew-field. Denoted by ${(a,b)}_p$ the standard $p$-adic Hilbert symbol. Then the quaternion algebra 
$$
Q_p\!:\,=\left(\frac{a,\,b}{\mathbb{Q}_p}\right)\simeq M_2(\mathbb{Q}_p)
$$ 
if and only if ${(a,b)}_p=1$. When $F=\mathbb{Q}$, two algebras 
$$
Q\!:\,=\left(\frac{a,\,b}{\mathbb{Q}}\right),\,\,\, \, Q'\!:\,=\left(\frac{a',\,b'}{\mathbb{Q}}\right)
$$ 
are isomorphic if and only if ${(a,b)}_p={(a',b')}_p$ for all primes $p$ and $\infty$. A quaternion algebra $Q$ over $\mathbb{Q}$ is said to be ramified at $p$ if $Q_p=\mathbb{Q}_p\otimes_{\mathbb Q} Q$ is a skew-field, and split at $p$ if $Q_p=\mathbb{Q}_p\otimes_{\mathbb Q} Q\simeq M_2(\mathbb{Q}_p)$. If $Q$ ramifies at $\infty$, it is called definite. A quaternion algebra $Q$ ramifies only at finitely many places and the number of these places is always even. Conversely, for a given set $S$ of even number of places, we can always find $a,b$ in $\mathbb{Q}^{\times}$ such that ${(a,b)}_p=-1$ exactly for places $p$ in $S$, and $Q=\left(\frac{a,\,b}{\mathbb{Q}}\right)$ is a quaternion algebra that ramifies exactly at primes in $S$.
For an element $\alpha = t + xi + yj + zk \in Q$, its conjugation $\overline{\alpha}$ is defined as
	\[
	\overline{\alpha}\!:\, = t - xi - yj - zk.
	\]
	This is an involution since it clearly satisfies $\overline{\overline{\alpha}} = \alpha$ and $\overline{\alpha\beta} = \overline{\beta} \, \overline{\alpha}$ for all $\alpha, \beta \in Q$. The reduced trace $\tr$ and the reduced norm $\n$ are defined by
	\[
	\tr(\alpha)\!:\, = \alpha + \overline{\alpha}, \quad \n(\alpha)\!:\, = \alpha \overline{\alpha}.
	\]
Every element $\alpha \in Q$ satisfies its reduced characteristic equation
	\[
	x^2 - \tr(\alpha) x + \n(\alpha) = 0.
	\]
	If $Q \simeq M_2(\mathbb{Q})$, then we may write instead of the above 
	\[
	A = \begin{pmatrix} a & b \\ c & d \end{pmatrix} \in Q,
	\]
	and we have
	\[
	\overline{A} = \begin{pmatrix} d & -b \\ -c & a \end{pmatrix}, \quad \tr(A) = a + d, \quad \n(A) = ad - bc = \det(A).
	\]

\subsection{Basic concepts and results on orders}
	An ideal of a quaternion algebra $Q$ is defined to be a full $\mathbb{Z}$-lattice in $Q$, that is, a finitely generated $\mathbb{Z}$-submodule of $Q$ that contains a $\mathbb{Q}$-basis of $Q$. 
	A (quaternion) order $\mathcal{O}$ in $Q$ is an ideal that is also a ring containing $\mathbb{Z}$. 
	From now on $\mathcal{O}$ is always used to denote orders in $Q$, and $\mathcal{O}_p$ is the localization of $\mathcal{O}$ over $\mathbb{Z}_p$ for all primes $p\geq 2$. 
	It is known that $\tr(\alpha),\n(\alpha)\in \mathbb{Z}$ for every $\alpha\in\mathcal{O}$. 
	Let $I$ be the ideal of $\mathbb{Z}$ generated by all $\det(\tr(\alpha_i\bar{\alpha_j}))$, where $\alpha_1,\ldots,\alpha_4\in \mathcal{O}$ is a $\mathbb{Z}$-basis of $\mathcal{O}$. 
	It is known that $I$ is the square of an ideal. 
	The reduced discriminant, $\discrd(\mathcal{O})$, of $\mathcal{O}$ is defined to be the square root of the positive generator of $I$. 	
		
Two orders $\mathcal{O}$ and $\mathcal{O}'$ are of the same type if $\mathcal{O}'=\alpha^{-1}\mathcal{O}\alpha$ for some $\alpha\in Q^{\times}$; 
or equivalently, if they are isomorphic as $\mathbb{Z}$-algebras. 
They are said to be locally of the same type (or locally isomorphic) if $\mathcal{O}_p$ and $\mathcal{O}_p'$ are of the same type for every finite prime $p$. 
The genus of $\mathcal{O}$ is defined to be the set of all orders in $Q$ locally isomorphic to $\mathcal{O}$, and the type set of $\mathcal{O}$ is defined to be the set of isomorphism classes of orders in this genus. It is known that the type set is finite, and its cardinality is called the type number of $\mathcal{O}$. 	

A fractional right $\mathcal O$-ideal is a full $\mathbb Z$-lattice
$I\subset Q$ satisfying $I\mathcal O\subseteq I$. We call $I$ locally
principal if, for every finite prime $p$, there exists
$\alpha_p\in Q_p^\times$ such that $I_p=\alpha_p\mathcal O_p$.
For such ideals, set
\[
 I\sim J \quad\Longleftrightarrow\quad
 J=\alpha I\ \text{for some }\alpha\in Q^\times.
\]
We denote the resulting class set by $\operatorname{Cls}(\mathcal O)$,
and write $h(\mathcal O)=|\operatorname{Cls}(\mathcal O)|$.
The left order of $I$ is
\[
 \mathcal O_L(I)\!:\,=\{\alpha\in Q:\alpha I\subseteq I\}.
\]
When $Q$ is definite, we use the conventions
\[
 w_I\!:\,=[\mathcal O_L(I)^\times:\{\pm1\}],
 \quad
 \operatorname{mass}(\operatorname{Cls}(\mathcal O))\!:\,
 =\sum_{[I]\in\operatorname{Cls}(\mathcal O)}\frac1{w_I}.
\]

Let $\operatorname{Idl}(\mathcal O)$ be the group of invertible two-sided
fractional $\mathcal O$-ideals. Here invertibility means that
$IJ=JI=\mathcal O$ for another two-sided fractional $\mathcal O$-ideal
$J$. We define
\[
 \operatorname{Pic}(\mathcal O)\!:\,
 =
 \operatorname{Idl}(\mathcal O)/
 \{a\mathcal O:a\in\mathbb Q^\times\}.
\]
Thus $\operatorname{Pic}(\mathcal O)$ is the Picard group of invertible
$\mathcal O$-bimodules over $\mathbb Z$.
We also write
\[
 \begin{split}
 N_{Q^\times}(\mathcal O)\!:\,
 &=\{\alpha\in Q^\times:
       \alpha\mathcal O\alpha^{-1}=\mathcal O\},\\
 \operatorname{PIdl}(\mathcal O)\!:\,
 &=\{\alpha\mathcal O:\alpha\in N_{Q^\times}(\mathcal O)\}.
 \end{split}
\]
The latter consists of all principal two-sided fractional
$\mathcal O$-ideals. Finally,
$\operatorname{Aut}(\mathcal O)$ denotes the group of
$\mathbb Z$-algebra automorphisms of $\mathcal O$; by Skolem--Noether's theorem,
\[
 \operatorname{Aut}(\mathcal O)
 \simeq N_{Q^\times}(\mathcal O)/\mathbb Q^\times.
\]

Suppose that $Q$ is definite, and choose representatives
$I_1,\ldots,I_h$ for $\operatorname{Cls}(\mathcal O)$.
For every integer $n\geq1$, define
\[
 T_{\mathcal O}(n)\!:\,=(b_{ij}(n))_{1\leq i,j\leq h},
 \quad
 b_{ij}(n)\!:\,
 =
 \#\{J\subseteq I_j:[J]=[I_i],\ [I_j:J]=n^2\},
\]
where $J$ ranges over locally principal fractional right
$\mathcal O$-ideals, and $[I_j:J]$ is the additive index.
We use this definition also when
$\gcd(n,\operatorname{discrd}(\mathcal O))>1$.
In particular, $T_{\mathcal O}(1)=I_{h\times h}$, the unit matrix of order $h$. Its diagonal entries satisfy
\[
 b_{ii}(n)
 =
 \frac{\#\{\alpha\in\mathcal O_L(I_i):\n(\alpha)=n\}}
      {|\mathcal O_L(I_i)^\times|}
 =
 \sum_{\substack{r\in\mathbb Z\\r^2\leq4n}}
 \frac{\rho_{\mathcal O_L(I_i)}(n,r)}
      {|\mathcal O_L(I_i)^\times|}.
\]
Indeed, a subideal $J$ in the class of $I_i$ has the form
$J=\alpha I_i$. When $J\subseteq I_i$, we have
$\alpha\in\mathcal O_L(I_i)$ and
$[I_i:\alpha I_i]=\n(\alpha)^2$.
Two such elements yield the same subideal precisely when they differ
by right multiplication by an element of $\mathcal O_L(I_i)^\times$.
This action is free for $n\geq1$.

\subsection{Basics on ternary quadratic forms} 
Let $f=(a,b,c,r,s,t)$ be a ternary quadratic form as in Section ~\ref{Sec:1}.
The Gram matrix associated with $f$ is defined as 
\begin{equation*}
M=M_f\!:\,=
\begin{pmatrix}
2a & t & s\\
t & 2b & r\\
s & r & 2c
\end{pmatrix}
.
\end{equation*}
We use
\begin{equation}
\begin{aligned}\label{gongshi2.01}
M_{11}=4bc-r^2, \,\,\,&M_{23}=M_{32}=st-2ar,\\
M_{22}=4ac-s^2, \,\,\,&M_{13}=M_{31}=rt-2bs,\\
M_{33}=4ab-t^2, \,\,\,&M_{12}=M_{21}=rs-2ct
\end{aligned}
\end{equation}
to denote the cofactors of $M$. 
We say $f$ is primitive if its content $\widetilde{m}_f=\gcd(a,b,c,r,s,t) = 1$.
Recall that for primitive $f$, the reciprocal of $f$ is defined as the ternary form 
\begin{equation*}
\psi=\psi_f=(M_{11}/m_f,M_{22}/m_f,M_{33}/m_f,2M_{23}/m_f,2M_{13}/m_f,2M_{12}/m_f).
\end{equation*} 
For general $f=\widetilde m_f f_1$, we generalize this by defining
\begin{equation}\label{gongshi2.02}
\psi=\psi_f\!:\,=\widetilde m_f \psi_{f_1},
\end{equation} 
so that 
$$
M_{\psi}=M_{\psi_f}=\frac{4\widetilde m_fd_f}{m_f}M_f^{-1}=\widetilde m_fN_fM_f^{-1}.
$$

Two integral ternary quadratic forms $f$ and $g$ are said to be equivalent, denoted as $f\sim g$, if there exists a matrix $M\in {\rm GL}_3(\mathbb{Z})$ (i.e., with integer entries and $\det(M)=\pm1$) such that $M_{f}=M^{t}M_{g}M$. Equivalent forms are said to belong to the same class. If $f \sim  g$, then it is clear that $d_f=d_g$ and $N_f=N_g$.
Two integral forms are said to be semi-equivalent if they are equivalent over the $p$-adic integers $\mathbb{Z}_p$ for all primes $p$, and over the real numbers $\mathbb{R}$. Semi-equivalent forms are said to be in the same genus of forms. Equivalent forms are clearly semi-equivalent. We write $\Aut(f)$ for the finite group of integral automorphisms of $f$, that is, $\Aut(f)$ is the set of all integral matrices $M\in {\rm GL}_3(\mathbb{Z})$ such that $M_f=M^{t}M_fM$. 
Two integral ternary quadratic forms are similar if one is equivalent to a unit
multiple of the other. 

\begin{definition}[\cite{Leh92}]\label{def:genussym}
	Let \(f=(a,b,c,r,s,t)\) be a primitive ternary quadratic form and \(\psi\) its reciprocal. Assume, after a suitable normalization, that \(a\) and \(M_{33}/m_f\) are coprime to \(m_f m_\psi\). If \(p\) is an odd prime with \(p \mid m_f\) (resp.\ \(p \mid m_\psi\)), we set
	\[
	\left(\frac{f}{p}\right) = \left(\frac{a}{p}\right) \quad
	\left(\text{resp. } \left(\frac{\psi}{p}\right) = \left(\frac{M_{33}/m_f}{p}\right)\right);
	\]
	for prime \(2\), if \(16 \mid m_f\) we set
	\[
	\left(\frac{f}{4}\right) = (-1)^{(a-1)/2},
	\]
	and if moreover \(32 \mid m_f\) we also set
	\[
	\left(\frac{f}{8}\right) = (-1)^{(a^2-1)/8};
	\]
	the symbols \(\left(\frac{\psi}{4}\right)\) and \(\left(\frac{\psi}{8}\right)\) are defined analogously. These symbols are collectively called the genus symbols of \(f\).
\end{definition}

Using genus symbols we have the following characterization of a genus.

\begin{lemma}[\cite{Leh92}, Proposition 4]\label{prop2.7}
	Two primitive positive definite ternary quadratic forms $f_1$ and $f_2$ lie in the same genus if and only if they have the same level, same discriminant, and same collection of genus symbols.
\end{lemma}

\subsection{Lehman's correspondences}\label{section3.2} 
Let $C(N,d)$ denote the set of all classes of primitive positive definite ternary quadratic forms of level $N$ and discriminant $d$. For odd primes $p\neq2$, the two sets $C(p^gN,\,p^hd)$ and $C(p^gN,\,p^{3g-h}d)$ with $p\nmid Nd$ are in bijection via Lehman's correspondence $\phi_p$ (see~\cite[Theorem~4]{Leh92}),
\begin{align*}
\phi_p :C(p^gN,\,p^{h}d) & \to C(p^gN,\,p^{3g-h}d)\\
(p^ga,\,p^{h-g}b,\,c,\,p^{h-g}r,\,p^gs,\,p^gt) & \mapsto (a,\,p^{2g-h}b,\,p^gc,\,p^gr,\,p^gs,\,p^{2g-h}t),
\end{align*}
where $a, b, c, r, s$ and $t$ are integers satisfying $p\nmid ac$. 

For $p=2$, the two sets $C(2^{g}N,\,2^{h}d)$ and $C(2^{g}N,\,2^{3g-h-2}d)$ with $2\nmid Nd$ are in bijection via Lehman's one-to-one correspondence $\phi_2$ as follows (see~\cite[Theorem~5]{Leh92}):
\begin{enumerate}
    \item If $h=g-2$:
\begin{align*}
\phi_2 : C(2^gN,\,2^{g-2}d) & \to C(2^gN,\,2^{2g}d)\\
(2^{g-2}a,\,b,\,c,\,r,\,2^{g-1}s,\,2^{g-1}t)& \mapsto (a,\,2^{g}b,\,2^{g}c,\,2^{g}r,\,2^{g}s,\,2^{g}t).
\end{align*}
    \item If $g\leq h\leq 2g-2$:
\begin{align*}
\phi_2 : C(2^gN,\,2^{h}d) & \to C(2^gN,\,2^{3g-h-2}d)\\
(2^{g-2}a,\,2^{h-g}b,\,c,\,2^{h-g+1}r,\,2^{g-1}s,\,2^{g-1}t)& \mapsto (a,\,2^{2g-h-2}b,\,2^{g-2}c,\,2^{g-1}r,\,2^{g-1}s,\,2^{2g-h-1}t).
\end{align*}
    \item If $h=2g$:
\begin{align*}
\phi_2 : C(2^gN,\,2^{2g}d) & \to C(2^gN,\,2^{g-2}d)\\
(2^{g}a,\,2^{g}b,\,c,\,2^{g}r,\,2^{g}s,\,2^{g}t)& \mapsto (a,\,b,\,2^{g-2}c,\,2^{g-1}r,\,2^{g-1}s,\,t).
\end{align*}
\end{enumerate}
\begin{remark}\label{rem:phi}
For a form $f$ of level $N=\prod p^{v_p(N)} $, we have $\phi_{p_1} \circ \cdots \circ \phi_{p_{e(N)}}(f)=\psi_f$, where $e(N)$ denotes the number of distinct prime factors of $N$. 
It is clear that if $p\neq2$, then $\phi_p(f(x,y,z))=p^gf(p^{-g}x,p^{g-h}y,z)$. Hence for any prime $q\neq p$, we have $\phi_p(f)\sim_q p^g f$, and whence $\phi_p(f)\sim_p (p^{g}/N)\psi_f$. The same holds for $p=2$. 
\end{remark}

\subsection{Watson's transformation}\label{section3.3} 
Let $\Lambda_f$ be a 3-dimensional lattice with the corresponding integral ternary quadratic form $f$. 
For any positive integer $m\geq 1$, define $\Lambda_m(f)$ to be the set of all $\boldsymbol{x}$ in $\Lambda_f$ such that 
\begin{equation*}
f(\boldsymbol{x}+\boldsymbol{z})\equiv f(\boldsymbol{z})\!\!\!\pmod m,\,\, \forall \,\boldsymbol{z}\in\Lambda_f.
\end{equation*}
It is known that $\Lambda_m(f)$ is also a 3-dimensional lattice, and we can choose a matrix $P$ such that $\boldsymbol{x}$ is in $\Lambda_m(f)$ if and only if $\boldsymbol{x}=P\boldsymbol{y}, \,\, y \in \Lambda_f$. 
Define the Watson transformation $\lambda_m$ on $f$ by 
\begin{equation*}
\lambda_m(f)\!:\,=m^{-1}f(P\boldsymbol{y}),\,\,\boldsymbol{y}\in\Lambda_f.
\end{equation*}
Then we have the following fundamental result of Watson.
\begin{lemma}[\cite{Wat62}]\label{Watsontrans}
For any positive integer $m\geq 1$, Watson transformation $\lambda_m$ is a well-defined mapping. More precisely, if $f\sim g$, then $\lambda_m(f)\sim \lambda_m(g)$, and if $f$ and $g$ are semi-equivalent, then $\lambda_m(f)$ and $\lambda_m(g)$ are semi-equivalent. Furthermore, the class number does not increase under Watson transformations, that is, $\lambda_m(f)$ can range over the whole of the genus of $\lambda_m(f)$.
\end{lemma}
For the particular case of $m=4$, we have  
\begin{lemma}[\cite{LZ24}, Proposition 2.7]\label{Watsontrans111}
If $2\nmid Nd$, then the Watson transformation $\lambda_4$ is a bijection between:
\begin{enumerate}
\item $C(8N,\,64d)$ and $C(4N,\,4d)$; or
\item $C(4N,\,16d)$ and $C(4N,\,d)$; and in this particular case, $\lambda_4$ is exactly the same as Lehman's one-to-one correspondence $\phi_2$.  
\end{enumerate}
\end{lemma}

\subsection{Even Clifford algebras}
For a non-degenerate ternary quadratic form $f$ that is integral over $\mathbb{Z}$, let $C_0(f)$ denote the even Clifford algebra over $\mathbb{Z}$ associated to $f$. Then it is known that $C_0(f)$ is an order in a quaternion algebra over $\mathbb{Q}$. For more details on Clifford algebra, one can see \cite[Chapter 22]{Voi21}. 

Conversely, if $\mathcal{O}=\mathbb{Z}+\mathbb{Z}e_1+\mathbb{Z}e_2+\mathbb{Z}e_3$ is an order in a quaternion algebra $Q$, then $\Lambda =\mathcal{O}^{\sharp}\cap Q^0$ is a 3-dimensional $\mathbb{Z}$-lattice in $Q^0$, where and in the sequel, $\mathcal{O}^{\sharp}$ denotes the dual of $\mathcal{O}$ under the bilinear form $\tr(\alpha\beta)$, and 
$$
Q^0= \{ \alpha \in Q : \tr(\alpha) = 0 \}.
$$
More accurately, if $\mathcal{O}^{\sharp}= \langle e_0', e_1', e_2', e_3'\rangle$ where $\{e_0', e_1', e_2', e_3'\}$ is the dual basis of $\{1, e_1, e_2, e_3\}$, then $\Lambda=\langle e_1', e_2', e_3'\rangle$. Thus the ternary quadratic form $f_{\mathcal{O}}$ associated to $\mathcal{O}$ is, by definition, 
\begin{equation}\label{gongshi3.1}
f_{\mathcal{O}}=\discrd(\mathcal{O}) \n(xe_1'+ye_2'+ze_3').
\end{equation}

We note that, the above constructions can be generalized to the case with a general principal ideal domain $R$ instead of $\mathbb{Z}$. As for the more general construction we have the following important result.

\begin{lemma}[\cite{Voi21}, Main Theorem 22.1.1]\label{Lemma3.3}
Let $R$ be a principal ideal domain with fraction field $K$ of
characteristic different from $2$. The even Clifford construction
induces a bijection between similarity classes of non-degenerate
integral ternary quadratic forms over $R$ and isomorphism classes of
quaternion $R$-orders. Its inverse is induced by
\[
  \mathcal O\longmapsto
  f_{\mathcal O}
  =\delta\,\operatorname{nrd}\big|_{\mathcal O^\sharp\cap Q^0},
\]
where  $\mathcal O^\sharp$ is the dual with respect to
$\operatorname{tr}(x y)$, and $\delta$ is any generator of the
reduced discriminant ideal of $\mathcal O$. Moreover,
\[
  \discrd(C_0(f))=(d_f).
\]
In particular, $f_{\mathcal O}$ is well-defined up to similarity.
\end{lemma}

For $R=\mathbb{Z}$, choosing the positive generator of the reduced discriminant
ideal gives the normalization used in \eqref{gongshi3.1}. When restricted to
positive definite forms, similarity over $\mathbb{Z}$ is the same as integral
equivalence, because the only unit preserving positive definiteness is
$1$.

As for practical computations we have the following fundamental explicit correspondences. 
Assume that $f=(a,b,c,r,s,t),\,\, d_f=d,\,\, \mathcal{O}=C_0(f)=\langle 1, e_1, e_2, e_3\rangle$. 
Then we have
\begin{equation}\label{gongshi3.2}
\begin{aligned}
& \!\!\!\! e_1^2=re_1-bc,\,\,\,e_2e_3=a\overline{e_1},\\
& \!\!\!\! e_2^2=se_2-ac,\,\,\,e_3e_1=b\overline{e_2},\\
& \!\!\!\! e_3^2=te_3-ab,\,\,\,e_1e_2=c\overline{e_3}.
\end{aligned}
\end{equation}
For the dual basis $\{e_0', e_1', e_2', e_3'\}$ of $\mathcal{O}^{\sharp}$ with respect to $\{1, e_1, e_2, e_3\}$ we have
\begin{equation}\label{gongshi3.3}
\tr(e_0')=1,\,\, \tr(e_1')=\tr(e_2')=\tr(e_3')=0, 
\end{equation}
\begin{equation}\label{gongshi3.4}
\begin{aligned}
& \!\!\!\! de_0'=d-2(abc+rst)+(ar+st)e_1+(bs+rt)e_2+(ct+rs)e_3,\\
& \!\!\!\! de_1'=ar+st-2ae_1-te_2-se_3,\\
& \!\!\!\! de_2'=bs+rt-te_1-2be_2-re_3,\\
& \!\!\!\! de_3'=ct+rs-se_1-re_2-2ce_3,
\end{aligned}
\end{equation}
and 
\begin{equation}\label{gongshi3.5}
\begin{aligned}
& \!\!\!\! \n(de_1')=da,\,\,\tr(de_2'\overline{de_3'})=dr,\\
& \!\!\!\! \n(de_2')=db,\,\,\tr(de_3'\overline{de_1'})=ds,\\
& \!\!\!\! \n(de_3')=dc,\,\,\tr(de_1'\overline{de_2'})=dt.
\end{aligned}
\end{equation}

\subsection{Gorenstein orders}
Recall that the dual of $\mathcal O$ is
\[
 \mathcal O^\sharp
 =
 \{\alpha\in Q:
       \tr(\alpha\beta)\in\mathbb Z
       \text{ for every }\beta\in\mathcal O\}.
\]
We call $\mathcal O$ a \emph{Gorenstein order} if
$\mathcal O^\sharp$ is an invertible two-sided fractional
$\mathcal O$-ideal.
Equivalently, for every prime $p$, there is an element
$\alpha_p\in Q_p^\times$ such that
\[
 \mathcal O^\sharp\otimes_{\mathbb Z}\mathbb Z_p
 =\alpha_p\mathcal O_p.
\]
By \cite[Theorem~24.2.10]{Voi21}, this condition is equivalent
to the primitivity of the associated ternary quadratic form
$f_{\mathcal O}$.
The following result describes the Gorenstein saturation of
an arbitrary order.
\begin{lemma}[\cite{Voi21}, Proposition 24.2.15]\label{the:Gor}
For any order $\mathcal{O}\subset Q$, there exists a unique integer $n=f(\mathcal{O})\geq 1$ and a unique Gorenstein order $\operatorname{Gor}(\mathcal{O})$ such that
\begin{equation*}
\mathcal{O}=\mathbb{Z}+n\operatorname{Gor}(\mathcal{O}).
\end{equation*}
Two orders $\mathcal{O}$, $\mathcal{O}'$ are isomorphic if and only if $f(\mathcal{O})=f(\mathcal{O}')$ and $\operatorname{Gor}(\mathcal{O})$ is isomorphic to $\operatorname{Gor}(\mathcal{O}')$.
\end{lemma}
Referring to this, for fixed $n\geq 1$, we can define a map $\sigma_n$ on Gorenstein orders $\mathcal{O}$, 
\begin{equation}\label{gongshi3.6}
\sigma_n: \mathcal{O} \mapsto \mathbb{Z}+n\mathcal{O},
\end{equation} 
which is bijective up to isomorphisms, and genus preserving, among others. 

\subsection{Trace-zero norm forms}
The following lemma is a basis for the separate considerations of trace-zero norm forms. 
\begin{lemma}[\cite{LZ24}, Proposition 4.1]\label{prop:basis}
Let $\mathcal{O}$ be an order. If $\mathcal{O}$ is even, then there exist $\alpha_1,\alpha_2,\alpha_3\in\mathcal{O}$ with $\tr(\alpha_1)=\tr(\alpha_2)=\tr(\alpha_3)=0$, such that
\begin{equation*}
\mathcal{O}=\mathbb{Z}+\mathbb{Z}\alpha_1+\mathbb{Z}\alpha_2+\mathbb{Z}\alpha_3.
\end{equation*}
If $\mathcal{O}$ is odd, then there exist $\beta_1,\beta_2,\beta_3\in\mathcal{O}$ with $\tr(\beta_1)=\tr(\beta_2)=0$ and $\tr(\beta_3)=1$, such that
\begin{equation*}
\mathcal{O}=\mathbb{Z}+\mathbb{Z}\beta_1+\mathbb{Z}\beta_2+\mathbb{Z}\beta_3.
\end{equation*}
\end{lemma}

From this we can define the trace-zero norm forms $f_{\mathcal{O}^{0}}$ according as $\mathcal{O}$ is even or odd as follows.
For positiveness we first assume that $Q$ is positive definite, i.e., it ramifies at $\infty$. 
If $\mathcal{O}$ is even, then $\mathcal{O}^{0}=\mathbb{Z}\alpha_1+\mathbb{Z}\alpha_2+\mathbb{Z}\alpha_3$; hence we can define 
\begin{align}\label{gongshi3.7}
f_{\mathcal{O}^{0}}\!: & =\n(x\alpha_1+y\alpha_2+z\alpha_3)\\
& =\n(\alpha_1)x^2+\n(\alpha_2)y^2+\n(\alpha_3)z^2+\tr(\alpha_2\overline{\alpha_3})yz
+\tr(\alpha_1\overline{\alpha_3})xz+\tr(\alpha_1\overline{\alpha_2})xy. \nonumber
\end{align}
If $\mathcal{O}$ is odd, then $\mathcal{O}^{0}=\mathbb{Z}\beta_1+\mathbb{Z}\beta_2+\mathbb{Z}(2\beta_3-1)$; hence we can define 
\begin{align}\label{gongshi3.8}
f_{\mathcal{O}^{0}}\!: & =\n(x\beta_1+y\beta_2+z(2\beta_3-1))\\
& =\n(\beta_1)x^2+\n(\beta_2)y^2+(4\n(\beta_3)-1)z^2+2\tr(\beta_2\overline{\beta_3})yz
+2\tr(\beta_1\overline{\beta_3})xz+\tr(\beta_1\overline{\beta_2})xy. \nonumber
\end{align}

\section{Characterization and classification for odd orders}

In this section we come to discuss the characterization and classification for odd orders by virtue of mainly the localization at prime $2$. The material of this section will serve as a solid foundation for the discussion of the remaining of this paper. We note that similar theory might have been established for odd primes in precisely the same way. 
\subsection{Odd orders in general}
We will give several characterizations for odd orders $\mathcal{O}$ in this subsection. 
The first one depends on the arithmetic invariants of $f_{\mathcal{O}}$. 
\begin{theorem}\label{the:odd}
Let $\mathcal{O}\subset Q$ be a Gorenstein order, and $f_{\mathcal{O}}$ be defined as in \eqref{gongshi3.1}. Let $d_{f_{\mathcal{O}}}$, $N_{f_{\mathcal{O}}}$ and $m_{f_{\mathcal{O}}}$ denote respectively the discriminant, the level and the divisor of $f_{\mathcal{O}}$. Then the following statements are equivalent:
\begin{enumerate}
\item $\mathcal{O}$ is an odd order;
\item $v_2(d_{f_{\mathcal{O}}})=v_2(N_{f_{\mathcal{O}}})-2$;
\item $2\nmid m_{f_{\mathcal{O}}}$.
\end{enumerate}
Here and in the sequel $v_p(n)$ always denotes the highest exponent of prime $p\geq 2$ in the prime factorization of positive integer $n\geq 1$.
\end{theorem}

\begin{proof}
We first prove the equivalence of statements (1) and (2). On the one hand, we assume $h=g-2$, where $h=v_2(d_{f_{\mathcal{O}}})$, $g=v_2(N_{f_{\mathcal{O}}})\geq 2$. Since $\mathcal{O}$ is Gorenstein, $f_{\mathcal{O}}$ is primitive. Hence in view of the list of the classification of primitive forms in Section~\ref{section3.2}, due to ~\cite[Theorem 5]{Leh92}, we may suppose without loss of generality that 
$$
f_{\mathcal{O}} = (2^{g-2}a,\,b,\,c,\,r,\,2^{g-1}s,\,2^{g-1}t),  
$$ 
where $2\nmid ac$. Whence by definition we have
$$
d_{f_{\mathcal{O}}}=\frac12\det 
\begin{pmatrix}
2^{g-1}a & 2^{g-1}t & 2^{g-1}s \\
2^{g-1}t & 2b & r \\
2^{g-1}s & r & 2c 
\end{pmatrix}
=2^{g-2}\det 
\begin{pmatrix}
a & t & s \\
2^{g-1}t & 2b & r \\
2^{g-1}s & r & 2c 
\end{pmatrix}.
$$
This together with the assumption $h=g-2$ clearly implies $2\nmid r$. 
Again by Lemma \ref{Lemma3.3} and \eqref{gongshi3.2} we may write  
\begin{equation*}
\mathcal{O}\simeq\mathbb{Z}+\mathbb{Z}e_1+\mathbb{Z}e_2+\mathbb{Z}e_3=\mathcal{O}', \,\,\,\text{say}, 
\end{equation*}
with
\begin{align*}
\begin{cases}
e_1^2=re_1-bc, & e_2e_3=2^{g-2}a\overline{e_1},\\
e_2^2=2^{g-1}se_2-2^{g-2}ac, & e_3e_1=b\overline{e_2},\\
e_3^2=2^{g-1}te_3-2^{g-2}ab, & e_1e_2=c\overline{e_3}.
\end{cases}
\end{align*}
Then from the first equality $e_1^2=re_1-bc$ above we see that $\tr(e_1)=r$. Thus $\mathcal{O}'$ is odd by definition. Hence the Gorenstein order $\mathcal{O}$ itself is odd since $\mathcal{O}\simeq\mathcal{O}'$, by the use of Skolem-Noether's theorem.  

On the other hand, if $\mathcal{O}$ is odd, we need to prove $h=g-2$.
Suppose the contrary, then by ~\cite[Theorem~2]{Leh92} there would exist two possibilities: (i) $g\leq h\leq 2g-2$, and (ii) $h=2g$. 
If $g\leq h\leq 2g-2$, then similar to the above we may assume, using the list of the classification of primitive forms in section~\ref{section3.2} again, 
$$
f_{\mathcal{O}}=(2^{g-2}a,\,2^{h-g}b,\,c,\,2^{h-g+1}r,\,2^{g-1}s,\,2^{g-1}t).
$$ 
Hence we can write
\begin{equation*}
\mathcal{O}\simeq\mathbb{Z}+\mathbb{Z}e_1+\mathbb{Z}e_2+\mathbb{Z}e_3=\mathcal{O}'', \,\,\,\text{say}, 
\end{equation*}
with
\begin{align*}
\begin{cases}
e_1^2=2^{h-g+1}re_1-2^{h-g}bc, & e_2e_3=2^{g-2}a\overline{e_1},\\
e_2^2=2^{g-1}se_2-2^{g-2}ac, & e_3e_1=2^{h-g}b\overline{e_2},\\
e_3^2=2^{g-1}te_3-2^{h-2}ab, & e_1e_2=c\overline{e_3}.
\end{cases}
\end{align*}
Note that by the left three identities above we have $\tr(e_1)=2^{h-g+1}r$, $\tr(e_2)=2^{g-1}s$ and $\tr(e_3)=2^{g-1}t$. 
Hence in view of $g\geq2$ and $g\leq h$ we obtain $2\mid\tr(e_1)$, $2\mid\tr(e_2)$, $2\mid\tr(e_3)$. This implies that $\mathcal{O}''$, and hence $\mathcal{O}$ is even by definition, which is a contradiction. 
For the case of $h=2g$, the argument is similar. The only difference is the use of the form
$f_{\mathcal{O}}=(2^{g}a,\,2^{g}b,\,c,\,2^{g}r,\,2^{g}s,\,2^{g}t)$ instead of the last one.

Turn to the proof of the equivalence of statements (2) and (3). It is a direct consequence of the identity $m_{f_{\mathcal{O}}}=4d_{f_{\mathcal{O}}}/N_{f_{\mathcal{O}}}$.
\end{proof}

The following proposition gives an explicit characterization to the localization at place $2$ of primitive positive definite quadratic forms $f$ with $v_2(d_f)=v_2(N_f)-2$.
\begin{prop}\label{lem:oddtqf}
Let $f$ be a primitive positive definite integral ternary
quadratic form such that
\[
v_2(d_f)=v_2(N_f)-2.
\]
Put $d_f^*=d_f/2^{v_2(d_f)}$. Then either
\[
f\sim_2 {d_f^*}
\bigl(3\cdot 2^{v_2(d_f)}x^2-(y^2+z^2+yz)\bigr)
\]
or
\[
f\sim_2 {d_f^*}
\bigl(-2^{v_2(d_f)}x^2-yz\bigr).
\]
Moreover, the two displayed normal forms given above are inequivalent {\rm(}resp. equivalent{\rm)} over $\mathbb Z_2$ 
if $v_2(d_f)\geq 1$ {\rm(}resp. $v_2(d_f)=0${\rm)}.
\end{prop}

\begin{proof}
The identity $N_f=4d_f/m_f$ shows that the hypothesis $v_2(d_f)=v_2(N_f)-2$ is equivalent to
$m_f$ being odd. Write
\[
 f=ax^2+by^2+cz^2+ryz+sxz+txy.
\]
The definition of $m_f$ then implies that at least one of $r,s,t$ is
odd. After permuting the variables, we may assume $r$ is odd. The Gram
matrix of the restriction to the $y,z$-plane has determinant
$4bc-r^2\in\mathbb{Z}_2^\times$. Consequently, this binary sublattice splits
off orthogonally over $\mathbb{Z}_2$.

Set $g={d_f^*}^{-1}f$. The preceding splitting applies equally to
$g$. The classification of even unimodular binary lattices over
$\mathbb{Z}_2$ says that their quadratic forms are equivalent to either
$-yz$ or $-(y^2+z^2+yz)$. Hence, for some
$\alpha\in\mathbb{Z}_2\setminus\{0\}$,
\[
 g\sim_2\alpha x^2-yz
 \quad\text{or}\quad
 g\sim_2\alpha x^2-(y^2+z^2+yz).
\]
Since
\[
 d_g={d_f^*}^{-3}d_f=2^{v_2(d_f)}{d_f^*}^{-2},
\]
its unit part is a square in $\mathbb{Z}_2^\times$. 
Recall that discriminants of equivalent forms differ by the square of a
$\mathbb{Z}_2$-unit. Clearly, the discriminants of the
last two displayed forms are respectively $-\alpha$ and $3\alpha$.
It follows that in the first case
$\alpha=-2^{v_2(d_f)} u^2$, and in the second case
$\alpha=3\cdot2^{v_2(d_f)} u^2$, for a suitable $u\in\mathbb{Z}_2^\times$.
Rescaling $x$ by $u^{-1}$ and multiplying back by ${d_f^*}$ proves
the asserted normal forms.

If $v_2(d_f)\geq 1$, the two normal forms in the statement
reduce modulo $2$ to $y^2+z^2+yz$ and $yz$, respectively,
since ${d_f^*}$ is odd.
These forms have respectively two and six zeros in
$\mathbb F_2^3$, and hence are not equivalent over
$\mathbb Z_2$.

For $v_2(d_f)=0$, let
\[
 q_1=3x^2-y^2-z^2-yz,\qquad q_2=-x^2-yz,
 \qquad
 T=\frac13
 \begin{pmatrix}
  1&1&-1\\
  4&1&-1\\
  -2&1&2
 \end{pmatrix}.
\]
One verifies that $\det T=-1/3\in\mathbb{Z}_2^\times$ and
$q_1(T\mathbf x)=q_2(\mathbf x)$. Thus the two forms are equivalent
when $v_2(d_f)=0$, also after multiplication by ${d_f^*}$.
\end{proof}

Translating this into Gorenstein orders with the help of Theorem~\ref{the:odd} we may give the following explicit formulation for the localization of odd Gorenstein orders at $2$, which shows that for odd Gorenstein orders $\mathcal{O}$, there exist exactly three families of $\mathcal{O}_2$ over $\mathbb{Z}_{2}$.
\begin{theorem}\label{the:oddorder}
Let $\mathcal{O}\subset Q$ be a Gorenstein order. Then $\mathcal{O}$ is odd if and only if one of the following statements holds:
\begin{itemize}
\item[] When $Q$ ramifies at $2$, $\mathcal{O}_2$ is isomorphic over $\mathbb{Z}_{2}$ to a residually inert order
\begin{equation}\label{gongshi4.1}
\mathcal{O}^{(2u+1)}_2\!:\,=\left\{\begin{pmatrix}
\alpha & 2^u\beta\\
2^{u+1}\bar{\beta} & \bar{\alpha}
\end{pmatrix}
: \, \alpha,\,\beta \in\mathbb{Z}_2+\frac{1+\sqrt{-3}}{2}\mathbb{Z}_2\right\}
\end{equation}
with a nonnegative integer $u$;

\vskip 0.2cm
\item[] When $Q$ splits at $2$, $\mathcal{O}_2$ is isomorphic over $\mathbb{Z}_{2}$ to a residually inert order
\begin{equation}\label{gongshi4.2}
\mathcal{O}^{(2u+2)}_2\!:\,=\left\{ M\in M_2(\mathbb{Z}_2):M\equiv
\begin{pmatrix}
a+b & b\\
b & a
\end{pmatrix}\!\!\!
\pmod{2^{u+1}},\,\,a,\,b\in \mathbb{Z}_2 \right\}
\end{equation}
or equivalent
\begin{equation*}
\mathbb{Z}_2[T]+2^{u+1}M_2(\mathbb{Z}_2),
\quad
T=\begin{pmatrix}1&1\\1&0\end{pmatrix}
\end{equation*}
with a nonnegative integer $u$;

\vskip 0.2cm
\item[] Again when $Q$ splits at $2$, $\mathcal{O}_2$ is isomorphic over $\mathbb{Z}_{2}$ to an Eichler order of level $2^v$,
\begin{equation}\label{gongshi4.3}
\begin{pmatrix}
\mathbb{Z}_2 & \mathbb{Z}_2 \\
2^v\mathbb{Z}_2  & \mathbb{Z}_2 
\end{pmatrix}
\end{equation}
with a nonnegative integer $v$.
\end{itemize}
\end{theorem}

\begin{proof}
Throughout the proof, dual bases are taken with respect to the
reduced trace pairing
\[
  \langle X,Y\rangle=\operatorname{tr}(XY).
\]
Thus, if $B_1=I$, then $B_2^*,B_3^*,B_4^*$ form a basis of the
trace-zero sublattice of the dual lattice.

First consider the ramified model. Put
$s=\sqrt{-3}$, $\omega=(1+s)/2$, and $k=2^u$.
A $\mathbb Z_2$-basis of $\mathcal{O}^{(2u+1)}_2$ is
\[
\begin{aligned}
 & B_1=I,&
 B_2=\begin{pmatrix}\omega&0\\0&\bar\omega\end{pmatrix},&
 &B_3=\begin{pmatrix}0&k\\2k&0\end{pmatrix},&
 &B_4=\begin{pmatrix}0&k\omega\\2k\bar\omega&0\end{pmatrix}.
\end{aligned}
\]
Its dual basis is
\[
\begin{aligned}
 B_1^*&=
 \begin{pmatrix}
  (3+s)/6&0\\0&(3-s)/6
 \end{pmatrix},&
 B_2^*&=
 \begin{pmatrix}
  -s/3&0\\0&s/3
 \end{pmatrix},\\
 B_3^*&=
 \begin{pmatrix}
  0&(3-s)/(12k)\\
  (3+s)/(6k)&0
 \end{pmatrix},&
 B_4^*&=
 \begin{pmatrix}
  0&s/(6k)\\
  -s/(3k)&0
 \end{pmatrix}.
\end{aligned}
\]
Indeed, direct multiplication gives
$\operatorname{tr}(B_iB_j^*)=\delta_{ij}$.
The reduced discriminant ideal of this order is $(2k^2)$, and
\[
 2k^2\operatorname n(xB_2^*+yB_3^*+zB_4^*)
 =\frac{1}{3}\bigl(2k^2x^2-y^2+yz-z^2\bigr).
\]
The substitution
\[
 (x,y,z)\longmapsto(3x,y-z,-y-2z)
\]
has determinant $-9\in\mathbb Z_2^\times$ and transforms the
right-hand side into
$6k^2x^2-y^2-z^2-yz$.
Since the global reduced discriminant and $2k^2$ differ by a
$2$-adic unit, this proves
\begin{equation}
 f_{\mathcal O}\sim_2
 \alpha\bigl(3\cdot2^{2u+1}x^2-(y^2+z^2+yz)\bigr),
 \quad \alpha\in\mathbb Z_2^\times.
\end{equation}

Next consider the split residually inert model, and put
$k=2^{u+1}$. A basis is
\[
\begin{aligned}
& B_1=I,&
& B_2=\begin{pmatrix}1&1\\1&0\end{pmatrix},
& B_3=\begin{pmatrix}0&k\\0&0\end{pmatrix},&
& B_4=\begin{pmatrix}0&0\\k&0\end{pmatrix}.
\end{aligned}
\]
The corresponding dual basis is
\[
\begin{aligned}
& B_1^*=\begin{pmatrix}0&0\\0&1\end{pmatrix},&
& B_2^*=\begin{pmatrix}1&0\\0&-1\end{pmatrix},
& B_3^*=\frac1k\begin{pmatrix}-1&0\\1&1\end{pmatrix},&
& B_4^*=\frac1k\begin{pmatrix}-1&1\\0&1\end{pmatrix}.
\end{aligned}
\]
Again, $\operatorname{tr}(B_iB_j^*)=\delta_{ij}$.
The reduced discriminant ideal is $(k^2)$, and
\[
 k^2\operatorname n(xB_2^*+yB_3^*+zB_4^*)
 =-k^2x^2+2kx(y+z)-y^2-3yz-z^2.
\]
Replacing $y,z$ by $y+2kx/5,z+2kx/5$, respectively, gives
\[
 -\frac{k^2}{5}x^2-(y^2+3yz+z^2).
\]
The binary forms $y^2+3yz+z^2$ and $y^2+yz+z^2$ are equivalent
over $\mathbb Z_2$, by the classification of even unimodular
binary lattices: both have anisotropic reduction modulo $2$.
Moreover, $-15\equiv1\pmod8$, so $-15$ is a square in
$\mathbb Z_2^\times$. Rescaling $x$ by a square root of $-15$
therefore yields
\begin{equation}
 f_{\mathcal O}\sim_2
 \alpha\bigl(3\cdot2^{2u+2}x^2-(y^2+z^2+yz)\bigr),
 \quad \alpha\in\mathbb Z_2^\times.
\end{equation}

Finally, consider the Eichler order
\[
 \begin{pmatrix}
  \mathbb Z_2&\mathbb Z_2\\
  2^v\mathbb Z_2&\mathbb Z_2
 \end{pmatrix},
 \quad v\geq0.
\]
Write $E_{ij}$ for the standard matrix units, and choose
\[
 B_1=I,\quad B_2=E_{12},\quad
 B_3=2^vE_{21},\quad B_4=E_{22}.
\]
Its dual basis is
\[
 B_1^*=E_{11},\quad B_2^*=E_{21},\quad
 B_3^*=2^{-v}E_{12},\quad B_4^*=-E_{11}+E_{22}.
\]
The reduced discriminant ideal is $(2^v)$, and
\[
 2^v\operatorname n(xB_4^*+yB_2^*+zB_3^*)
 =-2^vx^2-yz.
\]
Consequently,
\begin{equation}
 f_{\mathcal O}\sim_2
 \alpha(-2^vx^2-yz),
 \quad \alpha\in\mathbb Z_2^\times.
\end{equation}
By Theorem~\ref{the:odd} and Proposition~\ref{lem:oddtqf},
these models realize all local similarity classes associated
with odd Gorenstein orders. When $v_2(d_{f_{\mathcal O}})=0$,
the two normal forms in Proposition~\ref{lem:oddtqf} are
equivalent, and the Eichler model with $v=0$ suffices.
Multiplication of a ternary form by a $2$-adic unit does not
change the isomorphism class of its even Clifford order.
Lemma~\ref{Lemma3.3} therefore identifies $\mathcal O_2$ with
one of the displayed models, whose ambient algebras are,
respectively, ramified, split, and split.
This completes the proof.
\end{proof}

\subsection{Classification of odd orders}
Denote by $[\mathcal{O}]$ the isomorphic class of an order $\mathcal{O}\subset Q$. For $i=1,2$, we introduce the notations $\operatorname{Typ}_i(2^n)$ as follows. For $n\geq 1$, define 
\begin{equation}\label{Typ1}
\operatorname{Typ}_1(2^n):\,=\left\{[\mathcal{O}]:\, \mathcal{O}\subset Q\,\, \text{is odd and Gorenstein},\,\,\mathcal{O}_2\simeq \mathcal{O}^{(n)}_2\right\},
\end{equation}
where $\mathcal{O}^{(n)}_2$ is defined by \eqref{gongshi4.1} or \eqref{gongshi4.2} according as $n$ is odd or even respectively. Clearly, $\operatorname{Typ}_1(2^n)$ makes sense only for odd $n$ if $Q$ ramifies at $2$, and $\operatorname{Typ}_1(2^n)$ makes sense only for even $n$ if $Q$ splits at $2$. For $n\geq 0$, define 
\begin{equation}\label{Typ2}
\operatorname{Typ}_2(2^n):\,=\left\{[\mathcal{O}]:\, \mathcal{O}\subset Q\,\, \text{is odd and Gorenstein, 
$\mathcal{O}_2\simeq
\begin{pmatrix}
 \mathbb{Z}_2 & \mathbb{Z}_2 \\
 2^n\mathbb{Z}_2 & \mathbb{Z}_2
\end{pmatrix}
$}\right\}.
\end{equation}
Clearly, $\operatorname{Typ}_2(2^n)$ makes sense for all $n\geq 0$ if $Q$ splits at $2$, but it does not make sense for any $n$ if $Q$ ramifies at $2$.
\begin{remark}
For $n=0$, we extend the notation by setting
\[
 \operatorname{Typ}_1(1)\!:\,=\operatorname{Typ}_2(1).
\]
Thus the two families coincide at $n=0$, corresponding to
orders whose localization at $2$ is isomorphic to
$M_2(\mathbb Z_2)$.
\end{remark}
Moreover, by Skolem-Noether's theorem we can define for any positive integer $m\geq 1$, 
\begin{equation}\label{gongshiZ+m}
\mathbb{Z}+m\operatorname{Typ}_i(2^{n})\!:\,=\{[\mathbb{Z}+m\mathcal{O}]:\,[\mathcal{O}]\in\operatorname{Typ}_i(2^{n})\}.
\end{equation}

Here, it is appropriate to point out that, for convenience of later discussion, we say occasionally that $\mathcal{O}$ is of local level $2^n$ at $2$ of type-$i$ if $\mathcal{O}$ belongs to a class of $\operatorname{Typ}_i(2^n)$.
From this point of view, the newly defined sets $\operatorname{Typ}_i(2^n)$ may be regarded as an isomorphic equivalence classification to the set of odd Gorenstein orders of type-$i$. Thus they are crucial for our classification in what follows. 

With these in place we can give a decomposition for the set of equivalent classes of odd orders as follows.
\begin{theorem}
If $Q$ ramifies at $2$, then we have
\begin{equation*}
\{[\mathcal{O}]\!:\,\mathcal{O}\subset Q\,\, \text{is odd}\,\}=\bigsqcup_{\substack{m\geq1\\2\nmid m}}\bigsqcup_{n\geq0}\left(\mathbb{Z}+m\operatorname{Typ}_1(2^{2n+1})\right).
\end{equation*}
If $Q$ splits at $2$, then we have
\begin{equation*}
\{[\mathcal{O}]\!:\,\mathcal{O}\subset Q\,\, \text{is odd}\,\}=\left(\bigsqcup_{\substack{m\geq1\\2\nmid m}}\bigsqcup_{n\geq0}\left(\mathbb{Z}+m\operatorname{Typ}_1(2^{2n+2})\right)\right)\sqcup\left(\bigsqcup_{\substack{m\geq1\\2\nmid m}}\bigsqcup_{n\geq0}\left(\mathbb{Z}+m\operatorname{Typ}_2(2^{n})\right)\right).
\end{equation*}
\end{theorem}
\begin{proof}
By Lemma~\ref{the:Gor}, for any order $\mathcal{O}\subset Q$, there exists a unique integer $m\geq 1$ such that
\begin{equation*}
\mathcal{O}=\mathbb{Z}+m\operatorname{Gor}(\mathcal{O}).
\end{equation*}
Clearly, $\mathcal{O}$ is odd if and only if $2\nmid m$ and $\operatorname{Gor}(\mathcal{O})$ is odd. Then the theorem follows from  Theorem~\ref{the:oddorder}. 
\end{proof}

\subsection{The even part of odd orders}
For an order $\mathcal{O}\subset Q$, set 
\begin{equation}\label{O^+}
\mathcal{O}^+\!\!:\,=\{\alpha\in\mathcal{O}:\,\tr(\alpha)\in2\mathbb{Z}\}. 
\end{equation}
Clearly, $\mathcal{O}^+=\mathcal{O}$ if $\mathcal{O}$ is even. For odd $\mathcal{O}$, $\mathcal{O}^+$ is clearly a lattice, but it is not necessarily an order. More accurately, we have the following theorem.

\begin{theorem}\label{O+4}
Let $\mathcal{O}\subset Q$ be an odd order. Then $\mathcal{O}^+$ is an order if and only if $\mathcal{O}_2\not\simeq M_2(\mathbb{Z}_2)$.
\end{theorem}
\begin{proof}
Since $2\mathcal{O}\subseteq\mathcal{O}^+\subseteq\mathcal{O}$
and $\mathbb{Z}\subseteq\mathcal{O}^+$, the set
$\mathcal{O}^+$ is a full $\mathbb{Z}$-lattice containing
$\mathbb{Z}$. Thus it remains to check closure under
multiplication, which can be done locally. For every odd prime
$p$, we have
$\mathcal{O}^+_p=\mathcal{O}_p$, since
$2\mathbb{Z}_p=\mathbb{Z}_p$. It therefore suffices to consider
\[
\mathcal{O}^+_2
=
\{X\in\mathcal{O}_2:\tr(X)\in2\mathbb{Z}_2\}.
\]

First suppose that $Q$ ramifies at $2$. Put
\[
R_2=\mathbb{Z}_2+\mathbb{Z}_2\omega,
\quad
\omega=\frac{1+\sqrt{-3}}{2}.
\]
The order $\mathcal{O}_2$ is contained in the maximal order
\[
\left\{
\begin{pmatrix}
\alpha&\beta\\
2\bar\beta&\bar\alpha
\end{pmatrix}
:\alpha,\beta\in R_2
\right\}.
\]
Write $\tr(\alpha)=\alpha+\bar\alpha$ for the
field trace on $R_2$. Since
$\tr(a+b\omega)=2a+b$, we have
\[
\{\alpha\in R_2:\tr(\alpha)\in2\mathbb{Z}_2\}
=
\mathbb{Z}_2+2R_2,
\]
which is a subring of $R_2$.
Let
\[
X_i=
\begin{pmatrix}
\alpha_i&\beta_i\\
2\bar\beta_i&\bar\alpha_i
\end{pmatrix}
\in\mathcal{O}^+_2
\quad(i=1,2).
\]
Then $\alpha_1,\alpha_2\in\mathbb{Z}_2+2R_2$, and hence
$\tr(\alpha_1\alpha_2)\in2\mathbb{Z}_2$.
Direct multiplication gives
\[
\tr(X_1X_2)
=
\tr(\alpha_1\alpha_2)
+
2\tr(\beta_1\bar\beta_2)
\in2\mathbb{Z}_2.
\]
Since $X_1X_2\in\mathcal{O}_2$, it follows that
$X_1X_2\in\mathcal{O}^+_2$.

Now suppose that $Q$ splits at $2$. To apply the preceding
local classification to an arbitrary odd order, write
\[
\mathcal{O}
=
\mathbb{Z}+m\operatorname{Gor}(\mathcal{O}).
\]
Since $\mathcal{O}$ is odd, $m$ is odd, so
\[
\mathcal{O}_2=\operatorname{Gor}(\mathcal{O})_2.
\]
Moreover, $\operatorname{Gor}(\mathcal{O})$ is odd because it
contains $\mathcal{O}$. Thus the preceding classification of
odd Gorenstein orders of Theorem ~\ref{the:oddorder} shows that, after conjugation,
$\mathcal{O}_2$ has one of the following two forms:
\[
\mathbb{Z}_2[T]+2^{u+1}M_2(\mathbb{Z}_2),
\quad
T=\begin{pmatrix}1&1\\1&0\end{pmatrix},
\quad u\geq0,
\]
or
\[
\begin{pmatrix}
\mathbb{Z}_2&\mathbb{Z}_2\\
2^k\mathbb{Z}_2&\mathbb{Z}_2
\end{pmatrix},
\quad k\geq0.
\]

In the residually inert case, write
\[
X=aI+bT+2^{u+1}Y,
\quad
a,b\in\mathbb{Z}_2,\quad Y\in M_2(\mathbb{Z}_2).
\]
As $\tr(T)=1$, we have
\[
\tr(X)=2a+b+2^{u+1}\tr(Y).
\]
Consequently,
\[
\mathcal{O}^+_2
=
\mathbb{Z}_2I+2\mathbb{Z}_2T
+2^{u+1}M_2(\mathbb{Z}_2).
\]
The identity $T^2=T+I$ shows that
$\mathbb{Z}_2I+2\mathbb{Z}_2T$ is a subring of
$M_2(\mathbb{Z}_2)$. Since
$2^{u+1}M_2(\mathbb{Z}_2)$ is a two-sided ideal of
$M_2(\mathbb{Z}_2)$, the displayed expression for
$\mathcal{O}^+_2$ proves that $\mathcal{O}^+_2$ is also a ring.

In the Eichler case with $k\geq1$, let
\[
X_i=
\begin{pmatrix}
a_i&b_i\\
2^kc_i&d_i
\end{pmatrix}
\in\mathcal{O}^+_2
\quad(i=1,2).
\]
The trace condition gives $d_i\equiv a_i\pmod2$. Hence
\[
\begin{aligned}
\tr(X_1X_2)
&=
a_1a_2+d_1d_2
+2^k(b_1c_2+c_1b_2)\\
&\equiv a_1a_2+a_1a_2
\equiv0\!\!\!\pmod2.
\end{aligned}
\]
Again $X_1X_2\in\mathcal{O}_2$, so
$X_1X_2\in\mathcal{O}^+_2$.
This proves sufficiency.

For necessity, the only remaining local case is
$\mathcal{O}_2\simeq M_2(\mathbb{Z}_2)$.
In this case, the matrices
\[
U=\begin{pmatrix}0&1\\0&0\end{pmatrix},
\quad
V=\begin{pmatrix}0&0\\1&0\end{pmatrix}
\]
have trace zero, whereas
\[
UV=\begin{pmatrix}1&0\\0&0\end{pmatrix}
\]
has trace one. Thus $\mathcal{O}^+_2$ is not closed under
multiplication, and $\mathcal{O}^+$ is not an order.
The proof is complete.
\end{proof}

\subsection{Gross suborders associated to odd orders}
Let $S=\mathbb{Z}+2\mathcal{O}$ be the Gross suborder associated to an odd Gorenstein order $\mathcal{O}$.
It is clear that $S$ is not a Gorenstein order. To have a better understanding for the essence and importance of $S$, we first give a definition relative to $\operatorname{Typ}_i(2^{n})$.
For $i=1,2$ and $n\geq1$, in view of Theorem~\ref{O+4}, we can define the set $\operatorname{Typ}^+_i(2^{n})$ by
\begin{equation}\label{O+1}
\operatorname{Typ}^+_1(2^{n})\!:\,=\left\{[\mathcal{O}]:\,\mathcal{O}\subset Q,\,\, \mathcal{O}_2\simeq (\mathcal{O}^{(n)}_2)^+\,\, \text{and $\mathcal{O}_p$ is Gorenstein for $p\neq2$}\right\}.
\end{equation}
and
\begin{equation}\label{O+2}
\operatorname{Typ}^+_2(2^{n})\!:\,=\left\{[\mathcal{O}]:\,\mathcal{O}\subset Q,\,\, \mathcal{O}_2\simeq
\begin{pmatrix}
 \mathbb{Z}_2 & \mathbb{Z}_2 \\
 2^n\mathbb{Z}_2 & \mathbb{Z}_2
\end{pmatrix}^+\,\, \text{and $\mathcal{O}_p$ is Gorenstein for $p\neq2$}
\right\}.
\end{equation}

We will prove in Theorem ~\ref{the:M+} a moment later that $\operatorname{Typ}^+_i(2^{n})$ defined by~\eqref{O+1} and~\eqref{O+2} have a more direct and natural formulation as given by \eqref{O+} below, although at the first sight the latter seems just a subset of the former. 
\begin{equation}\label{O+}
\operatorname{Typ}^+_i(2^{n})=\left\{[\mathcal{O}^+]:\,[\mathcal{O}]\in \operatorname{Typ}_i(2^{n}) \right\}.
\end{equation}

Before this we give the following theorem, which establishes a one-to-one correspondence from the set $\operatorname{Typ}_i(2^n)$ of isomorphism classes of odd Gorenstein orders to the set $\operatorname{Typ}^+_i(2^{n+2})$ of isomorphism classes of orders which are neither Gorenstein nor odd. 
This perhaps can be viewed, to some extent, among others as an explanation of the reason why the Gross suborders $S$ are introduced.

\begin{theorem}\label{the:sigma2}
Let $\mathcal{O}$ denote odd Gorenstein orders. Then the map
\begin{equation}\label{sigma2}
\sigma_2: \mathcal{O} \mapsto S=\mathbb{Z}+2\mathcal{O}
\end{equation}
defined as in ~\eqref{gongshi3.6} canonically induces a  bijection from the set $\operatorname{Typ}_i(2^n)$ to the set $\operatorname{Typ}^+_i(2^{n+2})$ for $n\geq0$. 
\end{theorem}

\begin{proof}
We first prove that for any odd Gorenstein order $\mathcal{O}$ such that $\mathcal{O}\in \operatorname{Typ}_i(2^n)$ one has 
\begin{equation}\label{gongshi4.10}
S=\mathbb{Z}+2\mathcal{O}\in \operatorname{Typ}^+_i(2^{n+2}).
\end{equation}
Here and in the sequel, we sometimes use $\mathcal{O}\in \operatorname{Typ}_i(2^n)$ to stand for $[\mathcal{O}]\in \operatorname{Typ}_i(2^n)$ for ease of notations; this should not cause confusions from the context. 
We do this by considering the localization $\mathcal{O}_2$ of $\mathcal{O}$ at $2$. In fact, by Theorem~\ref{the:oddorder} we may compute isomorphically as follows. For any nonnegative integer $u\geq 0$, we have by direct computations, 
\begin{align*}
S_2&\simeq\mathbb{Z}_2+2\mathcal{O}^{(2u+1)}_2 \\
&=\left\{\begin{pmatrix}
\alpha & 2^{u+1}\beta\\
2^{u+2}\bar{\beta} & \bar{\alpha}
\end{pmatrix}
: \alpha\in\mathbb{Z}_2+\sqrt{-3}\mathbb{Z}_2,\,\,\beta\in\mathbb{Z}_2+\frac{1+\sqrt{-3}}{2}\mathbb{Z}_2\,\right\} \\
&=\left\{\begin{pmatrix}
\alpha & 2^{u+1}\beta\\
2^{u+2}\bar{\beta} & \bar{\alpha}
\end{pmatrix}
: \alpha,\,\beta\in\mathbb{Z}_2+\frac{1+\sqrt{-3}}{2}\mathbb{Z}_2,\,\,2\mid\tr(\alpha)\,\right\} \\ 
&=(\mathcal{O}^{(2u+3)}_2)^+ ,
\end{align*}
and
\begin{align*}
S_2&\simeq\mathbb{Z}_2+2\mathcal{O}^{(2u+2)}_2 \\
&=\left\{ M\in M_2(\mathbb{Z}_2):M\equiv
\begin{pmatrix}
a+b & b\\
b & a
\end{pmatrix}\!\!\!
\pmod{2^{u+2}},\,\,2\mid b\right\} \\
&=(\mathcal{O}^{(2u+4)}_2)^+.
\end{align*}
This implies what we need for \eqref{gongshi4.10} when $i=1$. 
For $i=2$, direct computation yields, for nonnegative integer $v\geq 0$, 
\begin{align*}
S_2&\simeq\mathbb{Z}_2+2\mathcal{O}^{(v)}_2 \\
&=\left\{\begin{pmatrix}
a & 2b \\
2^{v+1}c  & d
\end{pmatrix}:\,\,a,b,c,d\in\mathbb{Z}_2,\,\,a+d\in2\mathbb{Z}_2\right\} \\
&\simeq 
\left\{\begin{pmatrix}
a & b \\
2^{v+2}c  & d
\end{pmatrix}:\,\,a,b,c,d\in\mathbb{Z}_2,\,\,a+d\in2\mathbb{Z}_2\right\} \\
&=(\mathcal{O}^{(v+2)}_2)^+.
\end{align*}
This again proves \eqref{gongshi4.10}. 
Conversely, let $[\mathcal{O}']\in\operatorname{Typ}_i^+(2^{n+2})$.
The local descriptions above show that the Gorenstein conductor of
$\mathcal{O}'$ is $2$, and that
$\operatorname{Gor}(\mathcal{O}')\in\operatorname{Typ}_i(2^n)$.
By the uniqueness in Lemma~\ref{the:Gor},
\[
 \mathcal{O}'=\mathbb Z+2\operatorname{Gor}(\mathcal{O}').
\]
Thus $[\mathcal{O}']\mapsto[\operatorname{Gor}(\mathcal{O}')]$
is the inverse of $\sigma_2$.
The proof of the theorem is complete.  
\end{proof}

\begin{cor}\label{cor:O^+isGor}
Let $\mathcal{O}\in\operatorname{Typ}^+_i(2^n) $ be an order. Then $\mathcal{O}$ is Gorenstein if and only if $n=1$.
\end{cor}
\begin{proof}
By definition, $\mathcal \mathcal{O}_p$ is Gorenstein for every odd prime
$p$, and $\mathcal \mathcal{O}_2$ is isomorphic to the even-trace suborder
of the corresponding odd local model.
Thus $\operatorname{tr}(\mathcal \mathcal{O})\subseteq2\mathbb Z$,
so $\mathcal \mathcal{O}$ is even.
If $n\geq2$,  by Theorem~\ref{the:sigma2} we see that $\mathcal{O}_2$ is not Gorenstein, hence $\mathcal{O}$ is not Gorenstein.
Conversely, suppose $n=1$.
For $x,y\in\mathcal \mathcal{O}_2$, the product $xy$ belongs to
$\mathcal \mathcal{O}_2$, and hence $\operatorname{tr}(xy)\in2\mathbb Z_2$.
It follows that
\[
 \frac12\mathcal \mathcal{O}_2\subseteq\mathcal \mathcal{O}_2^\sharp.
\]
The local order $\mathcal O_2$ has reduced
discriminant $4$, so
$[\mathcal O_2^\sharp:\mathcal O_2]=4^2$.
On the other hand,
\[
 [\mathcal \mathcal{O}_2^\sharp:\mathcal \mathcal{O}_2]=2^4
 =[\tfrac12\mathcal \mathcal{O}_2:\mathcal \mathcal{O}_2].
\]
Therefore
\[
 \mathcal \mathcal{O}_2^\sharp=\frac12\mathcal \mathcal{O}_2,
\]
which is an invertible fractional $\mathcal \mathcal{O}_2$-ideal.
Hence $\mathcal O_2$ is Gorenstein. Together with the
Gorenstein property at every odd prime, this proves that
$\mathcal O$ is Gorenstein.
\end{proof}

\begin{prop}\label{prop:Watsontrans}
Let $\mathcal{O}$ be an odd order, then $\lambda_4(f_{S^{0}})=f_{\mathcal{O}^{0}}$, where $\lambda_4$ is defined as in Section {\rm ~\ref{section3.3}}.
\end{prop}

\begin{proof}
Since $\mathcal{O}$ is odd, by Lemma \ref{prop:basis} we can write $\mathcal{O}=\mathbb{Z}+\mathbb{Z}\beta_1+\mathbb{Z}\beta_2+\mathbb{Z}\beta_3$ with $\tr(\beta_1)=\tr(\beta_2)=0,\,\tr(\beta_3)=1$. Hence $\mathcal{O}^0=\mathbb{Z}\beta_1+\mathbb{Z}\beta_2+\mathbb{Z}(2\beta_3-1)$, $S^0=\mathbb{Z}2\beta_1+\mathbb{Z}2\beta_2+\mathbb{Z}(2\beta_3-1)$, and by the definition formulas ~\eqref{gongshi3.7} and  ~\eqref{gongshi3.8} we have
\begin{align*}
&f_{S^{0}}=4\n(\beta_1)x^2+4\n(\beta_2)y^2+(4\n(\beta_3)-1)z^2+4\tr(\beta_2\overline{\beta_3})yz
+4\tr(\beta_1\overline{\beta_3})xz+4\tr(\beta_1\overline{\beta_2})xy,\\
&f_{\mathcal{O}^{0}}=\n(\beta_1)x^2+\n(\beta_2)y^2+(4\n(\beta_3)-1)z^2+2\tr(\beta_2\overline{\beta_3})yz
+2\tr(\beta_1\overline{\beta_3})xz+\tr(\beta_1\overline{\beta_2})xy. 
\end{align*}
Whence by the definition of $\lambda_4$ we get
\begin{align*}
&\quad \lambda_4(f_{S^{0}})(x,y,z)=\frac14\cdot f_{S^{0}}(x,y,2z)\\
&=\frac14\cdot \left(4\n(\beta_1)x^2+4\n(\beta_2)y^2+4(4\n(\beta_3)-1)z^2+8\tr(\beta_2\overline{\beta_3})yz
+8\tr(\beta_1\overline{\beta_3})xz+4\tr(\beta_1\overline{\beta_2})xy\right)\\
&=f_{\mathcal{O}^{0}}(x,y,z),
\end{align*}
as desired. 
\end{proof}

\begin{theorem}\label{lem:odd}
Let $\mathcal{O}$ be an odd Gorenstein order, $S=\mathbb{Z}+2\mathcal{O}$. Then we have
\begin{equation*}
m_{f_{\mathcal{O}}}\psi_{f_{\mathcal{O}}}=f_{S^0}.
\end{equation*} 
\end{theorem}

\begin{proof}
Let $\mathcal{O}=\mathbb{Z}+\mathbb{Z}e_1+\mathbb{Z}e_2+\mathbb{Z}e_3$, and the corresponding primitive quadratic form by even Clifford algebra construction is  
\begin{equation*}
f_{\mathcal{O}}=(a,b,c,r,s,t),
\end{equation*}
with
\begin{equation*}
\begin{cases}
& \!\!\!\! e_1^2=re_1-bc,\,\,e_2e_3=a\overline{e_1},\\
& \!\!\!\! e_2^2=se_2-ac,\,\,e_3e_1=b\overline{e_2},\\
& \!\!\!\! e_3^2=te_3-ab,\,\,e_1e_2=c\overline{e_3}.
\end{cases}
\end{equation*}
Then we have $\tr{(2e_1-r)}=\tr{(2e_2-s)}=\tr{(2e_3-t)}=0$. Hence 
\begin{equation*}
S^0=\mathbb{Z}(2e_1-r)+\mathbb{Z}(2e_2-s)+\mathbb{Z}(2e_3-t).
\end{equation*}
Also, simple computation gives, by definition, 
\begin{equation*}
\begin{aligned}
& \!\!\!\! \n(2e_1-r)=4bc-r^2,\,\,\,\,\tr((2e_2-s)\overline{(2e_3-t)})=2st-4ar,\\
& \!\!\!\! \n(2e_2-s)=4ac-s^2,\,\,\,\,\tr((2e_3-t)\overline{(2e_1-r)})=2rt-4bs,\\
& \!\!\!\! \n(2e_3-t)=4ab-t^2,\,\,\,\,\tr((2e_1-r)\overline{(2e_2-s)})=2rs-4ct.
\end{aligned}
\end{equation*}
Thus by ~\eqref{gongshi3.7}, and then by ~\eqref{gongshi2.01} and ~\eqref{gongshi2.02} we get 
\begin{equation*}
f_{S^0}=(4bc-r^2,4ac-s^2,4ab-t^2,2st-4ar,2rt-4bs,2rs-4ct)=m_{f_\mathcal{O}} \psi_{f_\mathcal{O}},
\end{equation*}
as stated. 
\end{proof}

\subsection{Chains of families of orders}
Let $\mathcal{O}$ denote odd Gorenstein orders. Define the map ${\mathcal M}_+$ by
\begin{equation}\label{gongshi3.16}
{\mathcal M}_+:\,\mathcal{O}\mapsto\mathcal{O}^+=\{\alpha\in\mathcal{O}:\,2\mid\tr(\alpha)\,\}.
\end{equation}
In this subsection we first prove that ${\mathcal M}_+$ may induce a canonical surjection from $\operatorname{Typ}_i(2^n)$ to $\operatorname{Typ}^+_i(2^n)$, and then use the properties of ${\mathcal M}_+$ and $\sigma_2$, which is defined by ~\eqref{sigma2}, to establish chains containing both the sequences of $\operatorname{Typ}_i(2^n)$ and $\operatorname{Typ}^+_i(2^n)$ as $n$ varies. 

\begin{theorem}\label{the:even}
Let $\mathcal{O}$ be an even Gorenstein order, then
\begin{equation*}
m_{f_{\mathcal{O}}}\psi_{f_{\mathcal{O}}}=4f_{\mathcal{O}^0}. 
\end{equation*} 
\end{theorem}

\begin{proof}
Let $\mathcal{O}=\mathbb{Z}+\mathbb{Z}e_1+\mathbb{Z}e_2+\mathbb{Z}e_3$, and the corresponding primitive quadratic form by even Clifford algebra construction is  
\begin{equation*}
f_{\mathcal{O}}=(a,b,c,r,s,t),
\end{equation*}
with
\begin{equation*}
\begin{cases}
& \!\!\!\! e_1^2=re_1-bc,\,\,e_2e_3=a\overline{e_1},\\
& \!\!\!\! e_2^2=se_2-ac,\,\,e_3e_1=b\overline{e_2},\\
& \!\!\!\! e_3^2=te_3-ab,\,\,e_1e_2=c\overline{e_3}.
\end{cases}
\end{equation*}
Since $\mathcal{O}$ is even, we have $2\mid r,s,t$, the traces of $e_1$, $e_2$ and $e_3$ respectively. Then from $\tr(e_1-r/2)=\tr(e_2-s/2)=\tr(e_3-t/2)=0$ we obtain 
\begin{equation*}
\mathcal{O}^0=\mathbb{Z}(e_1-r/2)+\mathbb{Z}(e_2-s/2)+\mathbb{Z}(e_3-t/2).
\end{equation*}
Also, simple computation gives
\begin{equation*}
\begin{aligned}
& \!\!\!\! 4\n(e_1-r/2)=4bc-r^2,\,\,\,4\tr((e_2-s/2)\overline{(e_3-t/2)})=2st-4ar,\\
& \!\!\!\! 4\n(e_2-s/2)=4ac-s^2,\,\,\,4\tr((e_3-t/2)\overline{(e_1-r/2)})=2rt-4bs,\\
& \!\!\!\! 4\n(e_3-t/2)=4ab-t^2,\,\,\,4\tr((e_1-r/2)\overline{(e_2-s/2)})=2rs-4ct.
\end{aligned}
\end{equation*}
Thus by ~\eqref{gongshi3.7}, and then by ~\eqref{gongshi2.01} and ~\eqref{gongshi2.02} we get 
\begin{equation*}
4f_{{\mathcal{O}^0}}=(4bc-r^2,4ac-s^2,4ab-t^2,2st-4ar,2rt-4bs,2rs-4ct)=m_{f_\mathcal{O}}\psi_{f_\mathcal{O}},
\end{equation*}
as stated. 
\end{proof}

\begin{theorem}\label{the:M+}
Let ${\mathcal M}_+$ be defined as in ~\eqref{gongshi3.16}. Then the induced map, again denoted by ${\mathcal M}_+$, from the isomorphism classes of $\operatorname{Typ}_i(2^n)$ to the isomorphism classes of $\operatorname{Typ}^+_i(2^n)$ is surjective if $n\geq1$. Moreover, if $n=1$, then ${\mathcal M}_+$ is bijective.
\end{theorem}
\begin{proof}
Let $\mathcal{O}\in\operatorname{Typ}_i(2^n)$ denote odd Gorenstein orders.
Then by Theorem ~\ref{lem:odd} we see that the map ${\mu_1}\!:\,\mathcal{O}\mapsto f_{S^0}$ induces a bijection from $\operatorname{Typ}_i(2^n)$ to the equivalence classes of corresponding forms. 
If $n>1$, by Proposition ~\ref{prop:Watsontrans} and Lemma ~\ref{Watsontrans}, we see that there is a surjection between the equivalence classes of $f_{S^0}$ and of $f_{\mathcal{O}^0}$ by $\lambda_4$. 
Further, since $\sigma_2\!:\mathcal{O}\mapsto S$ is bijective, we see that ${\mu_0}:S\mapsto f_{S^0}$, i.e., ${\mu_0}={\mu_1}\circ\sigma_2^{-1}:\mathcal{O}^+\mapsto f_{(\mathcal{O}^+)^0}$ is bijective.
Again, it is clear that $\mathcal{O}^0=(\mathcal{O}^+)^0$. 
Thus we have the following sequence of correspondences
\begin{equation*}
\operatorname{Typ}_i(2^{n}) \xleftrightarrow[]{{\mu_1}} \left\{[f_{S^0}]:S=\mathbb{Z}+2\mathcal{O}\right\} \xrightarrow[]{\lambda_4} \left\{[f_{\mathcal{O}^0}]\right\}=\left\{[f_{(\mathcal{O}^+)^0}]\right\} \xleftrightarrow[]{{\mu_0}} \left\{[\mathcal{O}^+]\right\}=\operatorname{Typ}_i^+(2^{n}).
\end{equation*}
If $n=1$, then one has $\mathcal{O}^+$ is even Gorenstein order by Corollary~\ref{cor:O^+isGor}.
Since $\lambda_4(f_{S^0})=f_{(\mathcal{O})^0}=f_{(\mathcal{O}^+)^0}$ is bijective, and $\mu_0:\mathcal{O}^+\mapsto f_{(\mathcal{O}^+)^0}$ is bijective by Theorem~\ref{the:even}, then we see that ${\mathcal M}_+=\mu_0^{-1}\lambda_4\circ\mu_1$ is bijective. 
This completes the proof of the theorem.
\end{proof}

The combination of Theorems ~\ref{the:sigma2} and ~\ref{the:M+} clearly implies the following theorem, which gives a concise characterization to the relationship of odd Gorenstein orders $\mathcal{O}$ and the corresponding Gross orders $S=\mathbb{Z}+2\mathcal{O}$ of different local level at $2$ in terms of the sets $\operatorname{Typ}_i(2^{n})$ and $\operatorname{Typ}_i^+(2^{n})$, and hence demonstrates the importance of Gross orders. 

\begin{theorem}\label{chains1}
There exist chains of correspondences as for the odd Gorenstein orders and the corresponding Gross orders as follows.
\begin{enumerate}
\item  For odd positive integer $n$ we have the chain:
\begin{equation*}
\begin{tikzcd}
\cdots  \arrow[r, "{\mathcal M}_+", two heads] & \operatorname{Typ}_i^+(2^{n+4}) \arrow[leftrightarrow, r, "\sigma_2^{-1}"] & \operatorname{Typ}_i(2^{n+2}) \arrow[r, "{\mathcal M}_+", two heads]  & \operatorname{Typ}_i^+(2^{n+2})  \arrow[leftrightarrow, r, "\sigma_2^{-1}"] & \cdots  \arrow[leftrightarrow, r, "\sigma_2^{-1}"] &  \operatorname{Typ}_i(2) \arrow[leftrightarrow, r, "{\mathcal M}_+"] &  \operatorname{Typ}_i^+(2).
\end{tikzcd}
\end{equation*} 
\item For even positive integer $n$ we have two chains:
\begin{equation*}
\begin{tikzcd}
\cdots \arrow[r, "{\mathcal M}_+", two heads] & \operatorname{Typ}_1^+(2^{n+4}) \arrow[leftrightarrow, r, "\sigma_2^{-1}"] & \operatorname{Typ}_1(2^{n+2}) \arrow[r, "{\mathcal M}_+", two heads] & \operatorname{Typ}_1^+(2^{n+2}) \arrow[leftrightarrow, r, "\sigma_2^{-1}"] & \cdots \arrow[leftrightarrow, r, "\sigma_2^{-1}"] & \operatorname{Typ}_1(4) \arrow[d, "{\mathcal M}_+", two heads] &                                     &                         \\
              &                                           &                                         &                                           &               &                                     \operatorname{Typ}_i^+(4) \arrow[leftrightarrow, r, "\sigma_2^{-1}"] & \operatorname{Typ}_i(1). \\
\cdots \arrow[r, "{\mathcal M}_+", two heads] & \operatorname{Typ}_2^+(2^{n+4}) \arrow[leftrightarrow, r, "\sigma_2^{-1}"] & \operatorname{Typ}_2(2^{n+2}) \arrow[r, "{\mathcal M}_+", two heads] & \operatorname{Typ}_2^+(2^{n+2}) \arrow[leftrightarrow, r, "\sigma_2^{-1}"] & \cdots \arrow[leftrightarrow, r, "\sigma_2^{-1}"] & \operatorname{Typ}_2(4) \arrow[u, "{\mathcal M}_+", two heads] &                                     &                        
\end{tikzcd}
\end{equation*}
\end{enumerate}
\end{theorem}

At the end of this section, we formulate a relation of the correspondences $\sigma_2$ and $M^+$ on odd Gorenstein orders $\mathcal{O}$ with their radical idealizers in terms of the localization of them at $2$ in Proposition ~\ref{idealizer1} below, which shows that the result of our Theorem ~\ref{chains1} above is a refining of that of the following Proposition ~\ref{idealizer} for $p=2$. 

Recall that for a quaternion algebra $Q_p$ over $\mathbb{Q}_p$, the radical idealizer $\mathcal{O}_p^{\natural}$ of an order $\mathcal{O}_p$ is defined as 
\begin{equation*}
\mathcal{O}_p^{\natural}\!:\,=\mathcal{O}_{\operatorname{L}}(\operatorname{rad}(\mathcal{O}_p))
\cap\mathcal{O}_{\operatorname{R}}(\operatorname{rad}(\mathcal{O}_p)),
\end{equation*}
where 
\begin{equation*}
\mathcal{O}_{\operatorname{L}}(I)\!:\,=\{\alpha\in Q_p:\alpha I\subset I\},\,\,\,\mathcal{O}_{\operatorname{R}}(I)\!:\,=\{\alpha\in Q_p:I\alpha\subset I\},
\end{equation*}
and $\operatorname{rad}(\mathcal{O}_p)$ denotes the Jacobson radical of $\mathcal{O}_p$, which is defined as the intersection of all maximal left ideals of $\mathcal{O}$. It is known that 

\begin{prop}[\cite{Voi21}, 24.4.5, Proposition 24.4.7]\label{idealizer}
For prime $p\geq 2$, the following statements hold for every non-negative integer $n\geq 0$:
\begin{enumerate}
\item If $\mathcal{O}_p$ is an Eichler order of level $p^{n+2}$, then $\mathcal{O}_p^{\natural}$ is an Eichler order of level $p^{n}$.
\item If $\mathcal{O}_p$ is residually inert but is not maximal, then $\mathcal{O}_p^{\natural}$ is the unique minimal, residually inert or maximal superorder  of $\mathcal{O}_p$, and we have $[\mathcal{O}_p^{\natural}:\mathcal{O}_p]=p^2$. 
\end{enumerate}
\end{prop}

The following proposition is a straightforward verification from definitions. 
\begin{prop}\label{idealizer1}
For any positive integer $n\geq2$, if $\mathcal{O}$ is an order such that $[\mathcal{O}]\in \operatorname{Typ}_i(2^{n})$ then we have
\begin{equation*}
\sigma^{-1}_2\circ \mathcal{M}^+(\mathcal{O}_2)=\mathcal{O}_2^{\natural}.
\end{equation*}
\end{prop}

\section{Proof of the main results}
In this section we complete the proofs of Theorem ~\ref{the:bijection}, and Corollaries~\ref{cor:tqf} to ~\ref{cor:trace}. For this purpose we first prove the following Lemma ~\ref{lem:rep}. 

\begin{lemma}\label{lem:rep}
Let $\mathcal O$ be an order in a definite quaternion algebra $Q$
over $\mathbb Q$, and put $S=\mathbb Z+2\mathcal O$.
Then the function
\[
 {\frak q}_{\mathcal O}:\mathcal O/\mathbb Z\longrightarrow\mathbb Z,
 \quad
 {\frak q}_{\mathcal O}([\alpha])
 =4\n(\alpha)-\tr(\alpha)^2,
\]
is a well-defined positive definite integral ternary quadratic form
on the quotient $\mathbb Z$-module $\mathcal O/\mathbb Z$.
The map
\[
 \Phi:\mathcal O/\mathbb Z\longrightarrow S^0,
 \quad
 [\alpha]\longmapsto2\alpha-\tr(\alpha),
\]
is an isomorphism of $\mathbb Z$-modules satisfying
$\n(\Phi([\alpha]))={\frak q}_{\mathcal O}([\alpha])$.
Consequently, for $n\in\mathbb Z_{\geq0}$ and $r\in\mathbb Z$,
\[
 \rho_{\mathcal O}(n,r)
 =
 R_{{\frak q}_{\mathcal O}}(4n-r^2)
 =
 R_{f_{S^0}}(4n-r^2).
\]
If $\mathcal O$ is even, then $S^0=2\mathcal O^0$, and hence
\[
 \rho_{\mathcal O}(n,r)
 =
 R_{f_{\mathcal O^0}}\left(n-\frac{r^2}{4}\right).
\]
\end{lemma}
\begin{proof}
For $\alpha\in\mathcal O$ and $a\in\mathbb Z$, we have
\[
 \n(\alpha+a)=\n(\alpha)+a\tr(\alpha)+a^2,
 \quad
 \tr(\alpha+a)=\tr(\alpha)+2a.
\]
Thus $4\n(\alpha)-\tr(\alpha)^2$ is unchanged by adding an integer
to $\alpha$, i.e., ${\frak q}_{\mathcal O}$ is well defined.
Since $\mathcal O\cap\mathbb Q=\mathbb Z$,
the quotient $\mathcal O/\mathbb Z$ is free of rank three.

The map $\Phi$ is well defined and $\mathbb Z$-linear.
Its image lies in $S^0$, and its kernel is trivial:
if $2\alpha-\tr(\alpha)=0$, then
$\alpha\in\mathcal O\cap\mathbb Q=\mathbb Z$.
Conversely, every $v\in S^0$ can be written as
$v=a+2\alpha$ with $a\in\mathbb Z$ and $\alpha\in\mathcal O$.
The equality $\tr(v)=0$ gives $a=-\tr(\alpha)$, so
$v=\Phi([\alpha])$.
Moreover,
\[
 \n(2\alpha-\tr(\alpha))
 =4\n(\alpha)-\tr(\alpha)^2.
\]
This proves the asserted isometry and the positive definiteness
and integrality of ${\frak q}_{\mathcal O}$.

We now show that passage to the quotient gives a bijection
\[
 \{\alpha\in\mathcal O:\n(\alpha)=n,\ \tr(\alpha)=r\}
 \longrightarrow
 \{[\beta]\in\mathcal O/\mathbb Z:
                  {\frak q}_{\mathcal O}([\beta])=4n-r^2\}.
\]
Let $[\beta]$ belong to the set on the right and set
$t=\tr(\beta)$. Reducing
\[
 4\n(\beta)-t^2=4n-r^2
\]
modulo $4$ gives $t\equiv r\pmod2$. Therefore
\[
 \alpha=\beta+\frac{r-t}{2}
\]
belongs to $\mathcal O$ and has trace $r$.
Since ${\frak q}_{\mathcal O}([\alpha])={\frak q}_{\mathcal O}([\beta])$,
we also have $\n(\alpha)=n$.
This is the unique representative of $[\beta]$ with trace $r$,
because adding an integer $a$ changes the trace by $2a$.
Then the representation-number identities follow.

Finally, if $\mathcal O$ is even, then
$2\alpha-\tr(\alpha)=2(\alpha-\tr(\alpha)/2)\in2\mathcal O^0$,
and every $2\gamma$ with $\gamma\in\mathcal O^0$ occurs in this way.
Thus $S^0=2\mathcal O^0$, which proves the last formula.
\end{proof}

\begin{cor}\label{cor:m_f}
For a Gorenstein order $\mathcal{O}$ we have 
\begin{equation}\label{gongshi1.2}
m_{f_{\mathcal{O}}}=
\begin{cases}
4\widetilde{m}_{{f}_{\mathcal{O}^0}} & \!\!\!\! \text{if\, $\mathcal{O}$ is even,}\\
\widetilde{m}_{{f}_{S^0}} & \!\!\!\! \text{if\, $\mathcal{O}$ is odd.}
\end{cases}
\end{equation}
\end{cor}

\begin{proof}
It is a direct consequence of Theorems ~\ref{lem:odd} and ~\ref{the:even}. 
\end{proof}

\begin{remark}\label{rem:fO0}
Let $\mathcal O=\mathbb Z+m\mathcal O'$ with $\mathcal O'$ Gorenstein,
and put $S'=\mathbb Z+2\mathcal O'$.
Then
\[
 \mathcal O^0=
 \begin{cases}
 (m/2)S'^0&\text{if $\mathcal O'$ is odd and $2\mid m$},\\
 m\mathcal O'^0&\text{otherwise}.
 \end{cases}
\]
Indeed, $a+m\beta$ has trace zero precisely when
$2a+m\tr(\beta)=0$. If $m$ is even, this gives
$a+m\beta=(m/2)(2\beta-\tr(\beta))$; if $m$ is odd,
it forces $\tr(\beta)$ to be even, giving
$a+m\beta=m(\beta-\tr(\beta)/2)$.
When $\mathcal O'$ is even, we also have $S'^0=2\mathcal O'^0$.
Consequently,
\[
 f_{\mathcal O^0}=
 \begin{cases}
 (m/2)^2f_{S'^0}&
       \text{if $\mathcal O'$ is odd and $2\mid m$},\\
 m^2f_{\mathcal O'^0}&\text{otherwise}.
 \end{cases}
\]
\end{remark}

\begin{proof}[Proof of Theorem~{\upshape\ref{the:bijection}}]
By Lemma ~\ref{the:Gor} we see that two orders $\mathcal{O}'=\mathbb{Z}+m_1\operatorname{Gor}(\mathcal{O}')$ and $\mathcal{O}''=\mathbb{Z}+m_2\operatorname{Gor}(\mathcal{O}'')$ with conductors $m_1$ and $m_2$ respectively are isomorphic if and only if $m_1=m_2$, and $\operatorname{Gor}(\mathcal{O}')$ and $\operatorname{Gor}(\mathcal{O}'')$ are isomorphic. 
By definition we see that two quadratic forms $f'$ and $f''$ are equivalent if and only if they have the same content and their primitive parts are equivalent. These together with ~\eqref{gongshi2.02} and Definition ~\ref{def:tildef} show that, except for equality (3), to prove Theorem~{\upshape\ref{the:bijection}} in general one only needs to prove it when the Gorenstein and primitive conditions are imposed on orders and forms respectively. 
Now assume $\mathcal{O}$ is Gorenstein, then for odd $\mathcal{O}$, by Theorem ~\ref{lem:odd} we have 
\begin{equation*}
\psi_{f_{\mathcal{O}}}=f_{S^0}/m_{f_{\mathcal{O}}}=\widetilde{f}_{\mathcal{O}};
\end{equation*} 
and for even $\mathcal{O}$, by Theorem ~\ref{the:even} we have 
\begin{equation*}
\psi_{f_{\mathcal{O}}}=4f_{\mathcal{O}^0}/m_{f_{\mathcal{O}}}=\widetilde{f}_{\mathcal{O}}.
\end{equation*} 
The combination of these two identities proves assertion (1).  

The reciprocal operation preserves the content of a ternary
quadratic form. Hence assertion (1) gives
\[
 \widetilde m_{\widetilde f_{\mathcal O}}
 =
 \widetilde m_{f_{\mathcal O}}.
\]
By the primitivity criterion, we obtain
\[
 \mathcal O\text{ is Gorenstein}
 \quad\Longleftrightarrow\quad
 f_{\mathcal O}\text{ is primitive}
 \quad\Longleftrightarrow\quad
 \widetilde f_{\mathcal O}\text{ is primitive}.
\]

Turn to assertion (3). Write
$f_{\mathcal O}=(a,b,c,r,s,t)$ and choose a corresponding basis
$1,e_1,e_2,e_3$ of $\mathcal O$.
In the quotient basis $[e_1],[e_2],[e_3]$, direct computation gives
\begin{align*}
 {\frak q}_{\mathcal O}(x,y,z)
 ={}&
 (4bc-r^2)x^2+(4ac-s^2)y^2+(4ab-t^2)z^2\\
 &+(2st-4ar)yz+(2rt-4bs)xz+(2rs-4ct)xy.
\end{align*}
Since $f_{\mathcal O}=m f_{\operatorname{Gor}(\mathcal O)}$,
the content of $f_{\mathcal O}$ is $m$.
The definition of the reciprocal therefore gives
\[
 {\frak q}_{\mathcal O}
 =\frac{m_{f_{\mathcal O}}}{m}\,\psi_{f_{\mathcal O}}
 =\frac{m_{f_{\mathcal O}}}{m}\,\widetilde f_{\mathcal O},
\]
where the last equality follows from assertion (1).
Applying Lemma~\ref{lem:rep}, we obtain
\[
 \rho_{\mathcal O}(n,r)
 =R_{{\frak q}_{\mathcal O}}(4n-r^2)
 =R_{\widetilde f_{\mathcal O}}
   \left(\frac{m(4n-r^2)}{m_{f_{\mathcal O}}}\right).
\]
This proves assertion (3). For assertion (4), since reciprocal preserves genera then the correspondence also preserves genera.
By the preceding parts of the theorem and~\cite[Remark 22.4.3]{Voi21} we have for any quaternion order $\mathcal{O}$, 
\begin{equation*}
2|\operatorname{Aut}(\mathcal{O})|=|\operatorname{Aut}(f_{\mathcal{O}})|=|\operatorname{Aut}(\psi_{f_{\mathcal{O}}})|=|\operatorname{Aut}(\widetilde{f}_{\mathcal{O}})|.
\end{equation*}
Hence
\begin{equation*}
\sum_{\mu=1}^t \frac{\rho_{\mathcal{O}_\mu}(n,r)}{|\Aut(\mathcal{O}_\mu)|}=2\sum\limits_{f\in \operatorname{gen}(\widetilde f_{\mathcal{O}})}\frac{R_{f}\left({\frac{m(4n-r^2)}{m_{f_{\mathcal{O}}}}}\right)}{|\Aut(f)|}=2R_{\operatorname{gen}(\widetilde{f}_\mathcal{O})}\left(\frac{m(4n-r^2)}{m_{f_{\mathcal{O}}}}\right).
\end{equation*}
The proof of Theorem~{\upshape\ref{the:bijection}} is complete. 
\end{proof}

\begin{proof}[Proof of Corollary~{\upshape\ref{cor:tqf}}]
We first prove the necessity. So we assume that there exists a quaternion order $\mathcal{O}$ such that 
\begin{equation}\label{gongshi4.01}
f_{\mathcal{O}^0}=f.
\end{equation} 
We first consider the case that $\mathcal{O}$ is Gorenstein. 
For any Gorenstein order $\mathcal{O}$, since ${f}_\mathcal{O}$ is the reciprocal of $\widetilde{f}_\mathcal{O}$ by Theorem~\ref{the:bijection}, thus  
\begin{equation}\label{gongshi4.05}
m_{f_\mathcal{O}}=N_{\widetilde{f}_\mathcal{O}}^2/d_{\widetilde{f}_\mathcal{O}}.
\end{equation}
If $\mathcal{O}$ is even, then by~\eqref{gongshi4.01} we have $f=f_{\mathcal{O}^0}=\widetilde m_{f_{\mathcal{O}^0}}\widetilde{f}_\mathcal{O}$.  
Hence we have
\begin{equation}\label{gongshi4.04}
4d_{f}/N_{f}^2=4(\widetilde{m}_{f_{\mathcal{O}^0}}^3d_{\widetilde{f}_\mathcal{O}})/(\widetilde{m}_{f_{\mathcal{O}^0}}N_{\widetilde{f}_\mathcal{O}})^2
=4\widetilde{m}_{f_{\mathcal{O}^0}}d_{\widetilde{f}_\mathcal{O}}/N_{\widetilde{f}_\mathcal{O}}^2. 
\end{equation}
Thus from \eqref{gongshi4.05} and in view of $4\widetilde{m}_{f_{\mathcal{O}^0}}=m_{f_\mathcal{O}}$ by \eqref{gongshi1.2}, we get  
$$
16d_{f}/N_{f}^2=16\widetilde{m}_{f_{\mathcal{O}^0}}m_{f_\mathcal{O}}^{-1}=4, 
$$ 
a square integer as desired. 

If $\mathcal{O}$ is odd and $\mathcal{O}\in\operatorname{Typ}_i(2^{n+2})$ with $n\geq0$ and $1\leq i
\leq 2$, then we may set $\mathcal{O}'=\sigma_2^{-1}\circ {\mathcal M}_+(\mathcal{O})$. Hence by definition we have $S'\!:\,=\mathbb{Z}+2\mathcal{O}'=\mathcal{O}^+$. Here we note that $\mathcal{O}'$ is odd and Gorenstein as $\mathcal{O}$ does. 
Thus by ~\eqref{gongshi4.01}, and then by~\eqref{gongshi1.2} we obtain
\begin{equation*}
f=f_{\mathcal{O}^0}=f_{(\mathcal{O}^+)^0}=f_{S'^0}={m}_{f_{\mathcal{O}'}}\widetilde{f}_{\mathcal{O}'}.
\end{equation*}
Then by~\eqref{gongshi4.05} we get 
\begin{equation*}
16d_{f}/N_{f}^2
=16{m}_{{f}_{\mathcal{O}'}}d_{\widetilde{f}_{\mathcal{O}'}}/N_{\widetilde{f}_{\mathcal{O}'}}^2
=16m_{f_{\mathcal{O}'}} m^{-1}_{f_{\mathcal{O}'}}
=16, 
\end{equation*}
a square integer as desired. 

If $\mathcal{O}$ is odd and $\mathcal{O}\in\operatorname{Typ}_i(2)$ with $1\leq i \leq 2$, then by ~\eqref{gongshi4.01} and Proposition~\ref{prop:Watsontrans} we have
\begin{equation}\label{gongshi4.06}
f=f_{\mathcal{O}^0}=\lambda_4(f_{S^0})={m}_{{f}_{\mathcal{O}}}\lambda_4(\widetilde{f}_{\mathcal{O}}).
\end{equation}
Again in view of $\mathcal{O}\in\operatorname{Typ}_i(2)$, one can deduce $v_2(d_{f_{\mathcal{O}}})=1$; whence by Theorem ~\ref{the:odd} we obtain $v_2(N_{f_{\mathcal{O}}})=v_2(d_{f_{\mathcal{O}}})+2=3$. Hence as the reciprocal of ${f}_{\mathcal{O}}$ we have for $\widetilde{f}_{\mathcal{O}}$, 
\begin{equation*}
v_2(N_{\widetilde f_{\mathcal{O}}})=v_2(N_{f_{\mathcal{O}}})=3,\,\,\,
v_2(d_{\widetilde f_{\mathcal{O}}})=v_2(N_{f_{\mathcal{O}}}^3/(4d_{f_{\mathcal{O}}}))=6.
\end{equation*}
Thus by the first statement of Lemma~\ref{Watsontrans111} we have
\begin{equation*}
N_{\lambda_4(\widetilde{f}_{\mathcal{O}})}=N_{\widetilde{f}_{\mathcal{O}}}/2,\,\, d_{\lambda_4(\widetilde{f}_{\mathcal{O}})}=d_{\widetilde{f}_{\mathcal{O}}}/16.
\end{equation*}
Then by~\eqref{gongshi4.05} and ~\eqref{gongshi4.06} we get, on noting that $\lambda_4(\widetilde{f}_{\mathcal{O}})$ is primitive as $\widetilde{f}_{\mathcal{O}}$ does, 
\begin{equation*}
16d_{f}/N_{f}^2
=16{m}_{{f}_{\mathcal{O}}}
d_{\lambda_4(\widetilde{f}_{\mathcal{O}})}/N_{\lambda_4(\widetilde{f}_{\mathcal{O}})}^2
=16m_{f_{\mathcal{O}}}(d_{\widetilde{f}_{\mathcal{O}}}/16)/(N_{\widetilde{f}_{\mathcal{O}}}/2)^2
=4m_{f_{\mathcal{O}}}m^{-1}_{f_{\mathcal{O}}}
=4,
\end{equation*}
a square integer as desired. 

If $\mathcal{O}$ is odd and $\mathcal{O}\in\operatorname{Typ}_2(1)$, then again by ~\eqref{gongshi4.01} and Proposition~\ref{prop:Watsontrans} we see that  ~\eqref{gongshi4.06} still holds.
However, this time one can deduce $v_2(d_{f_{\mathcal{O}}})=0$; and whence by Theorem ~\ref{the:odd} we obtain $v_2(N_{f_{\mathcal{O}}})=v_2(d_{f_{\mathcal{O}}})+2=2$. Hence as the reciprocal of ${f}_{\mathcal{O}}$ we have for $\widetilde{f}_{\mathcal{O}}$, 
\begin{equation*}
v_2(N_{\widetilde f_{\mathcal{O}}})=v_2(N_{f_{\mathcal{O}}})=2,\,\,\,
v_2(d_{\widetilde f_{\mathcal{O}}})=v_2(N_{f_{\mathcal{O}}}^3/(4d_{f_{\mathcal{O}}}))=4.
\end{equation*}
Then by the second statement of Lemma~\ref{Watsontrans111} we have
\begin{equation*}
N_{\lambda_4(\widetilde{f}_{\mathcal{O}})}=N_{\widetilde{f}_{\mathcal{O}}},\,\, d_{\lambda_4(\widetilde{f}_{\mathcal{O}})}=d_{\widetilde{f}_{\mathcal{O}}}/16.
\end{equation*}
Thus by ~\eqref{gongshi4.05} and ~\eqref{gongshi4.06} once more, we have  
\begin{equation*}
16d_{f}/N_{f}^2
=16{m}_{{f}_\mathcal{O}}
d_{\lambda_4(\widetilde{f}_{\mathcal{O}})}/N_{\lambda_4(\widetilde{f}_{\mathcal{O}})}^2
=16m_{f_\mathcal{O}}(d_{\widetilde{f}_{\mathcal{O}}}/16)/N_{\widetilde{f}_{\mathcal{O}}}^2
=m_{f_{\mathcal{O}}}m^{-1}_{f_{\mathcal{O}}}
=1,
\end{equation*}
as what we need. Moreover, in this case $4d_{f}/N_{f}^2=1/4$ is not an integer and
we have $ 2\nmid d_{f}/\widetilde{m}^3_{f}=d_{\lambda_4(\widetilde{f}_{\mathcal{O}})}$, i.e.,
$$
2\widetilde{m}^3_f \nmid d_{f}. 
$$

Now assume $\mathcal{O}$ is not Gorenstein and write
$$
\mathcal{O}=\mathbb{Z}+m\mathcal{O}',  
$$
where $m>1$ and $\mathcal{O}'$ is Gorenstein. 
If $\mathcal{O}'$ is odd and $2\mid m$, then by \eqref{gongshi4.01}, Remark~\ref{rem:fO0}, and~\eqref{gongshi1.2}, we get
$$
f=f_{\mathcal{O}^0}=(m/2)^2f_{S'^0}=(m/2)^2m_{f_{\mathcal{O}'}}\widetilde f_{\mathcal{O}'}.
$$
Hence by~\eqref{gongshi4.05} we get
\begin{equation*}
d_{f}/N_{f}^2
=(m/2)^2m_{f_{\mathcal{O}'}} d_{\widetilde f_{\mathcal{O}'}}/N_{\widetilde f_{\mathcal{O}'}}^2
=(m/2)^2,
\end{equation*}
a square integer. 
As for the remaining cases, we have by  ~\eqref{gongshi4.01},
$$
f=f_{\mathcal{O}^0}=m^2f_{\mathcal{O}'^0}.
$$
Hence we have 
$$
d_{f}/N_{f}^2
=m^2d_{f_{\mathcal{O}'^0}}/N_{f_{\mathcal{O}'^0}}^2,
$$
and the problem is reduced to the Gorenstein case as discussed above. 

Turn to the proof of sufficiency. So we assume that $16d_{f}/N_{f}^2$ is a square integer. Then there exist three possibilities: 
\begin{itemize}
\item[(1)]  $d_{f}/N_{f}^2$ is a square integer $m^2$, 
\item[(2)]  $4d_{f}/N_{f}^2$ is a square integer $m^2$, but $d_{f}/N_{f}^2$ is not, 
\item[(3)]  $16d_{f}/N_{f}^2$ is a square integer $m^2$, but $4d_{f}/N_{f}^2$ is not. 
\end{itemize}
Write
\begin{equation}\label{gongshi4.08}
f=\widetilde{m}_{f}f', 
\end{equation}
where $f'$ is primitive. 
Then we may take a Gorenstein order $\mathcal{O}'$ by Theorem~\ref{the:bijection} such that 
\begin{equation}\label{gongshi4.09}
\widetilde{f}_{\mathcal{O}'}=\psi_{f_{\mathcal{O}'}}=f'.
\end{equation}
Here we note that by~\eqref{gongshi1.2} and ~\eqref{gongshi4.05}, one has  
\begin{equation}\label{gongshi4.010}
f_{S'^0}=\left(N_{\widetilde{f}_{\mathcal{O}'}}^2/d_{\widetilde{f}_{\mathcal{O}'}}\right)\widetilde{f}_{\mathcal{O}'}
=\left(N_{{f'}}^2/d_{{f'}}\right){f'},
\end{equation}
if $\mathcal{O}'$ is odd, where $S'=\mathbb{Z}+2\mathcal{O}'$; and 
\begin{equation}\label{gongshi4.011}
f_{\mathcal{O}'^0}
=\left(N_{\widetilde{f}_{\mathcal{O}'}}^2/4d_{\widetilde{f}_{\mathcal{O}'}}\right)\widetilde{f}_{\mathcal{O}'}
=\left(N_{{f'}}^2/4d_{{f'}}\right){f'},
\end{equation}
if $\mathcal{O}'$ is even. 
We first consider the case of 
$$
d_{f}/N_{f}^2=m^2.
$$ 
Put
$$
\mathcal{O}=\mathbb{Z}+2m\mathcal{O}'.
$$
If $\mathcal{O}'$ is even, then by Remark~\ref{rem:fO0},~\eqref{gongshi4.08} and~\eqref{gongshi4.011} we have 
\begin{equation*}
f_{{\mathcal{O}}^0}=4m^2f_{{\mathcal{O}}'^0}
=4(d_{f}/N_{f}^2)\left(N_{{f}'}^2/4d_{{f}'}\right){f}'
=\widetilde{m}_ff'
=f. 
\end{equation*}
If $\mathcal{O}'$ is odd, then by Remark~\ref{rem:fO0},~\eqref{gongshi4.08} and~\eqref{gongshi4.010} we have 
\begin{align*}
f_{{\mathcal{O}}^0}
=m^2f_{S'^0}
=(d_{f}/N_{f}^2)\left(N_{f'}^2/d_{f'}\right)f'
=\widetilde{m}_{f}f'
=f. 
\end{align*}

For the case in which 
\begin{equation}\label{gongshi4.012}
4d_{f}/N_{f}^2=m^2 
\end{equation}
for a positive integer $m$, but $d_{f}/N_{f}^2$ does not, we have clearly $2\nmid m$. Let
$$
h=v_2(d_{f'}),\,\,\,g=v_2(N_{f'}).
$$
Then we have $2+h-2g\leq0$ by ~\eqref{gongshi4.012}. 
Hence in view of ${f_{\mathcal{O}'}}=\psi_{f'}$ we see that $v_2(d_{f_{\mathcal{O}'}})\ne v_2(N_{f_{\mathcal{O}'}})-2$; whence by Theorem~\ref{the:odd} we find that the Gorenstein order $\mathcal{O}'$ defined by~\eqref{gongshi4.09} is even. Put 
$$
\mathcal{O}=\mathbb{Z}+m\mathcal{O}'.
$$ 
Then by Remark~\ref{rem:fO0},~\eqref{gongshi4.011}, and~\eqref{gongshi4.012}, we have
\begin{equation*}
f_{\mathcal{O}^0}=m^2f_{\mathcal{O}'^0}
=\left(4d_{f}/N_{f}^2\right)\left(N_{f'}^2/4d_{f'}\right)f'
=\widetilde{m}_{f}f'
=f.
\end{equation*}

For the case in which 
\begin{equation}\label{gongshi4.013}
16d_{f}/N_{f}^2=m^2 
\end{equation}
for a positive integer $m$, but $4d_{f}/N_{f}^2$ does not, then $2\nmid m$. 
Again, by assumption we have $2\nmid d_{f'}$, hence $h=0$, and whence $g=2$.
Thus by the second statement of Lemma~\ref{Watsontrans111}, we may take a primitive form $f''$ from $f'$ such that 
$$
f''=\phi^{-1}_2(f')
$$
with $v_2(N_{f''})=2$ and $v_2(d_{f''})=4$. Then by Theorem~\ref{the:bijection}, we can take a Gorenstein order $\mathcal{O}'$ such that
\begin{equation*}
\widetilde{f}_{\mathcal{O}'}=\psi_{f_{\mathcal{O}'}}=f''.
\end{equation*}
Also, by definition we have $v_2(m_{\psi_{f''}})=v_2(m_{f'})=0$. 
Hence $2\nmid m_{\psi_{f''}}=m_{f_{\mathcal O}'}$, and whence $\mathcal{O}'$ is odd by Theorem~\ref{the:odd}. 
Again, by Lemma~\ref{Watsontrans111} we have $N_{f''}=N_{f'}$ and $d_{f''}=16d_{f'}$.
Put 
$$
\mathcal{O}=\mathbb{Z}+m\mathcal{O}'.
$$
Then by Remark~\ref{rem:fO0} and Proposition ~\ref{prop:Watsontrans} we have
$$
f_{\mathcal{O}^0}
=m^2f_{\mathcal{O}'^0}
=m^2\phi_2(f_{S'^0}).
$$
Again, similar to ~\eqref{gongshi4.010} we have 
\begin{equation*}
f_{S'^0}=\left(N_{\widetilde{f}_{\mathcal{O}'}}^2/d_{\widetilde{f}_{\mathcal{O}'}}\right)\widetilde{f}_{\mathcal{O}'}
=\left(N_{f''}^2/d_{f''}\right)f''.
\end{equation*}
Thus by ~\eqref{gongshi4.013} we have
\begin{equation*}
f_{\mathcal{O}^0}=m^2\phi_2(f_{S'^0})
= \left(16d_{f}/N^2_{f}\right)\left(N^2_{f''}/d_{f''}\right)\phi_2(f'')  
=(d_{f}/N_{f}^2)(N_{f'}^2/d_{f'})f'  \\
=\widetilde{m}_{f}f'
=f.
\end{equation*}
The proof of the corollary is complete. 
\end{proof}

\begin{proof}[Proof of Corollary~{\upshape\ref{cor:unitgroup}}]
If $\mathcal{O}$ is Gorenstein, then by Theorem~\ref{the:bijection} we have  
\begin{equation*}
\rho_{\mathcal{O}}(1,r)=R_{\widetilde{f}_{\mathcal{O}}}\left(\frac{4-r^2}{m_{f_{\mathcal{O}}}}\right). 
\end{equation*}
Hence when $m_{f_{\mathcal{O}}}\geq5$, it is clear that $R_{\widetilde{f}_{\mathcal{O}}}\left(\frac{4-r^2}{m_{f_{\mathcal{O}}}}\right)>0$ if and only if $r^2=4$.
Again, for any $\alpha\in\mathcal{O}$, it is easy to see that $\n(\alpha)=1$ and $\tr(\alpha)=\pm2$ if and only if $\alpha=\pm1$. This proves
\begin{equation*}
\mathcal{O}^{\times}=\{\pm1\}.
\end{equation*}
If $\mathcal{O}$ is not Gorenstein, then we may write 
\begin{equation*}
\mathcal{O}=\mathbb{Z}+m\operatorname{Gor}(\mathcal{O}),
\end{equation*}
where $1<m\in \mathbb{Z}$. Hence
\begin{equation}\label{gongshi4.12}
\rho_{\mathcal{O}}(n,r)=R_{\widetilde{f}_{\operatorname{Gor}(\mathcal{O})}}\left(\frac{4n-r^2}{m^2m_{f_{\operatorname{Gor}(\mathcal{O})}}}\right).
\end{equation}
Now, if $m^2m_{f_{\operatorname{Gor}(\mathcal{O})}}\geq5$, then the desired result follows by the same argument as in the above case.
Hence we are led to consider the case of $m^2m_{f_{\operatorname{Gor}(\mathcal{O})}}\leq4$. 
Thus the only possibility in this case is $m=2$ and $m_{f_{\operatorname{Gor}(\mathcal{O})}}=1$.
Now suppose $\alpha$ is an element of $\mathcal{O}$ with $n(\alpha)=1$, then we have $4\cdot 1-\tr(\alpha)^2\geq 0$. 
Again, since $m=2$, we have $2\mid\tr(\alpha)$. Whence there exist two possibilities:
$$
\tr(\alpha)=0\,\,\,\text{or}\,\, \tr(\alpha)=\pm2.
$$
If $\tr(\alpha)=0$, then \eqref{gongshi4.12} reduces to 
\begin{equation*}
\rho_{\mathcal{O}}(1,0)=R_{\widetilde{f}_{\operatorname{Gor}(\mathcal{O})}}\left(1\right),
\end{equation*}
which is $0$ by Lemma~\ref{lem:rep}, since $\operatorname{Gor}(\mathcal{O})$ is odd due to $m_{f_{\operatorname{Gor}(\mathcal{O})}}=1$. This shows that the only possibility for an element $\alpha$ of $\mathcal{O}$ with $n(\alpha)=1$ is the case of 
$\tr(\alpha)=\pm2$, which then implies $\alpha=\pm1$. Therefore 
\begin{equation*}
\mathcal{O}^{\times}=\{\pm1\}.
\end{equation*}

Next, let $\mathcal{O}$ be an order such that $m_{f_{\mathcal{O}}}\geq5$ or it is not Gorenstein.
For each right ideal class $[I] \in \operatorname{Cls}(\mathcal{O})$, let $\mathcal{O}_{\operatorname{L}}(I)$ denote its left order. 
Since $\mathcal{O}_{\operatorname{L}}(I)$ is locally isomorphic to $\mathcal{O}$,  thus $\mathcal{O}_{\operatorname{L}}(I)$ satisfies the same condition as $\mathcal{O}$, i.e., $m_{f_{\mathcal{O}_{\operatorname{L}}(I)}}\geq5$ or $\mathcal{O}_{\operatorname{L}}(I)$ is not Gorenstein. 
Hence $\mathcal{O}_{\operatorname{L}}(I)^{\times} = \{\pm1\}$. Thus $w_I = |\mathcal{O}_{\operatorname{L}}(I)^{\times}/\{\pm1\}|=1$ for every class $[I]$. 
Therefore 
\[
\operatorname{mass}(\operatorname{Cls} (\mathcal{O})) = \sum_{[I]\in \operatorname{Cls}(\mathcal{O})} \frac{1}{w_I} = \sum_{[I]\in \operatorname{Cls}(\mathcal{O})} 1 = h(\mathcal{O}),
\]
as claimed in the corollary.
\end{proof}

\begin{proof}[Proof of Corollary~{\upshape\ref{cor:Eichlersym}}]
By~\cite[24.3.3]{Voi21}, Eichler symbol is in fact determined by the Kronecker symbol of the discriminant quadratic form $\tr(\alpha)^2-4\n(\alpha)$.
More accurately, we have
\begin{itemize}
\item $\left(\frac{\mathcal{O}}{p}\right)=*$ if and only if $\left\{\left(\frac{\tr(\alpha)^2-4\n(\alpha)}{p}\right):\,\alpha\in\mathcal{O}\right\}=\{-1, 0, 1\}$;
\item $\left(\frac{\mathcal{O}}{p}\right)=1$ (resp. $-1$ or $0$) if and only if $\left\{\left(\frac{\tr(\alpha)^2-4\n(\alpha)}{p}\right):\,\alpha\in\mathcal{O}\right\}=\{0, 1\}$ (resp. $\{0,-1\}$ or $\{0\}$).
\end{itemize}

Let $\mathcal{O}$ be a Gorenstein order.
Let $f_\mathcal{O}=(a,b,c,r,s,t)$ and $\mathcal{O}=C_0(f_\mathcal{O})=\mathbb{Z}+\mathbb{Z}e_1+\mathbb{Z}e_2+\mathbb{Z}e_3$.
Let $\alpha=t'+xe_1+ye_2+ze_3$, then
\begin{equation*}
\Delta(\alpha)=\tr(\alpha)^2-4\n(\alpha)=-(4bc-r^2)x^2-(4ac-s^2)y^2-(4ab-t^2)z^2-(2st-4ar)yz-(2rt-4bs)xz-(2rs-4ct)xy.
\end{equation*}
Hence $\Delta(\alpha)=-m_{f_\mathcal{O}}\widetilde{f}_\mathcal{O}$. 
Let $p\neq2$ be a prime. 
If $p\mid m_{f_\mathcal{O}}$, then we have $p\mid \Delta(\alpha)=\tr(\alpha)^2-4\n(\alpha)$. Hence by  definition, we have
\begin{equation*}
\left(\frac{\mathcal{O}}{p}\right)=0.
\end{equation*}
If $p\nmid m_{f_\mathcal{O}}$ and $p\mid m_{\widetilde{f}_\mathcal{O}}$, let $\widetilde{f}_\mathcal{O}=(A,B,C,R,S,T)$ and without loss generality assume $p\nmid C$ (see \cite[Proposition 2]{Leh92}).
Since $p\mid m_{\widetilde{f}_\mathcal{O}}$, then
\begin{equation*}
4AC\equiv S^2,\quad
4BC\equiv R^2,\quad
4CT\equiv2RS\pmod p.
\end{equation*}
Also since $p\nmid 2C$, then
\begin{equation*}
A\equiv\frac{S^2}{4C},\quad
B\equiv\frac{R^2}{4C},\quad
T\equiv\frac{RS}{2C}\pmod p.
\end{equation*}
Hence
\begin{equation*}
\widetilde{f}_\mathcal{O}\equiv C\left(z+\frac{S}{2C}x+\frac{R}{2C}y\right)^2 \pmod p.
\end{equation*}
We have $\left(\frac{\Delta(\alpha)}p\right)=0$ if $p\mid \widetilde{f}_\mathcal{O}$, otherwise
\begin{equation*}
\left(\frac{\Delta(\alpha)}p\right)
=\left(\frac{-m_{f_\mathcal{O}}}{p}\right)\left(\frac Cp\right).
\end{equation*}
By definition of genus symbol (see~\cite[Section 3]{Leh92}), we have $\left(\frac Cp\right)=\left(\frac{\widetilde{f}_\mathcal{O}}{p}\right)$.
Hence by the definition, we have
\begin{equation*}
\left(\frac{\mathcal{O}}{p}\right)=\left(\frac{-m_{f_\mathcal{O}}}{p}\right)\left(\frac{\widetilde{f}_\mathcal{O}}{p}\right).
\end{equation*}
If $p\nmid m_{f_\mathcal{O}}m_{\widetilde{f}_\mathcal{O}}=4N_{\widetilde{f}_\mathcal{O}}$, by~\cite[Theorem 2]{Leh92} we have $p\nmid d_{f_\mathcal{O}}=\discrd(\mathcal{O})$.
Hence by the definition we have $\mathcal{O}_p\simeq M_2(\mathbb{Z}_p)$ and 
\begin{equation*}
\left(\frac{\mathcal{O}}{p}\right)=*.
\end{equation*}

Let $p=2$. If $\mathcal{O}$ is even, i.e., $2\mid\tr(\alpha)$, then we have $2\mid\Delta(\alpha)$. Hence
\begin{equation*}
\left(\frac{\mathcal{O}}{2}\right)=0.
\end{equation*}
If $\mathcal{O}$ is odd and $2\nmid d_{f_\mathcal{O}}=\discrd(\mathcal{O})$, then by the definition we have $\mathcal{O}_2\simeq M_2(\mathbb{Z}_2)$ and 
\begin{equation*}
\left(\frac{\mathcal{O}}{2}\right)=*.
\end{equation*}
If $\mathcal{O}$ is odd and $2\mid d_{f_\mathcal{O}}=\discrd(\mathcal{O})$, 
since $m_{\widetilde{f}_\mathcal{O}}=16d_{f_\mathcal{O}}/m^2_{f_\mathcal{O}}$
and $2\nmid m_{f_\mathcal{O}}$, then we have $32\mid m_{\widetilde{f}_\mathcal{O}}$.
Similarly, one can also check that
\begin{equation*}
\widetilde{f}_\mathcal{O}\sim_2 Cz^2+8q(x,y)
\end{equation*}
where $q(x,y)$ is a binary quadratic form. We have $\left(\frac{\Delta(\alpha)}2\right)=0$ if $2\mid z$, otherwise
\begin{equation*}
\left(\frac{\Delta(\alpha)}2\right)
=\left(\frac{-m_{f_\mathcal{O}}}{2}\right)\left(\frac C2\right).
\end{equation*}
By the definition of genus symbol (see~\cite[Section 3]{Leh92}), we have $\left(\frac C2\right)=\left(\frac{\widetilde{f}_\mathcal{O}}{8}\right)$.
Hence by the definition, we have
\begin{equation*}
\left(\frac{\mathcal{O}}{2}\right)=\left(\frac{-m_{f_\mathcal{O}}}{2}\right)\left(\frac{\widetilde{f}_\mathcal{O}}{8}\right).
\end{equation*}
The proof of the corollary is complete.
\end{proof}

\begin{proof}[Proof of Corollary~{\upshape\ref{cor:automorphism}}]
For a given order $\mathcal{O}$, by Definition ~\ref{def:tildef} we can construct a quadratic form $\widetilde f_{\mathcal{O}}$,
and then for this $\widetilde f_{\mathcal{O}}$, by even Clifford algebra $C_0$ we have up to isomorphism an order $\mathcal{O}'$ such that $f_{\mathcal{O}'}=\widetilde f_{\mathcal{O}}$.  
Hence we can define a map, denoted by $\psi$, up to isomorphism, from the set of all orders to itself as follows:
\begin{equation}
\psi\!:\,\mathcal{O}\mapsto \mathcal{O}'\,\,\text{with}\,\,f_{\mathcal{O}'}=\widetilde f_{\mathcal{O}}. 
\end{equation}
It is clear that $\psi$ is bijective up to isomorphism, and by Theorem ~\ref{the:bijection} we have
$$
\widetilde{f}_{\psi(\mathcal{O})}=\psi_{f_{\psi{(\mathcal{O})}}}=\psi_{f_{{\mathcal{O}'}}}=\psi_{\widetilde f_{{\mathcal{O}}}}=f_{{\mathcal{O}}}.
$$ 
Here and in what follows, it should be not confused by the common use of $\psi$ for the maps both on orders and on quadratic forms from the context. 

Next, suppose $\mathcal{O}$ is odd and Gorenstein. Then by Theorem~\ref{the:odd}, we have $2\nmid m_{f_{\mathcal{O}}}$. Again, for any primitive form $f$ we have $m_{\psi_f}=N^2_f/d_f=4N_f/m_f$, hence
$$
m_{\psi_f}m_f=4N_f.
$$ 
This clearly implies $2\mid m_{\psi_{f_\mathcal{O}}}$ if we take $f=f_{\mathcal{O}}$ because $2\nmid m_{f_{\mathcal{O}}}$. Thus $\psi(\mathcal{O})$ is even by Theorem~\ref{the:odd} again.  
For a general order $\mathcal{O}=\mathbb{Z}+m\operatorname{Gor}(\mathcal{O})$ of odd conductor $m$, the desired result follows from $\psi(\mathcal{O})\simeq\mathbb{Z}+m\psi(\operatorname{Gor}(\mathcal{O}))$, which is a direct consequence of the definition of $\psi$.

Furthermore, let $p\geq 2$ be a given prime, and $q$ denote primes with $q\ne p$. 
Let $\mathcal{O}'$ and $\mathcal{O}''$ be two Gorenstein orders such that $\mathcal{O}'_{p}\cong\psi(\mathcal{O}'')_{p}$, $\mathcal{O}'_{q}\cong\mathcal{O}''_{q}$, where $\mathcal{O}'_{p}$ and $\mathcal{O}''_{p}$ denote respectively the localizations of $\mathcal{O}'$ and $\mathcal{O}''$ at primes $p\geq 2$. 
By Remark~\ref{rem:phi}, Lehman's
correspondence leaves the local unit-similarity class
unchanged away from $p$, and replaces it by its
reciprocal at $p$.
Unit similarity does not change the isomorphism class
of the associated local quaternion order.
Using $\mathcal M=\psi\circ C_0^{-1}$, we therefore obtain
\[
 \bigl(\mathcal M^{-1}(\phi_p(\mathcal M(\mathcal O)))\bigr)_q
 \simeq\mathcal O_q
 \qquad(q\ne p),
\]
and
\[
 \bigl(\mathcal M^{-1}(\phi_p(\mathcal M(\mathcal O)))\bigr)_p
 \simeq\psi(\mathcal O)_p.
\]
The genus compatibility in Theorem~\ref{the:bijection} and
Remark~\ref{rem:phi} therefore show that
$$
\varphi_p\!:\,\mathcal{O}\xrightarrow[]{\mathcal M} \widetilde{f}_{\mathcal{O}} \xrightarrow[]{\phi_p} \phi_p(\widetilde{f}_{\mathcal{O}}) \xrightarrow[]{\mathcal M^{-1}} \mathcal{O}'=\mathcal M^{-1}(\phi_p(\widetilde{f}_{\mathcal{O}}))=\mathcal M^{-1}(\phi_p(\mathcal M({\mathcal{O}}))),
$$
is a bijection between the indicated type sets.
For general orders, write
$\mathcal O''=\mathbb Z+m\operatorname{Gor}(\mathcal O'')$.
Reciprocity preserves the Gorenstein conductor, so the local
hypotheses imply that $\mathcal O'$ has the same conductor $m$.
For each $[\mathcal O]\in\operatorname{Typ}(\mathcal O'')$,
define
\[
 \varphi_p(\mathcal O)
 =\mathbb Z+m\varphi_p(\operatorname{Gor}(\mathcal O)),
\]
using the Gorenstein correspondence on the right.
Uniqueness of Gorenstein saturation makes this well-defined
on isomorphism classes. The local identities already proved,
together with
\[
 \psi(\mathbb Z+m\mathcal O)
 \simeq\mathbb Z+m\psi(\mathcal O),
\]
show that its image lies in $\operatorname{Typ}(\mathcal O')$.
Applying the inverse Gorenstein correspondence at the same
conductor gives the inverse map. This proves the assertion
for arbitrary orders.
\end{proof}

\begin{proof}[Proof of Corollary~{\upshape\ref{cor:trace}}]
Let $I_1,\ldots,I_h$ represent the locally principal fractional
right $\mathcal O$-ideal classes, and put
\[
 \nu_\mu
 =
 \#\{i:\mathcal O_{\operatorname L}(I_i)
           \simeq\mathcal O_\mu\},
 \qquad 1\leq\mu\leq t.
\]
By \cite[Proposition~18.5.10 and (18.5.8)]{Voi21},
\[
 \nu_\mu
 \bigl[N_{Q^\times}(\mathcal O_\mu):
       \mathbb Q^\times\mathcal O_\mu^\times\bigr]
 =
 |\operatorname{Pic}(\mathcal O_\mu)|.
\]
Since
$N_{Q^\times}(\mathcal O_\mu)/\mathbb Q^\times
\simeq\operatorname{Aut}(\mathcal O_\mu)$
and
$\mathcal O_\mu^\times\cap\mathbb Q^\times=\{\pm1\}$,
we have
\[
 \bigl[N_{Q^\times}(\mathcal O_\mu):
       \mathbb Q^\times\mathcal O_\mu^\times\bigr]
 =
 \frac{2|\operatorname{Aut}(\mathcal O_\mu)|}
      {|\mathcal O_\mu^\times|}.
\]
Also,
$|\operatorname{Pic}(\mathcal O_\mu)|
=|\operatorname{Pic}(\mathcal O)|$
by \cite[(18.4.7)]{Voi21}. Consequently,
\[
 \frac{\nu_\mu}{|\mathcal O_\mu^\times|}
 =
 \frac{|\operatorname{Pic}(\mathcal O)|}
      {2|\operatorname{Aut}(\mathcal O_\mu)|}.
\]

For $n\geq1$, our normalization of the Brandt matrices gives
\[
 \bigl(T_{\mathcal O}(n)\bigr)_{ii}
 =
 \frac{\#\{\alpha\in\mathcal O_{\operatorname L}(I_i):
                 n(\alpha)=n\}}
      {|\mathcal O_{\operatorname L}(I_i)^\times|}
 =
 \sum_{\substack{r\in\mathbb Z\\r^2\leq4n}}
 \frac{\rho_{\mathcal O_{\operatorname L}(I_i)}(n,r)}
      {|\mathcal O_{\operatorname L}(I_i)^\times|}.
\]
Grouping the ideal classes according to the isomorphism classes
of their left orders, and then applying (4) in Theorem~\ref{the:bijection}, we obtain
\begin{align*}
 \tr(T_{\mathcal O}(n))
 &=
 \sum_{\substack{r\in\mathbb Z\\r^2\leq4n}}
 \sum_{\mu=1}^t
 \frac{\nu_\mu\,\rho_{\mathcal O_\mu}(n,r)}
      {|\mathcal O_\mu^\times|}\\
 &=
 \frac{|\operatorname{Pic}(\mathcal O)|}{2}
 \sum_{\substack{r\in\mathbb Z\\r^2\leq4n}}
 \sum_{\mu=1}^t
 \frac{\rho_{\mathcal O_\mu}(n,r)}
      {|\operatorname{Aut}(\mathcal O_\mu)|}\\
 &=
 |\operatorname{Pic}(\mathcal O)|
 \sum_{\substack{r\in\mathbb Z\\r^2\leq4n}}
 R_{\operatorname{gen}(\widetilde f_{\mathcal O})}
 \left(\frac{m(4n-r^2)}{m_{f_{\mathcal O}}}\right).
\end{align*}
This completes proof of the corollary.
\end{proof}

\section{Applications: Theta series associated with orders}

Let $\mathcal{O}\subset Q$ be a definite quaternion order. Let $\tau$ be a variable in the upper half plane $\mathbb{H}$ and $z$ a complex variable in $\mathbb{C}$. Let $q\!:\,=\exp(2\pi i\tau)$, $\zeta\!:\,=\exp(2\pi i z)$ as convention. We will consider the following theta series associated to $\mathcal{O}$ in this section. 
\begin{align}
&\theta_\mathcal{O}(\tau,z)\!:\,=\sum_{\alpha\in\mathcal{O}}q^{\n(\alpha)}\zeta^{\tr(\alpha)}
=\sum_{\substack{n,r\in\mathbb{Z}\\4n-r^2\geq0}}\rho_\mathcal{O}(n,r)q^n\zeta^r
=\sum_{\substack{n,r\in\mathbb{Z}\\4n-r^2\geq0}}c(n,r)q^n\zeta^r,\label{gongshi5.1}\\
&\theta_\mathcal{O}(\tau)\!:\,=\sum_{\alpha\in\mathcal{O}}q^{\n(\alpha)}=\sum_{n\geq0}a(n)q^n,\label{gongshi5.2}\\
&\theta_{\mathcal{O}^0}(\tau)\!:\,=\sum_{\alpha\in\mathcal{O}^0}q^{\n(\alpha)}
=\sum_{n\geq0}R_{f_{\mathcal{O}^0}}(n)q^n=\sum_{n\geq0}a_0(n)q^n,\label{gongshi5.3}\\
&\theta_{S^0}(\tau)\!:\,=\sum_{\alpha\in S^0}q^{\n(\alpha)}=\sum_{n\geq0}R_{f_{S^0}}(n)q^n=\sum_{n\geq0}a_1(n)q^n. \label{gongshi5.4}
\end{align}
Here $S=\mathbb{Z}+2\mathcal{O}$, and 
\begin{equation}\label{gongshi5.5}
c(n,r)=\rho_\mathcal{O}(n,r),\,\,a(n)=\sum_{r^2\leq 4n}c(n,r),\,\,a_0(n)=R_{f_{\mathcal{O}^0}}(n)=c(n,0),\,\,a_1(n)=R_{f_{S^0}}(n). 
\end{equation}
Clearly, by Lemma~\ref{lem:rep}, if $\mathcal{O}$ is odd, then for any non-negative integer $D\in\mathbb{Z}$, 
\begin{equation}\label{gongshi5.6}
a_1(D)=
\begin{cases}
\rho_{\mathcal{O}}(n,r) & \text{if}\,\,\,D=4n-r^2\geq 0\,\, \text{for some} \,\,n,\,r\in\mathbb{Z}, \\
0 & \text{otherwise};
\end{cases}
\end{equation} 
and if $\mathcal{O}$ is even, then
\begin{equation}\label{gongshi5.7}
a_0(D)=
\rho_{\mathcal{O}}(n,r),\,\,\,\text{if}\,\,D=n-r^2/4\geq 0\,\, \text{for some}\,\, n,r\in\mathbb{Z}\,\,\text{and}\,\,2\mid r.
\end{equation} 

For ease of notations, in the sequel, we will also use $f$ to denote modular forms and Jacobi forms besides quadratic forms. It will be clear for the exact meaning of it when it occurs in different places from the context. The operators $U_m, V_m$ on modular forms and $U(m), V(m)$ on Jacobi forms $f$ will be used freely as convention. 

\subsection{Connections of theta series}
The following proposition gives some interesting connections as for the theta series defined by \eqref{gongshi5.1}, \eqref{gongshi5.3} and \eqref{gongshi5.4}. 

\begin{prop}\label{exam:theta1}
Let $k\geq2$ be an integer, $\mathcal{O}$ be an odd order in $\operatorname{Typ}_i(2^{k})\,\,(i=1,2)$, $\mathcal{O}'=\sigma_2^{-1}\circ {\mathcal M}_+(\mathcal{O})$, 
$S=\mathbb{Z}+2\mathcal{O}$, $S'=\mathbb{Z}+2\mathcal{O}'$.
Then we have
\begin{equation*}
U_4\theta_{S^0}(\tau)=\theta_{\mathcal{O}^0}(\tau)=\theta_{S'^0}(\tau),\,\,\,
U_4\theta_{\mathcal{O}^0}(\tau)=\theta_{\mathcal{O}'^0}(\tau).
\end{equation*}
\end{prop}

\begin{proof}
Since $k\geq2$ and $\mathcal{O}\in\operatorname{Typ}_i(2^{k})$, we have $\mathcal{O}'\in\operatorname{Typ}_i(2^{k-2})$ by Theorem ~\ref{chains1}, in particular $\mathcal{O}'$ is odd. Again, since $\mathcal{O}'=\sigma_2^{-1}\circ {\mathcal M}_+(\mathcal{O})$, we have $S'=\mathbb{Z}+2\mathcal{O}'=\mathcal{O}^+$ by definition. Hence $\mathcal{O}^0={(\mathcal{O}^+)}^0={S'}^0$, which clearly implies
$$
\theta_{\mathcal{O}^0}(\tau)=\theta_{S'^0}(\tau).
$$ 
Recall that for $m\geq 1$, the operator $U_m$ on modular forms $f(\tau)=\sum_{n\geq 0} a_nq^n$ is defined by 
\begin{equation}
U_mf(\tau)\!:\,=\sum_{n\geq 0} a_{nm}q^n.
\end{equation}
Hence by \eqref{gongshi5.3}, \eqref{gongshi5.4}, and in view of $a_0(n)=c(n,0)=a_1(4n)$ due to ~\eqref{gongshi5.5} and ~\eqref{gongshi5.6}, we get
\begin{equation*}
U_4\theta_{S^0}(\tau)=\theta_{\mathcal{O}^0}(\tau).
\end{equation*}
This together with $\theta_{\mathcal{O}^0}(\tau)=\theta_{S'^0}(\tau)$ yields
$$
U_4\theta_{\mathcal{O}^0}(\tau)=U_4\theta_{S'^0}(\tau)=\theta_{{\mathcal{O}'}^0}(\tau).
$$
Gathering together the above the proof is complete. 
\end{proof}

\begin{prop}\label{prop:theta}
Let $\mathcal{O}$ be a Gorenstein order of discriminant $d$, $p$ be a prime, and $\mathcal{O}'=\varphi_p(\mathcal{O})$ be defined as in Corollary {\rm ~\ref{cor:automorphism}}. Let $S=\mathbb{Z}+2\mathcal{O}$, $S'=\mathbb{Z}+2\mathcal{O}'$. If $\mathcal{O}$ is odd and $2\ne p\parallel d$, then 
\begin{equation*}
\theta_{{S'}^0}(\tau)=V_{p}U_{p}\theta_{S^0}(\tau). 
\end{equation*}
If $\mathcal{O}$ is odd and $p=2\parallel d$, then
\begin{equation*}
\theta_{{\mathcal{O}'}^0}(\tau)=V_{8}U_{8}\theta_{S^0}(\tau).
\end{equation*}
If $\mathcal{O}$ is even, then for any $2\leq p\parallel d$ we have
\begin{equation*}
\theta_{\mathcal{O}'^0}(\tau)=V_{p}U_{p}\theta_{\mathcal{O}^0}(\tau).
\end{equation*}
Moreover, if $2\nmid d$, and $p=2$, then  
\begin{equation*}
    \theta_{{\mathcal{O}'}^0}(\tau)=V_{4}U_{4}\theta_{S^0}(\tau).
\end{equation*}
\end{prop}

\begin{proof}
By the definition of $\mathcal{O}'$ as in Corollary ~\ref{cor:automorphism} we have  
\begin{equation}\label{gongshi5.12}
\widetilde{f}_\mathcal{O}=\phi_p(\widetilde{f}_{\mathcal{O}'}).
\end{equation}
Thus, for odd prime $p\parallel d$, by the definition of Lehman's correspondence in section ~\ref{section3.2} we have 
\begin{equation}\label{gongshi5.13}
pm_{f_\mathcal{O}}=m_{f_{\mathcal{O}'}},\,\,(p,\,m_{f_\mathcal{O}})=1.
\end{equation}
Hence if $\mathcal{O}$ is odd then $\mathcal{O}'$ is also odd; and by \cite[Proposition 6.6]{LZ24}, we have for any integer $m$, 
\begin{equation*}
R_{\widetilde{f}_\mathcal{O}}(pm)=R_{\widetilde{f}_{\mathcal{O}'}}(m).
\end{equation*}
This together with Theorem~\ref{the:bijection} and \eqref{gongshi5.13} implies that, for any non-negative integer $D=4n-r^2$ with $n,r\in\mathbb{Z}$, if $p\mid D$, then
\begin{equation*}
\rho_{\mathcal{O}'}(n,r)=
R_{\widetilde{f}_{\mathcal{O}'}}\left(\frac{D}{m_{f_{\mathcal{O}'}}}\right)=
R_{\widetilde{f}_\mathcal{O}}\left(\frac{pD}{m_{f_{\mathcal{O}'}}}\right)
=R_{\widetilde{f}_\mathcal{O}}\left(\frac{D}{m_{f_\mathcal{O}}}\right)=
\rho_\mathcal{O}(n,r);
\end{equation*}
and if $p\nmid D$, then $\rho_{\mathcal{O}'}(n,r)=0$.
Writing  
\begin{equation*}
\theta_{S^0}(\tau)=\sum_{n\geq0}a_1(n)q^n,~\theta_{S'^0}(\tau)=\sum_{n\geq0}a'_1(n)q^n,
\end{equation*}
then by ~\eqref{gongshi5.6} we see that, if $\mathcal{O}$ is odd, for any $D=4n-r^2\geq 0$ with $p\mid D$, one has 
$$
a_1(D)=a'_1(D),
$$
and otherwise, one has $a'_1(D)=0$. Thus we have, again by ~\eqref{gongshi5.6}, 
$$
V_{p}U_{p}\theta_{S^0}(\tau)=\sum_{n\geq0}a_1(pn)q^{pn}=\sum_{n\geq0}a'_1(pn)q^{pn}=\sum_{n\geq0}a'_1(n)q^{n}=\theta_{S'^0}(\tau),
$$
as stated for the first equality of this proposition. 

If $\mathcal{O}$ is odd and $2\parallel d$, then by Lehman's correspondence we see that  
\begin{equation*}
32m_{f_\mathcal{O}}=m_{f_{\mathcal{O}'}},\,\,2\nmid m_{f_\mathcal{O}}, 
\end{equation*}
and $\mathcal{O}'$ is even.
Hence similar to above, for any integer $m$ we have
\begin{equation*}
R_{\widetilde{f}_\mathcal{O}}(8m)=R_{\widetilde{f}_{\mathcal{O}'}}(m).
\end{equation*}
Thus by Theorem~\ref{the:bijection} we find that, for any non-negative integer $D=n-r^2/4$ with $n,r\in\mathbb{Z}$, if $8\nmid D=n-r^2/4$, then $\rho_{\mathcal{O}'}(n,r)=0$; and if $8\mid D$, then
\begin{equation*}
\rho_{\mathcal{O}'}(n,r)=R_{\widetilde{f}_{\mathcal{O}'}}\left(\frac{4D}{m_{f_{\mathcal{O}'}}}\right)
=R_{\widetilde{f}_{\mathcal{O}'}}\left(\frac{4D}{32m_{f_\mathcal{O}}}\right)
=R_{\widetilde{f}_{\mathcal{O}'}}\left(\frac{D}{8m_{f_\mathcal{O}}}\right)
=R_{\widetilde{f}_\mathcal{O}}\left(\frac{D}{m_{f_\mathcal{O}}}\right).
\end{equation*}
Again, for $8\mid D$, it is clear that one can always write $D=4n_1-r_1^2$ with $n_1,r_1\in\mathbb{Z}$. Hence by Theorem~\ref{the:bijection} we also have
$R_{\widetilde{f}_\mathcal{O}}\left(\frac{D}{m_{f_\mathcal{O}}}\right)=\rho_{\mathcal{O}}(n_1,r_1).$ 
This together with the above identity implies
\begin{equation*}
\rho_{\mathcal{O}'}(n,r)=\rho_{\mathcal{O}}(n_1,r_1).
\end{equation*}
Hence, if we write   
\begin{equation*}
\theta_{\mathcal{O}^0}(\tau)=\sum_{n\geq0}a_0(n)q^{n},\,\,\theta_{\mathcal{O}'^0}(\tau)=\sum_{n\geq0}a_0'(n)q^{n}, 
\end{equation*}
then by ~\eqref{gongshi5.6} and ~\eqref{gongshi5.7} we see that for  $8\mid D$, one has
\begin{equation*}
a_0'(D)=\rho_{\mathcal{O}'}(n,r)=\rho_{\mathcal{O}}(n_1,r_1)=a_1(D);
\end{equation*}
and for $8\nmid D$, one has $a_0'(D)=0$. Thus 
\begin{equation*}
V_{8}U_{8}\theta_{S^0}(\tau)=V_{8}U_{8}\sum_{n\geq0}a_1(n)q^{n}=\sum_{n\geq0}a_1(8n)q^{8n}
=\sum_{n\geq0}a_0'(8n)q^{8n}=\sum_{n\geq0}a'_0(n)q^{n}=\theta_{\mathcal{O}'^0}(\tau),
\end{equation*}
as desired for the second identity of the proposition.

If $\mathcal{O}$ is even, then $4\mid d$. Hence from $p\parallel d$ we see that $p$ is odd; whence the third identity follows by using completely the same argument as that for the first one, only with $\mathcal{O}$ instead of $S$. 

Finally, if $2\nmid d$, then $2\nmid m_{f_\mathcal{O}}$, and so $\mathcal{O}$ is odd; and in view of \eqref{gongshi5.12} and Lehman's correspondence, we have 
\begin{equation*}
16m_{f_\mathcal{O}}=m_{f_{\mathcal{O}'}}. 
\end{equation*}
Hence for any integer $m$ we have
\begin{equation*}
R_{\widetilde{f}_\mathcal{O}}(4m)=R_{\widetilde{f}_{\mathcal{O}'}}(m).
\end{equation*}
Thus for any non-negative integer $D=n-r^2/4$ with $n,r\in\mathbb{Z}$, if $4\nmid D=n-r^2/4$, then $a_0'(D)=\rho_{\mathcal{O}'}(n,r)=0$; and if $4\mid D$, then we can write further $D=4n_1-r_1^2$ with $n_1,r_1\in\mathbb{Z}$, so that 
\begin{equation*}
a_0'(D)=\rho_{\mathcal{O}'}(n,r)=R_{\widetilde{f}_{\mathcal{O}'}}\left(\frac{4D}{m_{f_{\mathcal{O}'}}}\right)
=R_{\widetilde{f}_{\mathcal{O}'}}\left(\frac{4D}{16m_{f_\mathcal{O}}}\right)
=R_{\widetilde{f}_{\mathcal{O}'}}\left(\frac{D}{4m_{f_\mathcal{O}}}\right)
=R_{\widetilde{f}_\mathcal{O}}\left(\frac{D}{m_{f_\mathcal{O}}}\right)=\rho_{\mathcal{O}}(n_1,r_1)=a_1(D).
\end{equation*}
Thus
$$
V_{4}U_{4}\theta_{S^0}(\tau)=\sum_{n\geq0}a_1(4n)q^{4n}=\sum_{n\geq0}a_0'(4n)q^{4n}=\theta_{{\mathcal{O}'}^0}(\tau), 
$$
which is the fourth identity of the proposition. 
\end{proof}

\subsection{The level of theta series}
The following theorem gives an explicit description for the level of the theta series associated with Gorenstein orders of discriminant $d$. 
\begin{theorem}\label{prop:mf}
Let $G\left(\frac{3}{2},N,\chi_{d}\right)$ denote the space of modular
forms of weight $3/2$, level $N$, associated with the character $\chi_{d}=\left(\frac{d}{\cdot}\right)$.
Let $\mathcal{O}$ be a Gorenstein order with discriminant $d$. If $\mathcal{O}$ is odd, then
\begin{equation*}
\theta_{S^0}(\tau)\in G\left(\frac{3}{2},4d, {\chi_0}\right), \,\,\, \theta_{\mathcal{O}^0}(\tau)\in 
\begin{cases}
G\left(\frac{3}{2},d,{\chi_0}\right) & \text{if $4\mid d$;}\\
G\left(\frac{3}{2},2d,{\chi_0}\right)& \text{if $2\parallel d$;}\\
G\left(\frac{3}{2},4d,{\chi_0}\right)& \text{if $2\nmid d$.}\\
\end{cases}
\end{equation*}
If $\mathcal{O}$ is even, then
\begin{equation*}
\theta_{\mathcal{O}^0}(\tau)\in G\left(\frac{3}{2},d,{\chi_0}\right).
\end{equation*}
Here $\chi_0$ denotes principal characters.  
\end{theorem}

\begin{proof}
Denote, for any positive definite ternary quadratic form $f$, 
\begin{equation}\label{gongshi5.8}
\theta_{f}(\tau)\!:\,=\sum_{n\geq0} R_f(n)q^n.
\end{equation}
It is known that (see, e.g. \cite[Theorem 10.1]{WP12}), if $f$ is of level $N_f$ and discriminant $d_f$, then 
$$
\theta_{f}(\tau)\in G\left(\frac{3}{2},N_f,\chi_{d_f}\right).
$$ 
For a positive integer $m\geq 1$, define the operator $V_m$ on modular forms $f(\tau)=\sum_{n\geq 0} a_nq^n$ by 
\begin{equation}\label{gongshi5.9}
V_mf(\tau)\!:\,=\sum_{n\geq 0} a_{n}q^{mn}.
\end{equation}
Now, if $\mathcal{O}$ is odd, then $f_{S^0}=m_{f_{\mathcal{O}}}\widetilde{f}_\mathcal{O}$ by~\eqref{gongshi1.2}.
Hence by ~\eqref{gongshi5.4}, ~\eqref{gongshi5.8} and ~\eqref{gongshi5.9} we have  
\begin{equation*}
\theta_{S^0}(\tau)=V_{m_{f_\mathcal{O}}}\theta_{\widetilde{f}_\mathcal{O}}(\tau).
\end{equation*}
Again, by the assumption $\operatorname{discrd}(\mathcal{O})=d$ and Lemma \ref{Lemma3.3} we have $d_{f_\mathcal{O}}=d$. By $\psi_{f_{\mathcal{O}}}=\widetilde{f}_{\mathcal{O}}$ and section ~\ref{section3.2} we have $N_{\widetilde{f}_\mathcal{O}}=N_{f_\mathcal{O}}$, $d_{\widetilde{f}_\mathcal{O}}=N_{f_\mathcal{O}}^3/4d_{f_\mathcal{O}}$. Hence from  $m_{f_\mathcal{O}}N_{f_\mathcal{O}}=4d_{f_\mathcal{O}}$ we get
$$
m_{f_\mathcal{O}}N_{\widetilde f_\mathcal{O}}=m_{f_\mathcal{O}}N_{f_\mathcal{O}}=4d_{f_\mathcal{O}}=4d,\,\,\,
m_{f_\mathcal{O}}d_{\widetilde{f}_\mathcal{O}}=m_{f_\mathcal{O}}N_{f_\mathcal{O}}^3/4d_{f_\mathcal{O}}=N_{f_\mathcal{O}}^2. 
$$
Thus from $\theta_{\widetilde{f}_\mathcal{O}}(\tau)\in G\left(\frac{3}{2},N_{\widetilde{f}_\mathcal{O}},\chi_{d_{\widetilde{f}_\mathcal{O}}}\right)$ we get
\begin{equation*}
\theta_{S^0}(\tau)\in G\left(\frac{3}{2},4d,{\chi_0}\right).
\end{equation*}

For $\theta_{{\mathcal{O}}^0}(\tau)$, if $\mathcal{O}$ is even (which implies $4\mid d$), the argument is completely the same as the above case except that we need use $m_{f_\mathcal{O}}/4$ instead of $m_{f_\mathcal{O}}$~\eqref{gongshi1.2}. 
If $\mathcal{O}$ is odd, then it is contained in some $\operatorname{Typ}_i(2^{k})$. We consider two cases according to $k\geq 2$ or not. If $k\geq 2$ then $[\mathcal{O}':\mathcal{O}]=4$, i.e., $\operatorname{discrd}(\mathcal{O})=4\operatorname{discrd}(\mathcal{O}')$ under the notations of Proposition ~\ref{exam:theta1}. Hence
\begin{equation*}
\theta_{\mathcal{O}^0}(\tau)=\theta_{{S'}^0}(\tau)\in G\left(\frac{3}{2},d,{\chi_0}\right)
\end{equation*}
as what we desire. If $k=0,1$, then the desired result follows directly from Lemma ~\ref{Watsontrans111}. 
\end{proof}

\begin{cor}\label{cor5.2}
Let $\mathcal{O}$ be a Gorenstein order of discriminant $d$. Then $\theta_\mathcal{O}(\tau,z)$ is a Jacobi form of weight $2$, index $1$ and level $d$ with trivial character. 
\end{cor}
\begin{proof}
The argument in the proof of
\cite[Proposition~4.1]{BSZ19} applies verbatim to the norm
lattice of $\mathcal O$ and yields the asserted Jacobi
transformation law. The Fourier expansion follows from
the definition of $\theta_{\mathcal O}(\tau,z)$.
\end{proof}

\begin{remark}
By~\cite[Theorem 5.4]{EZa85}, the relation with Kohnen's plus space can be described uniformly
for odd and even orders. Put $S=\mathbb{Z}+2\mathcal{O}$ and write
\[
H_{\mathcal{O}}(\tau)
:=\theta_{S^0}(\tau)
=\sum_{D\geq0}a_{\mathcal{O}}(D)q^D.
\]
Since
\[
S^0=\{2\alpha-\tr(\alpha):\alpha\in\mathcal{O}\}
\]
and
\[
\n(2\alpha-\tr(\alpha))
=4\n(\alpha)-\tr(\alpha)^2,
\]
we have $a_{\mathcal{O}}(D)=0$ unless
$D\equiv0,3\pmod4$. If $\mathcal{O}$ is even, then
$S^0=2\mathcal{O}^0$, so that
\[
H_{\mathcal{O}}(\tau)=\theta_{\mathcal{O}^0}(4\tau).
\]
Thus, in both cases, Theorem~\ref{prop:mf} gives
\[
H_{\mathcal{O}}\in
M_{3/2}^{+}(\Gamma_0(4d))
\]
with trivial character, where the plus subspace is defined
by the above Fourier coefficient condition.
\end{remark}

\subsection{Relations of traces of Brandt matrices}
The following proposition gives a description for the relations of the traces of Brandt matrices related to $\mathcal{O}$ and $\varphi_p(\mathcal{O})$.

\begin{prop}\label{prop:trace}
Let $\mathcal{O}$ be a Gorenstein order of discriminant $d$ and $T_{\mathcal{O}}(n)$ denote the $n$-th Brandt matrix related to $\mathcal{O}$. Let $\mathcal{O}'=\varphi_p(\mathcal{O})$, $S=\mathbb{Z}+2\mathcal{O}$, with $\varphi_p$ as in Corollary ~{\rm\ref{cor:automorphism}}. Then for positive integers $n,\,k\geq 1$, if $2<p\parallel d$, one has  
\begin{equation*}
\frac{\tr(T_\mathcal{O}(p^kn))}{|\operatorname{Pic}(\mathcal{O})|}-
\frac{\tr(T_{\mathcal{O}'}(p^kn))}{|\operatorname{Pic}(\mathcal{O}')|}
=\sum_{\substack{r^2\leq 4p^kn\\ p\nmid r}}R_{\operatorname{gen}(\widetilde{f}_{\mathcal{O}})}\left(\frac{4p^kn-r^2}{m_{f_{\mathcal{O}}}}\right)
;
\end{equation*}
if $2=p\parallel d$, one has 
\begin{equation*}
\frac{\tr(T_S(2^{k+2}n))}{|\operatorname{Pic}(S)|}-\frac{\tr(T_{\mathcal{O}'}(2^{k+2}n))}{|\operatorname{Pic}(\mathcal{O}')|}
=\sum_{\substack{r^2\leq 2^{k+4}n\\ 8\nmid r}}R_{\operatorname{gen}(\widetilde{f}_{S})}\left(\frac{2(2^{k+4}n-r^2)}{m_{f_{S}}}\right);
\end{equation*}
if $2\nmid d$ and $p=2$, one has   
\begin{equation*}
\frac{\tr(T_S(2^{k+1}n))}{|\operatorname{Pic}(S)|}-\frac{\tr(T_{\mathcal{O}'}(2^{k+1}n))}{|\operatorname{Pic}(\mathcal{O}')|}
=\sum_{\substack{r^2\leq 2^{k+3}n\\ 4\nmid r}}R_{\operatorname{gen}(\widetilde{f}_{S})}\left(\frac{2(2^{k+3}n-r^2)}{m_{f_{S}}}\right).
\end{equation*}
\end{prop}

\begin{proof}
By Corollary~\ref{cor:trace} we have
\begin{equation}\label{gongshi5.141}
\tr(T_{\mathcal{O}}(n))=|\operatorname{Pic}(\mathcal{O})|\sum_{r^2\leq 4n} R_{\operatorname{gen}(\widetilde{f}_\mathcal{O})}\left(\frac{4n-r^2}{m_{f_{\mathcal{O}}}}\right),
\end{equation}
and
\begin{equation}\label{gongshi5.151}
\tr(T_{\mathcal{O}'}(n))=|\operatorname{Pic}(\mathcal{O}')|\sum_{r^2\leq 4n} R_{\operatorname{gen}(\widetilde{f}_{\mathcal{O}'})}\left(\frac{4n-r^2}{m_{f_{\mathcal{O}'}}}\right),
\end{equation}
where
\begin{equation*}
R_{\operatorname{gen}(\widetilde f_{\mathcal{O}})}\left(n\right)=\sum\limits_{f\in \operatorname{gen}(\widetilde f_{\mathcal{O}})}\frac{R_{f}\left({n}\right)}{|\Aut(f)|},\,\,\,
R_{\operatorname{gen}(\widetilde f_{\mathcal{O}'})}\left(n\right)=\sum\limits_{f\in \operatorname{gen}(\widetilde f_{\mathcal{O}'})}\frac{R_{f}\left({n}\right)}{|\Aut(f)|}.
\end{equation*}
Recall that by the definition of $\varphi_p$ and $\mathcal{O}'=\varphi_p(\mathcal{O})$ we have $\widetilde{f}_{\mathcal{O}}=\phi_p(\widetilde{f}_{\mathcal{O}'})$, hence $\phi_p$ canonically induces a bijection between the equivalence classes of $\operatorname{gen}(\widetilde f_{\mathcal{O}'})$ and $\operatorname{gen}(\widetilde f_{\mathcal{O}})$. Note that by \cite[Proposition 2.5, 6.6]{LZ24} and in view of $p\parallel d$ we have for any $f\in \operatorname{gen}(\widetilde f_{\mathcal{O}'})$,
$$
R_{f}\left(n\right)=R_{\phi_p(f)}\left(pn\right),\,\,\, |\Aut(f)|=|\Aut(\phi_p(f))|.
$$
Thus
$$
R_{\operatorname{gen}(\widetilde f_{\mathcal{O}})}\left(pn\right)
=\sum\limits_{f\in \operatorname{gen}(\widetilde f_{\mathcal{O}})}\frac{R_{f}\left({pn}\right)}{|\Aut(f)|}
=\sum\limits_{f\in \operatorname{gen}(\widetilde f_{\mathcal{O}'})}\frac{R_{\phi_p(f)}\left({pn}\right)}{|\Aut(\phi_p(f))|}
=\sum\limits_{f\in \operatorname{gen}(\widetilde f_{\mathcal{O}'})}\frac{R_{f}\left({n}\right)}{|\Aut(f)|}
=R_{\operatorname{gen}(\widetilde f_{\mathcal{O}'})}\left(n\right).
$$
Hence by ~\eqref{gongshi5.13} we have for $p\parallel d$, 
\begin{align*}
&\quad \sum_{r^2\leq 4p^kn}R_{\operatorname{gen}(\widetilde{f}_{\mathcal{O}})}\left(\frac{4p^kn-r^2}{m_{f_{\mathcal{O}}}}\right)\\
&=\sum_{p^2r^2\leq 4p^kn}R_{\operatorname{gen}(\widetilde{f}_{\mathcal{O}})}\left(\frac{p(4p^kn-p^2r^2)}{m_{f_{\mathcal{O}'}}}\right)
+\sum_{r^2\leq 4p^kn,\,\, p\nmid r}R_{\operatorname{gen}(\widetilde{f}_{\mathcal{O}})}\left(\frac{4p^kn-r^2}{m_{f_{\mathcal{O}}}}\right)\\
&=\sum_{p^2r^2\leq 4p^kn}R_{\operatorname{gen}(\widetilde{f}_{\mathcal{O}'})}\left(\frac{4p^kn-p^2r^2}{m_{f_{\mathcal{O}'}}}\right)
+\sum_{r^2\leq 4p^kn,\,\, p\nmid r}R_{\operatorname{gen}(\widetilde{f}_{\mathcal{O}})}\left(\frac{4p^kn-r^2}{m_{f_{\mathcal{O}}}}\right)\\
&=\sum_{r^2\leq 4p^kn}R_{\operatorname{gen}(\widetilde{f}_{\mathcal{O}'})}\left(\frac{4p^kn-r^2}{m_{f_{\mathcal{O}'}}}\right)
+\sum_{r^2\leq 4p^kn,\,\, p\nmid r}R_{\operatorname{gen}(\widetilde{f}_{\mathcal{O}})}\left(\frac{4p^kn-r^2}{m_{f_{\mathcal{O}}}}\right).
\end{align*}
Here for the last equality we used the fact that if $p\nmid 4p^kn-r^2$, i.e. $p\nmid r$ then $R_{\operatorname{gen}(\widetilde{f}_{\mathcal{O}'})}\left(\frac{4p^kn-r^2}{m_{f_{\mathcal{O}'}}}\right)=0$.  
This together with ~\eqref{gongshi5.141} and ~\eqref{gongshi5.151} proves the first identity of the proposition.

If $2\parallel d$, then similar to the above, we have
\begin{equation}\label{gongshi5.161}
\tr(T_{S}(n))=|\operatorname{Pic}(S)|\sum_{r^2\leq 4n} R_{\operatorname{gen}(\widetilde{f}_S)}\left(\frac{2(4n-r^2)}{m_{f_{S}}}\right),
\end{equation}
and similar to \cite[Proposition 6.6]{LZ24} we have for any $f\in \operatorname{gen}(\widetilde f_{\mathcal{O}'})$, 
$$
R_{f}\left(n\right)=R_{\phi_2(f)}\left(8n\right).
$$
Thus from $\widetilde{f}_S=2\widetilde{f}_\mathcal{O}$ we get 
\begin{align*}
R_{\operatorname{gen}(\widetilde f_{S})}\left(16n\right)
&=R_{\operatorname{gen}(\widetilde f_{\mathcal{O}})}\left(8n\right)
=\sum\limits_{f\in \operatorname{gen}(\widetilde f_{\mathcal{O}})}\frac{R_{f}\left({8n}\right)}{|\Aut(f)|}\\
&=\sum\limits_{f\in \operatorname{gen}(\widetilde f_{\mathcal{O}'})}\frac{R_{\phi_2(f)}\left({8n}\right)}{|\Aut(\phi_2(f))|}
=\sum\limits_{f\in \operatorname{gen}(\widetilde f_{\mathcal{O}'})}\frac{R_{f}\left({n}\right)}{|\Aut(f)|}
=R_{\operatorname{gen}(\widetilde f_{\mathcal{O}'})}\left(n\right).
\end{align*}
Hence in view of $m_{f_{\mathcal{O}'}}=32m_{f_\mathcal{O}},\,\,(2,\,m_{f_\mathcal{O}})=1$ and $m_{f_S}=4m_{f_\mathcal{O}}$, we obtain 
\begin{align*}
&\quad \sum_{r^2\leq 4\cdot 2^{k+2}n}R_{\operatorname{gen}(\widetilde{f}_{S})}\left(\frac{2(4\cdot 2^{k+2}n-r^2)}{m_{f_{S}}}\right)\\
&=\sum_{2^6r^2\leq 4\cdot 2^{k+2}n}R_{\operatorname{gen}(\widetilde{f}_{S})}\left(\frac{16(4\cdot 2^{k+2}n-2^6r^2)}{m_{f_{\mathcal{O}'}}}\right)
+\sum_{r^2\leq 2^{k+4}n,\,\, 8\nmid r}R_{\operatorname{gen}(\widetilde{f}_{S})}\left(\frac{2(4\cdot 2^{k+2}n-r^2)}{m_{f_{S}}}\right)\\
&=\sum_{2^6r^2\leq 4\cdot 2^{k+2}n}R_{\operatorname{gen}(\widetilde{f}_{\mathcal{O}'})}\left(\frac{4\cdot 2^{k+2}n-2^6r^2}{m_{f_{\mathcal{O}'}}}\right)
+\sum_{r^2\leq 2^{k+4}n,\,\, 8\nmid r}R_{\operatorname{gen}(\widetilde{f}_{S})}\left(\frac{2(4\cdot 2^{k+2}n-r^2)}{m_{f_{S}}}\right)\\
&=\sum_{r^2\leq 4\cdot 2^{k+2}n}R_{\operatorname{gen}(\widetilde{f}_{\mathcal{O}'})}\left(\frac{4\cdot 2^{k+2}n-r^2}{m_{f_{\mathcal{O}'}}}\right)
+\sum_{r^2\leq 2^{k+4}n,\,\, 8\nmid r}R_{\operatorname{gen}(\widetilde{f}_{S})}\left(\frac{2(2^{k+4}n-r^2)}{m_{f_{S}}}\right).
\end{align*}
This together with ~\eqref{gongshi5.151} and ~\eqref{gongshi5.161} proves the second identity of the proposition.

If $2\nmid d$, then we have for any $f\in \operatorname{gen}(\widetilde f_{\mathcal{O}'})$, 
$$
R_{f}\left(n\right)=R_{\phi_2(f)}\left(4n\right).
$$
Thus from $\widetilde{f}_S=2\widetilde{f}_\mathcal{O}$ we get 
\begin{align*}
R_{\operatorname{gen}(\widetilde f_{S})}\left(8n\right)
&=R_{\operatorname{gen}(\widetilde f_{\mathcal{O}})}\left(4n\right)
=\sum\limits_{f\in \operatorname{gen}(\widetilde f_{\mathcal{O}})}\frac{R_{f}\left({4n}\right)}{|\Aut(f)|}\\
&=\sum\limits_{f\in \operatorname{gen}(\widetilde f_{\mathcal{O}'})}\frac{R_{\phi_2(f)}\left({4n}\right)}{|\Aut(\phi_2(f))|}
=\sum\limits_{f\in \operatorname{gen}(\widetilde f_{\mathcal{O}'})}\frac{R_{f}\left({n}\right)}{|\Aut(f)|}
=R_{\operatorname{gen}(\widetilde f_{\mathcal{O}'})}\left(n\right).
\end{align*}
Hence in view of $m_{f_{\mathcal{O}'}}=16m_{f_\mathcal{O}},\,\,(2,\,m_{f_\mathcal{O}})=1$ and $m_{f_S}=4m_{f_\mathcal{O}}$, we obtain 
\begin{align*}
&\quad \sum_{r^2\leq 4\cdot 2^{k+1}n}R_{\operatorname{gen}(\widetilde{f}_{S})}\left(\frac{2(4\cdot 2^{k+1}n-r^2)}{m_{f_{S}}}\right)\\
&=\sum_{2^4r^2\leq 4\cdot 2^{k+1}n}R_{\operatorname{gen}(\widetilde{f}_{S})}\left(\frac{8(4\cdot 2^{k+1}n-2^4r^2)}{m_{f_{\mathcal{O}'}}}\right)
+\sum_{r^2\leq 2^{k+3}n,\,\, 4\nmid r}R_{\operatorname{gen}(\widetilde{f}_{S})}\left(\frac{2(4\cdot 2^{k+1}n-r^2)}{m_{f_{S}}}\right)\\
&=\sum_{2^4r^2\leq 4\cdot 2^{k+1}n}R_{\operatorname{gen}(\widetilde{f}_{\mathcal{O}'})}\left(\frac{4\cdot 2^{k+1}n-2^4r^2}{m_{f_{\mathcal{O}'}}}\right)
+\sum_{r^2\leq 2^{k+3}n,\,\, 4\nmid r}R_{\operatorname{gen}(\widetilde{f}_{S})}\left(\frac{2(4\cdot 2^{k+1}n-r^2)}{m_{f_{S}}}\right)\\
&=\sum_{r^2\leq 4\cdot 2^{k+1}n}R_{\operatorname{gen}(\widetilde{f}_{\mathcal{O}'})}\left(\frac{4\cdot 2^{k+1}n-r^2}{m_{f_{\mathcal{O}'}}}\right)
+\sum_{r^2\leq 2^{k+3}n,\,\, 4\nmid r}R_{\operatorname{gen}(\widetilde{f}_{S})}\left(\frac{2(2^{k+3}n-r^2)}{m_{f_{S}}}\right).
\end{align*}
This together with ~\eqref{gongshi5.151} and ~\eqref{gongshi5.161} proves the last identity of the proposition.
\end{proof}

\subsection{The independence of isomorphism classes of orders and the associated theta series}
For a quaternion algebra $Q$ ramified exactly at $p$ and $\infty$, Goren and Love~\cite{GL25} established a result concerning the theta series $\theta_\mathcal{O}(\tau)$ defined by ~\eqref{gongshi5.2}, which asserts that the maximal orders $\mathcal{O}$ of $Q$ are uniquely determined up to isomorphism by their theta series $\theta_\mathcal{O}(\tau)$. 
As for the theta series defined by ~\eqref{gongshi5.3}, ~\eqref{gongshi5.4} (and hence ~\eqref{gongshi5.1} by ~\eqref{gongshi5.6} and ~\eqref{gongshi5.7}), we may establish similar results in the following Theorem ~\ref{the:theta}.  

\begin{theorem}\label{the:theta}
Let $\mathcal{O},\mathcal{O}'$ be orders in $Q$. If both $\mathcal{O}$ and $\mathcal{O}'$ are even, then
\begin{equation*}
\mathcal{O}\ncong \mathcal{O}'\iff \theta_{\mathcal{O}^0}(\tau)\neq\theta_{\mathcal{O}'^0}(\tau).
\end{equation*}
Alternatively, if both $\mathcal{O}$ and $\mathcal{O}'$ are odd, then
\begin{equation*}
\mathcal{O}\ncong \mathcal{O}'\iff \theta_{S^0}(\tau)\neq\theta_{S'^0}(\tau);
\end{equation*}
moreover, if one assumes further $\mathcal{O},\mathcal{O}'\in\mathbb{Z}+m\operatorname{Typ}_i(2)$, or  $\mathcal{O},\mathcal{O}'\in\mathbb{Z}+m\operatorname{Typ}_2(1)$ for some odd positive integer $m$, then
\begin{equation*}
\mathcal{O}\ncong \mathcal{O}'\iff\theta_{\mathcal{O}^0}(\tau)\neq\theta_{\mathcal{O}'^0}(\tau).
\end{equation*}
\end{theorem}

\begin{proof}

By Schiemann's theorem, equality of the relevant ternary theta
series implies integral equivalence of the corresponding forms.
The coefficient identity for ${\frak q}_{\mathcal O}$ established above gives,
in the corresponding bases,
\[
 M_{f_{S^0}}=2\operatorname{adj}(M_{f_{\mathcal O}}),\qquad
 M_{f_{\mathcal O}}=\sqrt{\frac{\det M_{f_{S^0}}}{2}}\,M_{f_{S^0}}^{-1}.
\]
Consequently, the integral equivalence class of $f_{S^0}$ determines that
of $f_{\mathcal O}$, for arbitrary orders, and hence determines
$[\mathcal O]$ by the classical Clifford correspondence.
For an even order, $S^0=2\mathcal O^0$, so that
$f_{S^0}\sim4f_{\mathcal O^0}$ and the same conclusion follows
from $f_{\mathcal O^0}$.
The converse implications follow from the invariance of reduced
norm and trace under isomorphisms of orders

Turn to the last statement. Without loss of generality we may assume $m=1$.
If $\mathcal{O},\mathcal{O}'\in\operatorname{Typ}_i(2)$, then 
$$
\mathcal{O}\ncong \mathcal{O}'\iff \mathcal{O}^+\ncong (\mathcal{O}')^+
$$ 
since ${\mathcal M}_+$ is bijective. Hence from the first statement we get     
\begin{equation*}
\mathcal{O}\ncong \mathcal{O}'\iff \theta_{(\mathcal{O}^+)^0}\neq \theta_{(\mathcal{O}'^+)^0},
\end{equation*}
which is what we need since $\theta_{\mathcal{O}^0}=\theta_{(\mathcal{O}^+)^0}, \theta_{\mathcal{O}'^0}=\theta_{(\mathcal{O}'^+)^0}$.
If $\mathcal{O},\mathcal{O}'\in\operatorname{Typ}_2(1)$, by Proposition ~\ref{prop:Watsontrans}, we have $\phi_2(f_{S^0})=\lambda_4(f_{S^0})=f_{\mathcal{O}^0}$. Hence 
\begin{equation*}
\mathcal{O}\ncong \mathcal{O}'\iff \theta_{S^0}\neq\theta_{S'^0}\iff \theta_{\mathcal{O}^0}\neq\theta_{\mathcal{O}'^0}.
\end{equation*}
The proof is complete.
\end{proof}

\begin{remark}
The parity condition in this theorem is important and can not be omitted. To see this we need only consider the following example.  
Suppose $\mathcal{O}$ is the Hurwitz order, $\mathcal{O}'$ is the Lipschitz order, then it is easy to check that $\theta_{\mathcal{O}^0}=\theta_{\mathcal{O}'^0}$ by definition, but $\mathcal{O}\ncong \mathcal{O}'$ since $\mathcal{O}$ is odd, $\mathcal{O}'$ is even.
\end{remark}

\begin{remark}\label{remark5.7} 
It would be appropriate to point out that, the conditions $\mathcal{O},\mathcal{O}'\in\mathbb{Z}+m\operatorname{Typ}_i(2)$, or  $\mathcal{O},\mathcal{O}'\in\mathbb{Z}+m\operatorname{Typ}_2(1)$ are crucial and can not be skipped for the last statement of the theorem. In fact, we may construct counterexamples:  
Examples ~\ref{exam:nodd} and ~\ref{exam:neven} in \S \ref{section6.00}, which show that $\theta_{\mathcal{O}^0}=\theta_{\mathcal{O}'^0}$ but $\mathcal{O}\ncong \mathcal{O}'$ if one drops the conditions. 
\end{remark}

\begin{remark}\label{remark5.8}
Indeed, Example ~\ref{exam:nodd} gives explicit odd orders
$\mathcal O$ and $\mathcal O'$ which are locally isomorphic but not isomorphic, such that
$\theta_{\mathcal O^0}=\theta_{(\mathcal O')^0}$.
It is a counterexample to a statement following ~\cite[Conjecture~4.3]{Has06} of Hashimoto concerning $T_1$ there, which says that for orders $\mathcal{O}$ and $\mathcal{O}'$, if $\mathcal{O}$ is not isomorphic but locally isomorphic to $\mathcal{O}'$, then $\theta_{\mathcal{O}^0}\neq\theta_{\mathcal{O}'^0}$. 
Example~\ref{exam:neven} shows that the same can occur even for odd orders in different genera.
\end{remark}

\subsection{Hashimoto's conjecture}
Now, we come to consider Hashimoto's conjecture ~\cite{Has06} on linear
relations among theta series associated with quaternion orders.
Let $\mathcal O_1,\ldots,\mathcal O_t$ represent the isomorphism
classes in the genus of a Gorenstein order $\mathcal O$, and put
$S_j=\mathbb Z+2\mathcal O_j$ for $1\leq j\leq t$.
The conjecture asserts that, if $\mathcal O$ has level $(q,N)$ (see~\cite[Section 2]{Has06}), 
then for any $c_1,\ldots,c_t\in\mathbb C$,
\begin{equation}\label{gongshi5.14}
\sum_{j=1}^t c_j\theta_{S_j^0}(\tau)=0
\quad\Longleftrightarrow\quad
\sum_{j=1}^t c_j\theta_{\mathcal O_j}(\tau)=0.
\end{equation}
And if, in addition, $4\nmid N$, then
\begin{equation}\label{gongshi5.15}
\sum_{j=1}^t c_j\theta_{\mathcal O_j^0}(\tau)=0
\quad\Longleftrightarrow\quad
\sum_{j=1}^t c_j\theta_{\mathcal O_j}(\tau)=0.
\end{equation}

We establish the linear-relation equivalence for even orders.
For odd Gorenstein orders, we compare the two ternary theta series
when $4\nmid d$, where $d=\operatorname{discrd}(\mathcal O)$,
and reformulate the remaining assertion as injectivity of the
restriction map on the subspace generated by the quaternionic
Jacobi theta series. We do not establish this injectivity at
arbitrary level.

\begin{prop}\label{prop:hashimoto-even-orders}
Let $\mathcal O_1,\ldots,\mathcal O_t$ be even orders.
For any $c_1,\ldots,c_t\in\mathbb C$, we have
\begin{equation}\label{eq:hashimoto-even-equivalences}
\sum_{j=1}^t c_j\theta_{\mathcal O_j^0}=0
\quad\Longleftrightarrow\quad
\sum_{j=1}^t c_j\theta_{\mathcal O_j}=0.
\end{equation}
\end{prop}

\begin{proof}
For an even order $\mathcal O$, every $\alpha\in\mathcal O$
has a unique decomposition
\[
\alpha=m+\beta,\qquad
m=\frac{\operatorname{tr}(\alpha)}2\in\mathbb Z,
\qquad \beta\in\mathcal O^0.
\]
Since $\operatorname n(m+\beta)=m^2+\operatorname n(\beta)$, we obtain
\begin{equation}\label{eq:hashimoto-even-factorization}
\theta_{\mathcal O}(\tau)
=\vartheta(\tau)\theta_{\mathcal O^0}(\tau),
\qquad
\vartheta(\tau)=\sum_{m\in\mathbb Z}q^{m^2}.
\end{equation}
The result follows because $\vartheta$ has constant term $1$, so
multiplication by $\vartheta$ is injective on formal power series.
\end{proof}

For the rest of the discussion, let $\mathcal O$ be odd and
Gorenstein, and write $d=\operatorname{discrd}(\mathcal O)$.
Thus $d=N_1N_2$ in the level-$(N_1,N_2)$ case. We use a common theta
decomposition to describe all three families of linear relations.
For $\mu\in\{0,1\}$, set
\[
\vartheta_\mu(\tau,z)
=\sum_{\substack{r\in\mathbb Z\\r\equiv\mu\!\!\!\!\pmod2}}
q^{r^2/4}\zeta^r,
\quad
\vartheta_\mu(\tau)=\vartheta_\mu(\tau,0),
\]
where $q=e^{2\pi i\tau}$ and $\zeta=e^{2\pi iz}$ as before.
Writing $a_j(D)=R_{f_{S_j^0}}(D)$, define
\[
h_{j,0}(\tau)=\sum_{n\geq0}a_j(4n)q^n,
\quad
h_{j,1}(\tau)=\sum_{n\geq1}a_j(4n-1)q^{n-1/4}.
\]
The norm--trace identity~\eqref{gongshi5.6} gives
\begin{equation}\label{eq:hashimoto-theta-components}
\begin{aligned}
&\theta_{\mathcal O_j}(\tau,z)
=h_{j,0}(\tau)\vartheta_0(\tau,z)
 +h_{j,1}(\tau)\vartheta_1(\tau,z),\\
&\theta_{\mathcal O_j^0}(\tau)=h_{j,0}(\tau),\\
&\theta_{S_j^0}(\tau)=h_{j,0}(4\tau)+h_{j,1}(4\tau).
\end{aligned}
\end{equation}
For fixed $c_1,\ldots,c_t\in\mathbb C$, put
\[
F(\tau,z)=\sum_{j=1}^t c_j\theta_{\mathcal O_j}(\tau,z),
\quad
h_\mu(\tau)=\sum_{j=1}^t c_jh_{j,\mu}(\tau)
\quad(\mu=0,1).
\]
By the Jacobi modularity established above,
$F\in J_{2,1}(\Gamma_0(d))$, with trivial character, and
$F=h_0\vartheta_0+h_1\vartheta_1$.
The Fourier supports of $h_0(4\tau)$ and $h_1(4\tau)$ lie in
$0$ and $3$ modulo $4$, respectively. Together with the uniqueness
of the theta decomposition, this yields
\begin{equation}\label{eq:hashimoto-components-vanishing}
\sum_{j=1}^t c_j\theta_{S_j^0}=0
\quad\Longleftrightarrow\quad h_0=h_1=0
\quad\Longleftrightarrow\quad F=0.
\end{equation}
In particular, without any restriction on $d$,
\begin{equation}\label{gongshi5.16}
\sum_{j=1}^t c_j\theta_{S_j^0}=0
\quad\Longrightarrow\quad
\sum_{j=1}^t c_j\theta_{\mathcal O_j^0}=0.
\end{equation}
Evaluating $F$ at $z=0$ also proves the forward implication in
\eqref{gongshi5.14}. 

\vskip0.2cm
The next proposition gives the converse of
\eqref{gongshi5.16} when $4\nmid d$.

\begin{prop}\label{prop:hashimoto-ternary-comparison}
With the above notation, suppose $4\nmid d$, i.e., $\mathcal O_1,\ldots,\mathcal O_t\in \operatorname{Typ}_i(2^k)$ for $k=0,1$. Then for any
$c_1,\ldots,c_t\in\mathbb C$,
\begin{equation}\label{eq:hashimoto-ternary-equivalence}
\sum_{j=1}^t c_j\theta_{S_j^0}=0
\quad\Longleftrightarrow\quad
\sum_{j=1}^t c_j\theta_{\mathcal O_j^0}=0.
\end{equation}
\end{prop}

\begin{proof}
By \eqref{gongshi5.16}, it remains to prove that $h_0=0$ implies
$h_1=0$. Let
\[
S=\begin{pmatrix}0&-1\\1&0\end{pmatrix},
\quad
T=\begin{pmatrix}1&1\\0&1\end{pmatrix}.
\]
The standard theta transformation matrices are
\[
\rho(T)=\begin{pmatrix}1&0\\0&i\end{pmatrix},
\quad
\rho(S)=\frac{e^{-\pi i/4}}{\sqrt2}
\begin{pmatrix}1&1\\1&-1\end{pmatrix}.
\]
The vector $(h_0,h_1)^t$ transforms with weight $3/2$ according
to the contragredient theta representation. For
\[
\gamma_d=ST^{-d}S^{-1}
=\begin{pmatrix}1&0\\d&1\end{pmatrix}\in\Gamma_0(d),
\]
we obtain
\begin{equation}\label{eq:hashimoto-component-mixing}
\begin{pmatrix}h_0(\gamma_d\tau)\\h_1(\gamma_d\tau)\end{pmatrix}
=\kappa_d(d\tau+1)^{3/2}\,
\frac12\begin{pmatrix}
1+i^d&1-i^d\\
1-i^d&1+i^d
\end{pmatrix}
\begin{pmatrix}h_0(\tau)\\h_1(\tau)\end{pmatrix},
\end{equation}
where $\kappa_d\in\mathbb C^\times$ accounts for the choice of
the half-integral weight multiplier. If $h_0=0$, the first component
implies $(1-i^d)h_1=0$. Since $4\nmid d$, we have $1-i^d\neq0$,
and hence $h_1=0$. The result follows from
\eqref{eq:hashimoto-components-vanishing}.
\end{proof}

Here it may be appropriate to point out that the condition $4\nmid d$ in
Proposition~\ref{prop:hashimoto-ternary-comparison} is crucial and cannot be
omitted in general. To see this, choose a genus of odd orders on
which the restriction of $\mathcal M_+$ is not injective, and let
$\mathcal O_u,\mathcal O_v$ be nonisomorphic representatives with
$(\mathcal O_u)^+\simeq(\mathcal O_v)^+$.
Such pairs occur in the families
$\mathbb Z+m\operatorname{Typ}_i(2^n)$ with $n\geq2$ and $m$
odd, already with $m=1$ in the examples below.
Since $(\mathcal O_j^+)^0=\mathcal O_j^0$, we have
\[
\theta_{\mathcal O_u^0}
=\theta_{(\mathcal O_u^+)^0}
=\theta_{(\mathcal O_v^+)^0}
=\theta_{\mathcal O_v^0}.
\]
On the other hand, Theorem~\ref{the:theta} gives
$\theta_{S_u^0}\neq\theta_{S_v^0}$.
Thus, taking $c_u=1$, $c_v=-1$ and $c_j=0$ for $j\neq u,v$, we get
\begin{equation}\label{gongshi5.17}
\sum_{j=1}^t c_j\theta_{\mathcal O_j^0}=0,
\quad
\sum_{j=1}^t c_j\theta_{S_j^0}\neq0.
\end{equation}
This shows precisely that the converse of \eqref{gongshi5.16}
can fail when $4\mid d$.

Returning to the odd Gorenstein genus fixed above, we can now
formulate the remaining part of the conjecture as an injectivity
problem.

\begin{cor}\label{cor:hashimoto-restriction}
Set
\[
\mathcal J(\mathcal O)
=\operatorname{span}_{\mathbb C}
\{\theta_{\mathcal O_j}(\tau,z):1\leq j\leq t\}
\subseteq J_{2,1}(\Gamma_0(d)).
\]
Then \eqref{gongshi5.14}, asserted for all
$c_1,\ldots,c_t\in\mathbb C$, is equivalent to the injectivity of
\[
D_0:\mathcal J(\mathcal O)\longrightarrow M_2(\Gamma_0(d)),
\quad F(\tau,z)\longmapsto F(\tau,0).
\]
If $4\nmid d$, this is also equivalent to
\eqref{gongshi5.15}. Consequently, under this condition the two
formulations of Hashimoto's conjecture are equivalent.
\end{cor}

\begin{proof}
By \eqref{eq:hashimoto-components-vanishing}, the left-hand side
of \eqref{gongshi5.14} is equivalent to $F=0$, whereas its
right-hand side is equivalent to $D_0F=0$. This proves the first
assertion. The second follows from
Proposition~\ref{prop:hashimoto-ternary-comparison}.
\end{proof}

\begin{remark}
With the notation above, the remaining injectivity
question is precisely whether
\[
 \mathcal J(\mathcal O)\cap\ker D_0=\{0\}.
\]
Indeed,
$D_0\theta_{\mathcal O_i}(\tau,z)=\theta_{\mathcal O_i}(\tau)$,
so this condition says that a linear combination of the
Jacobi theta series vanishes whenever its restriction
to $z=0$ vanishes.
The preceding identities and the comparison of ternary
theta series in
Proposition~\ref{prop:hashimoto-ternary-comparison}
do not establish this injectivity at arbitrary level.
For related work on restriction maps and differential
operators for Jacobi forms, see \cite{RS13,DR15}.
\end{remark}

\section{Applications: New examples of orders and order chains}\label{section6.00} 
In this section we give some concrete examples of new orders and order chains, and consider their relations from the reciprocal point of view of Corollary ~\ref{cor:automorphism}. 

\subsection{Analysis and generalization of a class of orders of Pizer}\label{sec:6.1}
Following Pizer~\cite{Piz80}, for given odd prime $p$ we let 
\begin{equation*}
Q_p'=
\left\{\begin{pmatrix}
\alpha & \beta\\
\epsilon\bar{\beta} & \bar{\alpha}
\end{pmatrix}
: \alpha,\beta\in\mathbb{Q}_p+\sqrt{p}\mathbb{Q}_p\right\}  
\end{equation*} 
denote the local quaternion division algebra over $\mathbb{Q}_p$, where $\epsilon$ is a quadratic non-residue modulo $p$. Set for any integer $r\geq 0$, 
\begin{equation*}
{\mathcal{O}}_p^{(r+1)}=\left\{\begin{pmatrix}
\alpha & \beta\\
\epsilon\bar{\beta} & \bar{\alpha}
\end{pmatrix}
: \alpha\in \mathbb{Z}_p+\sqrt{p}\mathbb{Z}_p,\,\beta\in \sqrt{p^r}(\mathbb{Z}_p+\sqrt{p}\mathbb{Z}_p)\right\}.  
\end{equation*} 
It is clear ${\mathcal{O}}_p^{(1)}$ is the unique maximal order, and ${\mathcal{O}}_p^{(2)}$ is a (unique) suborder of maximal order with index $p$.
Based on this, for any positive integer $N$, Pizer~\cite{Piz80} defined the orders of level $p^2N$.

\begin{definition}[\cite{Piz80}, Definition 3.1]\label{{Piz80}1111} 
For a given positive integer $N$, let $p$ be an odd prime such that $p\nmid N$. Let $Q$ be a definite quaternion algebra ramified exactly at $p$. Then an order $\mathcal{O}\subset Q$ is called of level $p^2N$ if $\mathcal{O}_p$ is isomorphic to ${\mathcal{O}}_p^{(2)}$ and for any prime $q\neq p$, $\mathcal{O}_q$ is isomorphic to $\begin{pmatrix}\mathbb{Z}_q&\mathbb{Z}_q\\ N\mathbb{Z}_q&\mathbb{Z}_q\end{pmatrix}$.
\end{definition}

The following proposition establishes an explicit connection of ${\mathcal{O}}_p^{(1)}$ and ${\mathcal{O}}_p^{(2)}$.

\begin{prop}\label{proposition7.2}
We have $\psi({\mathcal{O}}_p^{(1)})\simeq{\mathcal{O}}_p^{(2)}$, where $\psi$ is defined as in Corollary {\rm ~\ref{cor:automorphism}}, and $\psi({\mathcal{O}}_p)$ is defined as $\psi({\mathcal{O}})_p$ for any order $\mathcal{O}$ and any prime $p$. 
\end{prop}

\begin{proof}
Throughout the proof, similarity is taken over $\mathbb Z_p$.
Put
\[
 g_1=-\epsilon x^2-py^2+\epsilon pz^2,
 \qquad
 g_2=-\epsilon px^2-y^2+\epsilon z^2.
\]
Let $\mathcal O^{(1)}$ and $\mathcal O^{(2)}$ be global
orders whose localizations at $p$ are isomorphic to
$\mathcal O_p^{(1)}$ and $\mathcal O_p^{(2)}$, respectively.
The defining local models give
\[
 \widetilde f_{\mathcal O^{(1)}}\sim_p u_1g_1,
 \qquad
 \widetilde f_{\mathcal O^{(2)}}\sim_p u_2g_2,
 \qquad u_1,u_2\in\mathbb Z_p^\times.
\]
Here the unit factors account for the global normalization
of the associated forms.

We compare their local reciprocal similarity classes.
Since $p$ is odd, a primitive local reciprocal of a
diagonal form is obtained by taking its adjugate and
dividing out its content.
For $g_1$, this gives
\[
 \frac1p
 \operatorname{adj}
 \begin{pmatrix}
  -\epsilon&0&0\\
  0&-p&0\\
  0&0&\epsilon p
 \end{pmatrix}
 =
 \begin{pmatrix}
  -\epsilon p&0&0\\
  0&-\epsilon^2&0\\
  0&0&\epsilon
 \end{pmatrix}.
\]
The substitution $y\mapsto\epsilon^{-1}y$ therefore
identifies this form with $g_2$.
Consequently,
$\widetilde f_{\mathcal O^{(2)}}$ and
$\psi_{\widetilde f_{\mathcal O^{(1)}}}$
are similar over $\mathbb Z_p$.

Reciprocity preserves local unit-similarity classes,
and unit-similar ternary forms have isomorphic even
Clifford orders. Using
$\mathcal M=\psi\circ C_0^{-1}$ and the definition of
reciprocal orders, we conclude that
\[
 \psi(\mathcal O^{(1)})_p
 \simeq \mathcal O^{(2)}_p.
\]
This proves the assertion.
\end{proof}

The corresponding classical ternary forms satisfy
\[
 f_{\mathcal O^{(1)}}\sim_p
 v_1(-\epsilon px^2-y^2+\epsilon z^2),
 \qquad v_1\in\mathbb Z_p^\times,
\]
and
\[
 f_{\mathcal O^{(2)}}\sim_p
 v_2(-\epsilon x^2-py^2+\epsilon pz^2),
 \qquad v_2\in\mathbb Z_p^\times.
\]
Thus these local models determine the associated forms
up to unit similarity.

We remark also that, by Definition ~\ref{{Piz80}1111}, Proposition ~\ref{proposition7.2} and Corollary~\ref{cor:automorphism}, there exists a bijection $\varphi_p$ between Eichler orders of level $pN$ and orders of level $p^2N$.
But we have to point out that there exist two constraints here: one is $p\ne 2$, and the other is that $\mathcal{O}$ ramifies only at $p$. 

\vskip 0.2cm
The following example is an illustration for the result of Pizer above, in which we have $p=3$, $N=1$. 

\begin{example}\label{example1.00}
Let $Q=\left(\frac{-3,-1}{\mathbb{Q}}\right)$, where $i^2=-3$, $j^2=-1$, $k=ij$. It is easy to verify that $Q$ ramifies only at prime $3$ since the Hilbert symbol $(-3,-1)_p=-1$ if $p=3$, and $=1$ otherwise. Then one can check that 
\begin{equation*}
\mathcal{O}=\mathbb{Z}+\mathbb{Z}\frac{1+i}{2}+\mathbb{Z}\frac{3j+k}{2}+\mathbb{Z}\frac{j+k}{2}
\end{equation*}
is a maximal order; and its discriminant, by direct computation using the given basis, is $3$. 
Again, by Definition ~\ref{{Piz80}1111}, it can be checked directly that  
\begin{equation*}
\mathcal{O}'=\mathbb{Z}+\mathbb{Z}\frac{1+i}{2}+\mathbb{Z}\frac{3j+k}{2}+\mathbb{Z}\frac{3j+3k}{2} 
\end{equation*}
is an order of level $3^2$.
\end{example}

The following proposition is a generalization of the result of Pizer above; hence gives more explicit forms of orders associated with Eichler orders.

\begin{prop}\label{proposition7.4}
Let $p$ be an odd prime and let $\epsilon$ be a quadratic non-residue modulo $p$. Let $\psi$ be defined as in Corollary {\rm ~\ref{cor:automorphism}} and let 
\begin{equation*}
\mathcal{O}_p=\left\{\begin{pmatrix}
\alpha & \beta\\
p\bar{\beta} & \bar{\alpha}
\end{pmatrix}
: \alpha,\beta\in \mathbb{Z}_p+\sqrt{\epsilon}\mathbb{Z}_p\right\}, 
\end{equation*} 
then
\begin{equation*}
\psi(\mathcal{O}_p)=\left\{\begin{pmatrix}
\alpha & \beta\\
p\bar{\beta} & \bar{\alpha}
\end{pmatrix}
: \alpha\in \mathbb{Z}_p+p\sqrt{\epsilon}\mathbb{Z}_p,\,\beta\in \mathbb{Z}_p+\sqrt{\epsilon}\mathbb{Z}_p\right\}.
\end{equation*} 
Further, for $p=2$, if we let
\begin{equation*}
\mathcal{O}_2=\left\{\begin{pmatrix}
\alpha & \beta\\
2\bar{\beta} & \bar{\alpha}
\end{pmatrix}
: \alpha,\beta\in \mathbb{Z}_2+\frac{1+\sqrt{-3}}{2}\mathbb{Z}_2\right\},  
\end{equation*} 
then
\begin{equation*}
\psi(\mathcal{O}_2)=\left\{\begin{pmatrix}
\alpha & 2\beta\\
4\bar{\beta} & \bar{\alpha}
\end{pmatrix}
: \alpha\in\mathbb{Z}_2+4\sqrt{-3}\mathbb{Z}_2,\,\beta\in \mathbb{Z}_2+\frac{1+\sqrt{-3}}{2}\mathbb{Z}_2\right\}.  
\end{equation*} 
Moreover, for odd prime $p$ and $n\geq1$, we have
\begin{equation*}
\psi\left(\begin{pmatrix}
\mathbb{Z}_p & \mathbb{Z}_p\\
p^n\mathbb{Z}_p & \mathbb{Z}_p
\end{pmatrix}\right)=\begin{pmatrix}
a+p^n\mathbb{Z}_p & \mathbb{Z}_p\\
p^n\mathbb{Z}_p & a+p^n\mathbb{Z}_p
\end{pmatrix},~
\psi\left(\begin{pmatrix}
\mathbb{Z}_2 & \mathbb{Z}_2\\
2^n\mathbb{Z}_2 & \mathbb{Z}_2
\end{pmatrix}\right)=\begin{pmatrix}
b+2^{n+2}\mathbb{Z}_2 & \mathbb{Z}_2\\
2^{n+2}\mathbb{Z}_2 & b+2^{n+2}\mathbb{Z}_2
\end{pmatrix},
\end{equation*}  
where $a\in\mathbb{Z}_p$, $b\in\mathbb{Z}_2$. In particular, we have 
\begin{equation*}
\psi\left(\begin{pmatrix}
\mathbb{Z}_p & \mathbb{Z}_p\\
\mathbb{Z}_p & \mathbb{Z}_p
\end{pmatrix}\right)=\begin{pmatrix}
\mathbb{Z}_p & \mathbb{Z}_p\\
\mathbb{Z}_p & \mathbb{Z}_p
\end{pmatrix},~
\psi\left(\begin{pmatrix}
\mathbb{Z}_2 & \mathbb{Z}_2\\
\mathbb{Z}_2 & \mathbb{Z}_2
\end{pmatrix}\right)=\begin{pmatrix}
b+4\mathbb{Z}_2 & \mathbb{Z}_2\\
4\mathbb{Z}_2 & b+4\mathbb{Z}_2
\end{pmatrix}.
\end{equation*}  
\end{prop}

\begin{proof}
The proof is a direct consequence of Corollary~\ref{cor:automorphism} by comparing the relevant quadratic forms using Definition~\ref{def:tildef}.
\end{proof}

\subsection{Two examples of chains of odd orders}

\begin{example}\label{exam:nodd}
There exist chains of explicit odd Gorenstein orders in quaternion algebras of even discriminant of the following form: 
\begin{equation*}
\mathcal{O}_5\subset\mathcal{O}_3\subset\mathcal{O}_1,
\end{equation*}
and
\begin{equation*}
\mathcal{O}_5'\subset\mathcal{O}_3'\subset\mathcal{O}_1',
\end{equation*}
such that
$\mathcal{O}_i$ and $\mathcal{O}_i'$ are locally isomorphic for $i=1,3,5$, and
\begin{equation*}
\mathcal{O}_5\ncong\mathcal{O}_5',~\mathcal{O}_3\cong\mathcal{O}_3',~\mathcal{O}_1\cong\mathcal{O}_1',~\mathcal{O}_5^+=S_3,~\mathcal{O}_5'^+=S_3',~\mathcal{O}_3^+=S_1,~\mathcal{O}_3'^+=S_1'.
\end{equation*}

Indeed, we can take a quaternion algebra $Q=\left(\frac{-32,-3}{\mathbb{Q}}\right)$. 
Then $Q$ ramifies only at prime $2$ since the Hilbert symbol $(-32,-3)_p=-1$ for $p=2$, and $=1$ for every other finite prime. Then using the construction method due to Hashimoto ~\cite[Proposition 3.1]{Has06} we may deduce an odd order $\mathcal{O}_5$ of level $(2^5,1)$ (see~\cite[Section 2]{Has06}) given by 
\begin{equation*}
\mathcal{O}_5=\mathbb{Z}+\mathbb{Z}\frac{1+j}{2}+\mathbb{Z}\frac{i+ij}{2}+\mathbb{Z}\frac{32j+ij}{3},
\end{equation*}
where $i^2=-32$, $j^2=-3$. Notice that the isomorphic equivalence class $[\mathcal{O}_5]$ is in $\operatorname{Typ}_1(2^5)$, hence the notation $\mathcal{O}_5$. Next, we let 
 \begin{equation*}
\mathcal{O}_3=\mathbb{Z}+\mathbb{Z}\frac{1+j}{2}+\mathbb{Z}\frac{i+ij}{4}+\mathbb{Z}\frac{32j+ij}{6},
\end{equation*}
\begin{equation*}
\mathcal{O}_1=\mathbb{Z}+\mathbb{Z}\frac{1+j}{2}+\mathbb{Z}\frac{i+ij}{8}+\mathbb{Z}\frac{32j+ij}{12}.
\end{equation*}
Then it is clear that $\mathcal{O}_5\subset\mathcal{O}_3\subset\mathcal{O}_1,$ and
\begin{align*}
&\mathcal{O}_5^+=S_3=\mathbb{Z}+\mathbb{Z}j+\mathbb{Z}\frac{i+ij}{2}+\mathbb{Z}\frac{32j+ij}{3},\\
&\mathcal{O}_3^+=S_1=\mathbb{Z}+\mathbb{Z}j+\mathbb{Z}\frac{i+ij}{4}+\mathbb{Z}\frac{32j+ij}{6},\\
&\mathcal{O}_1^+=\mathbb{Z}+\mathbb{Z}j+\mathbb{Z}\frac{i+ij}{8}+\mathbb{Z}\frac{32j+ij}{12}.
\end{align*}
Also, by the definition formula ~\eqref{gongshi3.8} we see that the trace-zero norm forms associated to $\mathcal{O}_5$, $\mathcal{O}_3$ and $\mathcal{O}_1$ are respectively 
\begin{align*}
\begin{aligned}
&f_{\mathcal{O}_5^0}=f_{(\mathcal{O}_5^+)^0}=f_{S_3^0}=(3,32,352,32,64,0),\quad N_{f_{\mathcal{O}_5^0}}=2^{5},\quad d_{f_{\mathcal{O}_5^0}}=2^{10};\\
&f_{\mathcal{O}_3^0}=f_{(\mathcal{O}_3^+)^0}=f_{S_1^0}=(3,8,88,8,32,0),\quad N_{f_{\mathcal{O}_3^0}}=2^{3},\quad d_{f_{\mathcal{O}_3^0}}=2^{6};\\
&f_{\mathcal{O}_1^0}=f_{(\mathcal{O}_1^+)^0}=(3,2,22,2,16,0),\quad N_{f_{\mathcal{O}_1^0}}=2^{2},\quad d_{f_{\mathcal{O}_1^0}}=2^{2}.
\end{aligned}
\end{align*}
Moreover, using the even Clifford construction we can obtain 
\begin{align*}
\begin{aligned}
&f_{\mathcal{O}_5}=(1376,1,3,-3,-128,64),\quad N_{f_{\mathcal{O}_5}}=2^{7},\quad  d_{f_{\mathcal{O}_5}}=2^{5};\\
&f_{\mathcal{O}_3}=(344,1,3,-3,-64,32),\quad N_{f_{\mathcal{O}_3}}=2^{5},\quad  d_{f_{\mathcal{O}_3}}=2^{3};\\
&f_{\mathcal{O}_1}=(86,1,3,-3,-32,16),\quad N_{f_{\mathcal{O}_1}}=2^{3},\quad  d_{f_{\mathcal{O}_1}}=2.
\end{aligned}
\end{align*}

Similarly, we take another quaternion algebra $Q'=\left(\frac{-96,-11}{\mathbb{Q}}\right)$. 
Then $Q'$ ramifies only at prime $2$ too, since the Hilbert symbol $(-96,-11)_p=-1$ for $p=2$, and $=1$ for every other finite prime. This implies that $Q'$ and $Q$ are isomorphic. Then using Hashimoto's construction ~\cite[Proposition 3.1]{Has06} again we may deduce another odd order $\mathcal{O}_5'$ of level $(2^5,1)$ (see~\cite[Section 2]{Has06}) given by 
\begin{equation*}
\mathcal{O}_5'=\mathbb{Z}+\mathbb{Z}\frac{1+j}{2}+\mathbb{Z}\frac{i+ij}{6}+\mathbb{Z}\frac{160j+ij}{11}.
\end{equation*}
Let
\begin{equation*}
\mathcal{O}_3'=\mathbb{Z}+\mathbb{Z}\frac{1+j}{2}+\mathbb{Z}\frac{i+ij}{12}+\mathbb{Z}\frac{160j+ij}{22},
\end{equation*}
\begin{equation*}
\mathcal{O}_1'=\mathbb{Z}+\mathbb{Z}\frac{1+j}{2}+\mathbb{Z}\frac{i+ij}{24}+\mathbb{Z}\frac{160j+ij}{44}.
\end{equation*}
Then $\mathcal{O}_5'\subset\mathcal{O}_3'\subset\mathcal{O}_1'$, and 
\begin{align*}
&\mathcal{O}_5'^+=S_3'=\mathbb{Z}+\mathbb{Z}j+\mathbb{Z}\frac{i+ij}{6}+\mathbb{Z}\frac{160j+ij}{11},\\
&\mathcal{O}_3'^+=S_1'=\mathbb{Z}+\mathbb{Z}j+\mathbb{Z}\frac{i+ij}{12}+\mathbb{Z}\frac{160j+ij}{22},\\
&\mathcal{O}_1'^+=\mathbb{Z}+\mathbb{Z}j+\mathbb{Z}\frac{i+ij}{24}+\mathbb{Z}\frac{160j+ij}{44}.
\end{align*}
The associated ternary quadratic forms are 
\begin{align*}
\begin{aligned}
&f_{\mathcal{O}_5'^0}=f_{(\mathcal{O}_5'^+)^0}=f_{S_3'^0}=(11,32,2336,32,320,0),\quad N_{f_{\mathcal{O}_5'^0}}=2^{5},\quad  d_{f_{\mathcal{O}_5'^0}}=2^{10},\\
&f_{\mathcal{O}_3'^0}=f_{(\mathcal{O}_3'^+)^0}=f_{S_1'^0}=(11,8,584,8,160,0),\quad N_{f_{\mathcal{O}_3'^0}}=2^{3},\quad  d_{f_{\mathcal{O}_3'^0}}=2^{6},\\
&f_{\mathcal{O}_1'^0}=f_{(\mathcal{O}_1'^+)^0}=(11,2,146,2,80,0),\quad N_{f_{\mathcal{O}_1'^0}}=2^{2},\quad  d_{f_{\mathcal{O}_1'^0}}=2^{2};
\end{aligned}
\end{align*}
and
\begin{align*}
\begin{aligned}
&f_{\mathcal{O}_5'}=(9312,3,11,-11,-640,320),\quad N_{f_{\mathcal{O}_5'}}=2^{7},\quad d_{f_{\mathcal{O}_5'}}=2^{5},\\
&f_{\mathcal{O}_3'}=(2328,3,11,-11,-320,160),\quad N_{f_{\mathcal{O}_3'}}=2^{5},\quad d_{f_{\mathcal{O}_3'}}=2^{3},\\
&f_{\mathcal{O}_1'}=(582,3,11,-11,-160,80),\quad N_{f_{\mathcal{O}_1'}}=2^{3},\quad d_{f_{\mathcal{O}_1'}}=2.
\end{aligned}
\end{align*}
Again, using the elementary reduction to reduce the first coefficients $1376$ and $9312$ in $f_{\mathcal{O}_5}$ and $f_{\mathcal{O}_5'}$ above respectively, we find that 
$$
f_{\mathcal{O}_5}\sim(11,1,3,-3,1,0),\,\,\, f_{\mathcal{O}_5'}\sim (3,3,11,-11,-2,1).
$$ 
But, it is straightforward to check $R_{f_{\mathcal{O}_5}}(1)=6$ and $R_{f_{\mathcal{O}_5'}}(1)=2$. Hence
$
f_{\mathcal{O}_5}\nsim f_{\mathcal{O}_5'}; 
$
and thus 
$$
\mathcal{O}_5\ncong\mathcal{O}_5'.
$$ 
Notice that, since both $\mathcal{O}_5$ and $\mathcal{O}_5'$ are of level $(2^5,1)$ (see~\cite[Section 2]{Has06}), so $\mathcal{O}_5$ and $\mathcal{O}_5'$ are locally isomorphic. 

Here we also check that
\begin{equation*}
U_5=\begin{pmatrix}
29&-16&304\\
1&-1&10\\
-2&1&-21
\end{pmatrix},\,\,\ M_{f_{\mathcal{O}_5^0}}=U_5^TM_{f_{\mathcal{O}_5'^0}}U_5,
\end{equation*}
hence $f_{\mathcal{O}_5^0}\sim f_{\mathcal{O}_5'^0}$. Thus by definition, 
$$
\theta_{\mathcal{O}_5^0}(\tau)=\theta_{\mathcal{O}_5'^0}(\tau). 
$$ 
This gives the counterexample as described in Remarks ~\ref{remark5.7} and ~\ref{remark5.8}. 
Finally, one can deduce 
\begin{equation*}
f_{\mathcal{O}_3}\sim f_{\mathcal{O}_3'},\,\,\, U_3=\begin{pmatrix}
-21&-1&2\\
0&-1&0\\
-304&-15&29
\end{pmatrix},
\end{equation*}
hence $\mathcal{O}_3\cong\mathcal{O}_3'$. Similarly, 
\begin{equation*}
f_{\mathcal{O}_1}\sim f_{\mathcal{O}_1'},\,\,\,U_1=\begin{pmatrix}
-1&-1&0\\
-4&-1&1\\
-12&-8&1
\end{pmatrix},
\end{equation*}
and $\mathcal{O}_1\cong\mathcal{O}_1'$.
\end{example}

\begin{example}\label{exam:neven}
There exist chains of explicit odd Gorenstein orders in quaternion algebras of odd discriminant of the following form: 
\begin{equation*}
\mathcal{O}_4\subset\mathcal{O}_2\subset\mathcal{O}_0,
\end{equation*}
and
\begin{equation*}
\mathcal{O}_4'\subset\mathcal{O}_2'\subset\mathcal{O}_0',
\end{equation*}
such that $\mathcal{O}_i$ and $\mathcal{O}_i'$ are not locally isomorphic for $i=2,4$, and
\begin{equation*}
\mathcal{O}_0\cong\mathcal{O}_0',~\mathcal{O}_4^+=S_2,~\mathcal{O}_4'^+=S_2',~\mathcal{O}_2^+=S_0,~\mathcal{O}_2'^+=S_0'.
\end{equation*}

In fact, we can take the quaternion algebra $Q=\left(\frac{-3,-1}{\mathbb{Q}}\right)$ as in Example ~\ref{example1.00}. 
Then $Q$ ramifies only at prime $3$. 
We may construct an odd order $\mathcal{O}_4$ as follows: 
\begin{equation*}
\mathcal{O}_4=\mathbb{Z}+\mathbb{Z}\frac{1+i}{2}+\mathbb{Z}4j+\mathbb{Z}(2j+2k).
\end{equation*}
Let
\begin{equation*}
\mathcal{O}_2=\mathbb{Z}+\mathbb{Z}\frac{1+i}{2}+\mathbb{Z}2j+\mathbb{Z}(j+k),
\end{equation*}
\begin{equation*}
\mathcal{O}_0=\mathbb{Z}+\mathbb{Z}\frac{1+i}{2}+\mathbb{Z}j+\mathbb{Z}\frac{j+k}{2}.
\end{equation*}
Then we have
\begin{align*}
&\mathcal{O}_4^+=S_2=\mathbb{Z}+\mathbb{Z}i+\mathbb{Z}4j+\mathbb{Z}(2j+2k),\\
&\mathcal{O}_2^+=S_0=\mathbb{Z}+\mathbb{Z}i+\mathbb{Z}2j+\mathbb{Z}(j+k).
\end{align*}
The associated ternary quadratic forms are
\begin{align*}
\begin{aligned}
&f_{\mathcal{O}_4^0}=f_{(\mathcal{O}_4^+)^0}=f_{S_2^0}=(3,16,16,16,0,0),\quad N_{f_{\mathcal{O}_4^0}}=3\cdot2^{4},\quad  d_{f_{\mathcal{O}_4^0}}=3^2\cdot2^{8},\\
&f_{\mathcal{O}_2^0}=f_{(\mathcal{O}_2^+)^0}=f_{S_0^0}=(3,4,4,4,0,0),\quad N_{f_{\mathcal{O}_2^0}}=3\cdot2^{2},\quad d_{f_{\mathcal{O}_2^0}}=3^2\cdot2^{4},
\end{aligned}
\end{align*}
and
\begin{align*}
\begin{aligned}
&f_{\mathcal{O}_4}=(16,1,1,-1,0,0),\quad N_{f_{\mathcal{O}_4}}=3\cdot2^{6},\quad d_{f_{\mathcal{O}_4}}=3\cdot2^{4},\\
&f_{\mathcal{O}_2}=(4,1,1,-1,0,0),\quad N_{f_{\mathcal{O}_2}}=3\cdot2^{4},\quad d_{f_{\mathcal{O}_2}}=3\cdot2^{2},\\
&f_{\mathcal{O}_0}=(1,1,1,-1,0,0),\quad N_{f_{\mathcal{O}_0}}=3\cdot2^{2},\quad d_{f_{\mathcal{O}_0}}=3.
\end{aligned}
\end{align*}

Similarly, we take a $Q'=\left(\frac{-48,-7}{\mathbb{Q}}\right)$, which ramifies only at prime $3$ since the Hilbert symbol $(-48,-7)_p=-1$ for $p=3$, and $=1$ for every other finite prime.
Take
\begin{equation*}
\mathcal{O}_4'=\mathbb{Z}+\mathbb{Z}\frac{1+j}{2}+\mathbb{Z}\frac{i+ij}{2}+\mathbb{Z}\frac{48j+ij}{7},
\end{equation*}
which is of level $(3,16)$ by Hashimoto's construction ~\cite[Proposition 3.1]{Has06}. Let
\begin{equation*}
\mathcal{O}_2'=\mathbb{Z}+\mathbb{Z}\frac{1+j}{2}+\mathbb{Z}\frac{i+ij}{4}+\mathbb{Z}\frac{48j+ij}{14},
\end{equation*}
\begin{equation*}
\mathcal{O}_0'=\mathbb{Z}+\mathbb{Z}\frac{1+j}{2}+\mathbb{Z}\frac{i+ij}{8}+\mathbb{Z}\frac{48j+ij}{28}.
\end{equation*}
Then we have
\begin{equation*}
\mathcal{O}_4'^+=S_2'=\mathbb{Z}+\mathbb{Z}j+\mathbb{Z}\frac{i+ij}{2}+\mathbb{Z}\frac{48j+ij}{7},
\end{equation*}
\begin{equation*}
\mathcal{O}_2'^+=S_0'=\mathbb{Z}+\mathbb{Z}j+\mathbb{Z}\frac{i+ij}{4}+\mathbb{Z}\frac{48j+ij}{14}.
\end{equation*}
The associated ternary quadratic forms are
\begin{equation*}
f_{\mathcal{O}_4'^0}=f_{(\mathcal{O}_4'^+)^0}=f_{S_2'^0}=(7,96,336,48,96,0),\,\,\,N_{f_{\mathcal{O}_4'^0}}=3\cdot2^{4}, \,\,\, d_{f_{\mathcal{O}_4'^0}}=3^2\cdot2^{8},
\end{equation*}
\begin{equation*}
f_{\mathcal{O}_2'^0}=f_{(\mathcal{O}_2'^+)^0}=f_{S_0'^0}=(7,24,84,12,48,0),\,\,\,N_{f_{\mathcal{O}_2'^0}}=3\cdot2^{2}, \,\,\, d_{f_{\mathcal{O}_2'^0}}=3^2\cdot2^{4},
\end{equation*}
and
\begin{align*}
&f_{\mathcal{O}_4'}=(2640,1,14,-7,-384,96),\,\,\,N_{f_{\mathcal{O}_4'}}=3\cdot2^{6},\,\,\, d_{f_{\mathcal{O}_4'}}=3\cdot2^{4},\\
&f_{\mathcal{O}_2'}=(660,1,14,-7,-192,48),\,\,\,N_{f_{\mathcal{O}_2'}}=3\cdot2^{4},\,\,\, d_{f_{\mathcal{O}_2'}}=3\cdot2^{2},\\
&f_{\mathcal{O}_0'}=(165,1,14,-7,-96,24),\,\,\,N_{f_{\mathcal{O}_0'}}=3\cdot2^{2},\,\,\, d_{f_{\mathcal{O}_0'}}=3.
\end{align*}
Since $\mathcal{O}_4\in\operatorname{Typ}_1(16)$ and $\mathcal{O}_4'\in\operatorname{Typ}_2(16)$,
and $\mathcal{O}_2\in\operatorname{Typ}_1(4)$ and $\mathcal{O}_2'\in\operatorname{Typ}_2(4)$, hence   
$$
\mathcal{O}_4\ncong\mathcal{O}_4',\,\,\,\mathcal{O}_2\ncong\mathcal{O}_2'.
$$
Here we also check that
\begin{equation*}
U_2=\begin{pmatrix}
3&-14&-4\\
0&-1&0\\
-1&4&1
\end{pmatrix},\,\,\ M_{f_{\mathcal{O}^0_2}}=U_2^TM_{f_{(\mathcal{O}_2')^0}}U_2,
\end{equation*}
hence $f_{\mathcal{O}_2^0}\sim f_{\mathcal{O}_2'^0}$. Thus by definition, 
$$
\theta_{\mathcal{O}_2^0}(\tau)=\theta_{\mathcal{O}_2'^0}(\tau). 
$$ 
Similarly, 
\begin{equation*}
f_{\mathcal{O}_0}\sim f_{\mathcal{O}_0'},\,\,\,U_0=\begin{pmatrix}
-1&-1&1\\
-2&1&-2\\
-4&-3&3
\end{pmatrix},
\end{equation*}
and $\mathcal{O}_0\cong\mathcal{O}_0'$.
In fact, $\mathcal{O}_0$ and $\mathcal{O}_0'$ are  maximal orders respectively in $Q$ and $Q'$.
This gives another counterexample as described in Remarks ~\ref{remark5.7} and ~\ref{remark5.8}. 
 
\end{example}

\section{Applications: Explicit orders associated to congruent numbers}\label{section07}

In this section we give an explicit formulation for the orders and their relationships associated to the quadratic forms of Tunnell~\cite{Tun83} and Qin~\cite{Qin22} concerning congruent numbers, hence provide a new perspective to understand the relevant important forms. 

\subsection{An example}
As an illustration of the procedures of computations, we first give an explicit example to present Theorem ~\ref{the:bijection}, for which we have $m_{f_\mathcal{O}}>1$.

\begin{example}\label{example6.1}
Take $Q=\left(\frac{-3,-1}{\mathbb{Q}}\right)$, and let 
\begin{align*}
&\mathcal{O}=\mathbb{Z}+\mathbb{Z}\frac{1+i}{2}+\mathbb{Z}3j+\mathbb{Z}\frac{3j+3k}{2},\\
&\mathcal{O}'=\mathbb{Z}+\mathbb{Z}i+\mathbb{Z}3j+\mathbb{Z}3k.
\end{align*}
Using the explicit basis over $\mathbb{Z}$, direct computations show that $\mathcal{O}$ is an odd order of discriminant $27$, and $\mathcal{O}'$ is an even order of discriminant $108$. Hence by definition we have 
\begin{equation*}
S=\mathbb{Z}+2\mathcal{O}=\mathbb{Z}+\mathbb{Z}i+\mathbb{Z}6j+\mathbb{Z}(3j+3k),
\end{equation*}
and direct computation by ~\eqref{gongshi3.7} and ~\eqref{gongshi3.8} yields 
\begin{align*}
&f_{S^0}=3x^2+36y^2+36z^2+36yz=3(x^2+12y^2+12z^2+12yz),\\
&f_{\mathcal{O}'^0}=3x^2+9y^2+27z^2=3(x^2+3y^2+9z^2),
\end{align*}
and hence $\widetilde m_{f_{S^0}}=\widetilde m_{f_{\mathcal{O}'^0}}=3$. Thus by Theorem ~\ref{the:bijection} we have
\begin{align*}
&\psi_{f_{\mathcal{O}}}=x^2+12y^2+12z^2+12yz,\\
&\psi_{f_{\mathcal{O}'}}=x^2+3y^2+9z^2,
\end{align*}
and
$$
m_{f_{\mathcal{O}}}=3,\,\,\,m_{f_{\mathcal{O}'}}=4\cdot3=12.
$$ 
And so by the definition of the reciprocal of forms we get 
\begin{equation*}
f_{\mathcal{O}}=9x^2+y^2+z^2+yz,\,\,\,f_{\mathcal{O}'}=9x^2+3y^2+z^2. 
\end{equation*}
Also, by Theorem~\ref{the:bijection} we have 
\begin{align*}
&\rho_{\mathcal{O}}(n,r)=R_{(1,12,12,12,0,0)}\left(\frac{4n-r^2}{3}\right),\\
&\rho_{\mathcal{O}'}(n,r)=R_{(1,3,9,0,0,0)}\left(\frac{4n-r^2}{12}\right),
\end{align*}
where and in what follows, $R_f(n)=R_{(a,b,c,r,s,t)}(n)$ denotes the number of representations of integer $n$ by the form $f=(a,b,c,r,s,t)$. 
\end{example}

Notice that, as for $f_{\mathcal{O}}$ and $f_{\mathcal{O}'}$, one can certainly use the even Clifford algebra method to compute directly but tediously, giving $f_{\mathcal{O}}=9x^2+y^2+z^2+yz$ and $f_{\mathcal{O}'}=9x^2+3y^2+z^2$ again, and then giving the arithmetic quantities: 
$$
N_{f_{\mathcal{O}}}=36,\,\, d_{f_{\mathcal{O}}}=27,\,\, m_{f_{\mathcal{O}}}=3;\,\,N_{f_{\mathcal{O}'}}=36,\,\, d_{f_{\mathcal{O}'}}=108,\,\, m_{f_{\mathcal{O}'}}=12.
$$ 
It is clear that the two methods are compatible. 

\vskip 0.2cm
\subsection{Explicit orders associated to congruent numbers}
Now turn to the main focus of this section. Recall that the classical notable theorem of Tunnell~\cite{Tun83} asserts the conditional equivalence under the Birch and Swinnerton-Dyer (BSD) conjecture between the congruent number problem of $n$ and identities concerning the number of representations of $n$ or $n/2$ by the ternary quadratic forms: 
\begin{equation}\label{gongshi6.1} 
x^2+2y^2+8z^2,~x^2+2y^2+32z^2;~x^2+4y^2+8z^2,~x^2+4y^2+32z^2.
\end{equation}
In spite of the simple diagonal form of the forms, the first and the second two forms do not belong to the same genus. In view of this, Qin~\cite{Qin22} optimized this by giving the following two new forms instead of the second and the fourth of ~\eqref{gongshi6.1}, among many others:
\begin{equation}\label{gongshi6.2}
2x^2+4y^2+9z^2-4yz,~4x^2+4y^2+9z^2-4yz.
\end{equation}
Notice that, under this modification, the first ones in ~\eqref{gongshi6.1} and ~\eqref{gongshi6.2} are in the same genus, and the third one in  ~\eqref{gongshi6.1} and the second one in  ~\eqref{gongshi6.2} are also in the same genus.

\vskip0.2cm
Using our main result we can give a quaternion order characterization instead of the quadratic form version for the results of Tunnell~\cite{Tun83} and Qin~\cite{Qin22} as follows.  
\begin{prop}\label{exam:congruent}
Let $Q=\left(\frac{-2,-1}{\mathbb{Q}}\right)$ and 
\begin{equation}
\mathcal{O}=\mathbb{Z}+\mathbb{Z}2i+\mathbb{Z}2j+\mathbb{Z}4k, 
\end{equation}
where $i^2=-2$, $j^2=-1$, $ij=k$. Take 
\begin{equation}
\psi(\mathcal{O})=\mathbb{Z}+\mathbb{Z}i+\mathbb{Z}4j+\mathbb{Z}2k.
\end{equation}
Then 
\begin{equation*}
\widetilde{f}_{\mathcal{O}}=f_{\psi(\mathcal{O})}=(4x^2+8y^2+32z^2)/4=x^2+2y^2+8z^2,
\end{equation*}
\begin{equation*}
\widetilde{f}_{\psi(\mathcal{O})}=f_{\mathcal{O}}=(2x^2+8y^2+16z^2)/2=x^2+4y^2+8z^2.
\end{equation*}
This shows in addition $\psi(\mathcal{O})$ can serve as the reciprocal, up to isomorphism, of $\mathcal{O}$ in Corollary {\rm ~\ref{cor:automorphism}}. 

Next, let
\begin{equation}
\mathcal{O}_1=\mathbb{Z}+\mathbb{Z}4i+\mathbb{Z}4j+\mathbb{Z}16k,\,\,\,
\mathcal{O}_2=\mathbb{Z}+\mathbb{Z}2i+\mathbb{Z}16j+\mathbb{Z}4k.
\end{equation}
Then 
\begin{align*}
&\widetilde{f}_{\mathcal{O}_1}=(16x^2+32y^2+512z^2)/16=x^2+2y^2+32z^2,\\
&\widetilde{f}_{\mathcal{O}_2}=(8x^2+32y^2+256z^2)/8=x^2+4y^2+32z^2.
\end{align*}

Further, let
\begin{align}
\mathcal{O}_1'=\mathbb{Z}+\mathbb{Z}4i+\mathbb{Z}8j+\mathbb{Z}(8k-4j),\,\,\,
\mathcal{O}_2'=\mathbb{Z}+\mathbb{Z}4i+\mathbb{Z}(8j-2k)+\mathbb{Z}4k.
\end{align}
Then 
\begin{align*}
&\widetilde{f}_{\mathcal{O}_1'}=(32x^2+64y^2+144z^2-64yz)/16=2x^2+4y^2+9z^2-4yz,\\
&\widetilde{f}_{\mathcal{O}_2'}=(32x^2+32y^2+72z^2-32yz)/8=4x^2+4y^2+9z^2-4yz.
\end{align*}

Moreover, by Theorem {\rm ~\ref{the:bijection}} we have
$$
\rho_{\mathcal{O}}(n,r)=R_{(1,2,8,0,0,0)}\left(\frac{4n-r^2}{16}\right),\quad
\rho_{\psi(\mathcal{O})}(n,r)=R_{(1,4,8,0,0,0)}\left(\frac{4n-r^2}{8}\right),
$$
$$
\rho_{\mathcal{O}_1}(n,r)=R_{(1,2,32,0,0,0)}\left(\frac{4n-r^2}{64}\right),\quad
\rho_{\mathcal{O}_1'}(n,r)=R_{(2,4,9,-4,0,0)}\left(\frac{4n-r^2}{64}\right),
$$
$$
\rho_{\mathcal{O}_2}(n,r)=R_{(1,4,32,0,0,0)}\left(\frac{4n-r^2}{32}\right),\quad \rho_{\mathcal{O}_2'}(n,r)=R_{(4,4,9,-4,0,0)}\left(\frac{4n-r^2}{32}\right).
$$
\end{prop}

\begin{remark}
Here we point out that, for any form $f$ in~\eqref{gongshi6.1} and ~\eqref{gongshi6.2}, it is impossible to find an order $\mathcal{O}$ such that $f_{\mathcal{O}^0}=f$ on noting Corollary ~\ref{cor:tqf}. But with a modification by multiplying a suitable positive integer, this can be done. Indeed, the integers $2$, $4$, $8$ and $16$ arising in Proposition ~\ref{exam:congruent} above are the minimal multiples such that the relevant quadratic forms $mf$ can be written as $mf=f_{\mathcal{O}^0}$ with an order $\mathcal{O}$.  
\end{remark}

As an immediate consequence of Proposition ~\ref{exam:congruent}, together with \cite{Tun83} and~\cite[Theorem 1.3]{Qin22}, we have the following proposition. 
\begin{prop}
Retain the notations of Proposition ~{\rm \ref{exam:congruent}}. Let $N$ be a positive
squarefree congruent number, let $n,n_1,n_1'$ be positive integers,
and let $r,r_1,r_1'\in\mathbb Z$.
\begin{enumerate}
\item If $N$ is odd and
\[
\frac{4n_1-r_1^2}{64}
=\frac{4n-r^2}{16}
=\frac{4n_1'-(r_1')^2}{64}
=N,
\]
then
\[
\rho_{\mathcal O_1}(n_1,r_1)
=\frac12\rho_{\mathcal O}(n,r)
=\rho_{\mathcal O_1'}(n_1',r_1').
\]

\item If $N$ is even and
\[
\frac{4n_1-r_1^2}{32}
=\frac{4n-r^2}{8}
=\frac{4n_1'-(r_1')^2}{32}
=\frac N2,
\]
then
\[
\rho_{\mathcal O_2}(n_1,r_1)
=\frac12\rho_{\psi(\mathcal O)}(n,r)
=\rho_{\mathcal O_2'}(n_1',r_1').
\]
\end{enumerate}
\end{prop}

\subsection{Further relations}
The reciprocal operation gives further relations between the orders
occurring in the odd and even congruent-number criteria. We first fix
explicit representatives of their reciprocal isomorphism classes;
the inclusions below refer to these representatives inside
$Q=\left(\frac{-2,-1}{\mathbb Q}\right)$.

\begin{prop}\label{prop:cn-reciprocal-overorders}
With the notations of Proposition ~{\rm\ref{exam:congruent}}, put
\begin{equation}\label{eq:cn-reciprocal-orders}
\begin{aligned}
\psi(\mathcal O_1)
 &=\mathbb Z+\mathbb Zi+\mathbb Z8j+\mathbb Z4k,\\
\psi(\mathcal O_1')
 &=\mathbb Z+\mathbb Z4i+\mathbb Z(4j+k)+\mathbb Z2k,\\
\psi(\mathcal O_2)
 &=\mathbb Z+\mathbb Z4i+\mathbb Z2j+\mathbb Z8k,\\
\psi(\mathcal O_2')
 &=\mathbb Z+\mathbb Z4i+\mathbb Z(2j+4k)+\mathbb Z4j.
\end{aligned}
\end{equation}
Then we have 
\begin{equation}\label{eq:cn-cross-inclusions}
\mathcal O_1\subset\psi(\mathcal O_2),\quad
\mathcal O_1'\subset\psi(\mathcal O_2'),\quad
\mathcal O_2\subset\psi(\mathcal O_1),\quad
\mathcal O_2'\subset\psi(\mathcal O_1'),
\end{equation}
and each inclusion has index $4$.
\end{prop}

\begin{proof}
The relations $i^2=-2$, $j^2=-1$, and $ij=-ji=k$ show that
the four lattices in \eqref{eq:cn-reciprocal-orders} are closed
under multiplication. Thus they are orders, and they are even
because their displayed nonscalar basis elements have trace zero.

For such a basis, we have
\[
\begin{array}{c|c|l}
Orders & discriminant &f_{\mathcal O}\\ \hline
\psi(\mathcal O_1)&2^8&32x^2+y^2+2z^2\\
\psi(\mathcal O_1')&2^8&2x^2+4y^2+9z^2-4yz\\
\psi(\mathcal O_2)&2^9&4x^2+32y^2+z^2\\
\psi(\mathcal O_2')&2^9&4x^2+4y^2+9z^2-4yz.
\end{array}
\]
After permuting variables in the first and third rows, these are
precisely $\widetilde f_{\mathcal O_1}$,
$\widetilde f_{\mathcal O_1'}$,
$\widetilde f_{\mathcal O_2}$, and
$\widetilde f_{\mathcal O_2'}$. The bijection and reciprocal
identity established above therefore identify the four reciprocal
isomorphism classes as asserted.

The inclusions in \eqref{eq:cn-cross-inclusions} follow by
expressing the bases of the smaller orders in the displayed
bases of the larger ones. The absolute values of the resulting
determinants are all $4$. We shall also use the inclusions
\begin{equation}\label{eq:cn-doubled-lattices}
\begin{aligned}
2\psi(\mathcal O_2)^0&\subset\mathcal O_1^0,&
2\psi(\mathcal O_2')^0&\subset(\mathcal O_1')^0,\\
2\psi(\mathcal O_1)^0&\subset\mathcal O_2^0,&
2\psi(\mathcal O_1')^0&\subset(\mathcal O_2')^0,
\end{aligned}
\end{equation}
each of which has index $2$.
\end{proof}

The inclusions in Proposition~\ref{prop:cn-reciprocal-overorders} give an order-theoretic
interpretation of elementary parity relations between the
associated ternary representation numbers.

\begin{cor}\label{prop:cn-norm-transfer}
Let $n$ be a positive integer with $n\equiv1\pmod8$, then division by $2$ induces bijections
\begin{equation}\label{eq:cn-norm-bijections}
\begin{aligned}
\{\alpha\in\mathcal O_1^0:\operatorname n(\alpha)=16n\}
&\longrightarrow
\{\beta\in\psi(\mathcal O_2)^0:\operatorname n(\beta)=4n\},\\
\{\alpha\in\mathcal O_2^0:\operatorname n(\alpha)=8n\}
&\longrightarrow
\{\beta\in\psi(\mathcal O_1)^0:\operatorname n(\beta)=2n\}.
\end{aligned}
\end{equation}
In particular,
\begin{equation}\label{eq:cn-norm-counts}
\rho_{\mathcal O_1}(16n,0)=\rho_{\psi(\mathcal O_2)}(4n,0),
\qquad
\rho_{\mathcal O_2}(8n,0)=\rho_{\psi(\mathcal O_1)}(2n,0).
\end{equation}
More generally, for every $r\in\mathbb Z$,
\begin{equation}\label{eq:cn-trace-counts}
\begin{aligned}
\rho_{\mathcal O_1}(r^2+16n,2r)
 &=\rho_{\psi(\mathcal O_2)}(r^2+4n,2r),\\
\rho_{\mathcal O_2}(r^2+8n,2r)
 &=\rho_{\psi(\mathcal O_1)}(r^2+2n,2r).
\end{aligned}
\end{equation}
The statements also hold for $\mathcal O_1'$, $\mathcal O_2'$, $\psi(\mathcal O_1')$ and $\psi(\mathcal O_2')$.
\end{cor}
\begin{proof}
We verify the integrality of division by $2$ in all four cases.
Writing a trace-zero element in the bases of Proposition~7.1,
the relevant norm equations are
\[
\begin{array}{c|l|l}
\text{orders}&\alpha&\text{norm equation}\\ \hline
\mathcal O_1
 &4xi+4yj+16zk&2x^2+y^2+32z^2=n\\
\mathcal O_1'
 &4xi+8yj+z(8k-4j)&2x^2+(2y-z)^2+8z^2=n\\
\mathcal O_2
 &2xi+16yj+4zk&x^2+32y^2+4z^2=n\\
\mathcal O_2'
 &4xi+y(8j-2k)+4zk&4x^2+(2z-y)^2+8y^2=n.
\end{array}
\]
Here $x,y,z\in\mathbb Z$. In the first two rows, reduction
modulo $8$ implies that $x$ is even. In the third row it implies
that $z$ is even, and in the fourth row it implies that $x$ is
even. Using \eqref{eq:cn-reciprocal-orders}, these are exactly
the conditions ensuring, respectively, that
\[
\alpha/2\in\psi(\mathcal O_2)^0,\quad
\alpha/2\in\psi(\mathcal O_2')^0,\quad
\alpha/2\in\psi(\mathcal O_1)^0,\quad
\alpha/2\in\psi(\mathcal O_1')^0.
\]
Conversely, \eqref{eq:cn-doubled-lattices} shows that
multiplication by $2$ maps each target in
\eqref{eq:cn-norm-bijections} into its source. Since
$\operatorname n(2\beta)=4\operatorname n(\beta)$, these maps
are inverse bijections, proving \eqref{eq:cn-norm-counts}.

For any of the orders $\mathcal O$ under consideration, an element
of trace $2r$ has a unique expression $r+\alpha$ with
$\alpha\in\mathcal O^0$. Since
\[
\operatorname n(r+\alpha)=r^2+\operatorname n(\alpha),
\]
the bijections in \eqref{eq:cn-norm-bijections} give the
corresponding bijections at trace $2r$ by
$r+\alpha\mapsto r+\alpha/2$. This proves
\eqref{eq:cn-trace-counts}.
\end{proof}

\bibliographystyle{plain}
\bibliography{TQF}

\end{document}